\documentclass[10pt]{amsart}
\usepackage{graphicx}
 \usepackage{caption}
\usepackage{psfrag}
\usepackage{float}
\usepackage{subfigure}

\usepackage{amsmath}

 \usepackage{color}

\usepackage{amsthm,amsfonts,amsxtra,amssymb,amscd}
 \RequirePackage[T1]{fontenc}
 \usepackage{lmodern}

\usepackage[all]{xy}
\xyoption{2cell}
\usepackage{enumerate}

\usepackage{pdfsync}

  \usepackage[unicode]{hyperref}

\theoremstyle{plain}\definecolor{Red}{cmyk}{0,1,1,0} \definecolor{Green}{cmyk}{1,0,1,0}
\theoremstyle{plain}
\newtheorem*{lemma*}{Lemma}
\newtheorem{lemma}{Lemma}
\newtheorem*{theorem*}{Theorem}
\newtheorem{theorem}{Theorem}
\newtheorem*{proposition*}{Proposition}
\newtheorem{proposition}{Proposition}
\newtheorem*{corollary*}{Corollary}

\newtheorem{corollary}{Corollary}
\theoremstyle{definition}
\newtheorem*{definition*}{Definition}
\newtheorem{constraint}{Constraint}
\newtheorem{definition}{Definition}
\newtheorem*{example*}{Example}
\newtheorem{example}{Example}
\theoremstyle{remark}
\newtheorem*{remark*}{Remark}\newtheorem*{proof*}{Proof}
\newtheorem{remark}{Remark}

\def\Gama{{\Gamma_S}}
\def\Gamat{  \R^m_{X,S}}

\def\Z{{\mathbb Z}}

\def\R{{\mathbb R}}
\def\C{{\mathbb C}}

\def\GA{{\mathcal A}}

 \title[ The rectangle graphs ]{ The rectangle graphs}
\author{   C. Procesi{**}. }
 
\begin{document}
\begin{abstract}
We discuss a  combinatorial graph used in the study of the NLS.
 \end{abstract}\maketitle
 
\tableofcontents
\section{Introduction} In this paper  we   want to present in a unified form the results on a graph used in the papers  \cite{PP},\cite{PP1},\cite{PP3}, for the study of the cubic NLS.   We will not recall the origin of this graph  which can be found in the mentioned papers, nor its applications, but only the theory which appears scattered in the previous papers (with some unfortunate mistakes or obscure proofs), trying to give a more readable and unified treatment of the main Theorems. \bigskip

The rectangle graphs are infinite graphs which appear for any given integer $n$, in two versions an {\em arithmetic} and a {\em geometric } form.  In the first  the vertices are the points in $\Z^n$  while in the second  the points in $\R^n$ for some given dimension $n$.  

The construction of one of these graphs, that is the description of the edges,  depends on the choice of a set of  vectors $S:=\{\mathtt v_1,\cdots,\mathtt v_m\}$ (called {\em tangential sites} for dynamical reasons) in $\Z^n$ in the arithmetic case and in $\R^n$ in the geometric case. 

 We thus will have a family of graphs  depending on $S$, the corresponding graph  will be denoted $\Gamma_S$.  A more general set of graphs appears for the NLS  with non linear part of degree $2q+1$  for $q>1$. For these the results of Part 1 of this paper still hold, as shown in the Appendix and are sufficient for most applications, see \cite{PP3}.\medskip
\part{The graphs $\Gamma_S$}
\section{Edges and rectangles}
Given $S=\{\mathtt v_1,\cdots,\mathtt v_m\}\subset\R^n$, the graph  $\Gamma_S$ can be first defined as a geometric graph with vertices in $\R^n$ and, in case the $\mathtt v_i\in\Z^n$,  its restriction to $\Z^n$ is the arithmetic graph.  It is defined   taking the following edges.
\begin{definition}\label{pr}
Two  points  $p,q\in R^n$ are  connected  with an edge in $\Gama$, if there exist  two vectors $\mathtt v_i,\mathtt v_j\in S$ so that the vectors $p,q,\mathtt v_i,\mathtt v_j$ are the vertices of a rectangle.
\end{definition} 
Notice that  
 the  vectors $a, b, c,d$ are the vertices of a rectangle  if and only if   $$ a+c= b+d,\quad |a|^2+|c|^2= |b|^2+|d|^2 .$$   
 \begin{center}\includegraphics[width=3cm]{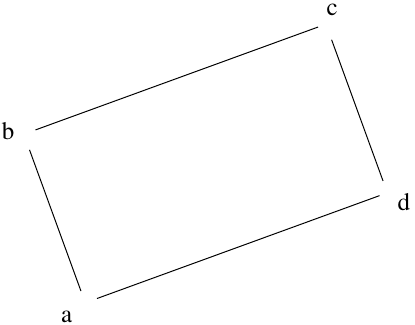}\end{center}
\begin{remark}\label{edg}\strut
In fact we have two different possibilities (two colors)

\begin{itemize}\item An oriented  black edge $\xymatrix{  p\ar@{->}[r]^{\mathtt v_i-\mathtt v_j} &q}$ connects  two points $p , \ q $ which are  {\em adjacent} in the rectangle with vertices     $p, q,\mathtt v_i,\mathtt v_j$  hence $$
q=p+\mathtt v_i-\mathtt v_j,\   |p|^2+|\mathtt v_i|^2=|q|^2+|\mathtt v_j|^2\implies |p|^2+|\mathtt v_i|^2=|p+\mathtt v_i-\mathtt v_j |^2+|\mathtt v_j|^2
$$$$  |p|^2+|\mathtt v_i|^2=|p|^2+2(p,\mathtt v_i-\mathtt v_j )+|\mathtt v_i-\mathtt v_j |^2+|\mathtt v_j|^2$$ 
\begin{equation}\label{lan}\implies  \boxed{(p,\mathtt v_i-\mathtt v_j )= (\mathtt v_i,\mathtt v_j)-|\mathtt v_j |^2}.\end{equation}

\item A red edge $\xymatrix{  p\ar@{=}[r]^{-\mathtt v_i-\mathtt v_j} &q}$  connects  two points $p, q  $ which are {\em opposite} in the rectangle with vertices     $p, \mathtt v_j,q,\mathtt v_i$   hence 
$$
q=-p+\mathtt v_i+\mathtt v_j,\   |p|^2+|q|^2=|\mathtt v_i|^2 +|\mathtt v_j|^2\implies |p|^2+|-p+\mathtt v_i+\mathtt v_j|^2=|\mathtt v_i|^2 +|\mathtt v_j|^2
$$
 $$  2|p|^2+| \mathtt v_i+\mathtt v_j|^2-2(p,\mathtt v_i+\mathtt v_j)=|\mathtt v_i|^2 +|\mathtt v_j|^2$$
 \begin{equation}\label{lar}\implies  \boxed{ |p|^2- (p,\mathtt v_i+\mathtt v_j)=- (\mathtt v_i,\mathtt v_j)}.\end{equation}

\end{itemize}

\end{remark}  
 \begin{definition}\label{lasfe}  1)\quad An edge $\ell= -\mathtt v_i-\mathtt v_j $ defines a sphere $S_\ell$ through the relation:

\begin{equation}\label{sfera}S_\ell=\{x\mid |x|^2+(x,-\mathtt v_i-\mathtt v_j)=- (  \mathtt v_i,\mathtt v_j) \iff |x-\frac{\mathtt v_i+\mathtt v_j}2|^2=\frac{|\mathtt v_i-\mathtt v_j|^2}4\}.
\end{equation} The sphere $S_\ell$ is the one in which  two vectors $\mathtt v_i,\mathtt v_j$ are the endpoints of a diameter, that is of center $\frac{\mathtt v_i+\mathtt v_j}2$ and diameter $|\mathtt v_i-\mathtt v_j|$. \smallskip

 Two points $p,q$ are joined by the red edge  $\ell= -\mathtt v_i-\mathtt v_j $ if and only if they are endpoints of a diameter of $S_\ell$.\medskip

2)\quad An edge $\ell=\mathtt v_i-\mathtt v_j$ defines a hyperplane $H_{\ell}$ through the relation \begin{equation}
\label{iperp} H_\ell=\{x\mid (x,\mathtt v_i-\mathtt v_j)=   |\mathtt v_i|^2- (  \mathtt v_i,\mathtt v_j)=(\mathtt v_i,\mathtt v_i -\mathtt v_j)
\}.
\end{equation}
 \end{definition}
 The hyperplane  $H_{\ell}$ is the one passing through $\mathtt v_i$ and perpendicular to $\mathtt v_i-\mathtt v_j$,  $H_{-\ell}$ is the one passing through $\mathtt v_j$ and perpendicular to $\mathtt v_i-\mathtt v_j$ that is parallel to $H_{\ell}$.

 Two points $p,q$ are joined by the black  edge  $\ell=  \mathtt v_i-  \mathtt v_j $ if and only if $p\in H_\ell$ and $q$ is the orthogonal projection of $p$ to $H_{-\ell}$.\medskip
 
 \includegraphics[width=11cm]{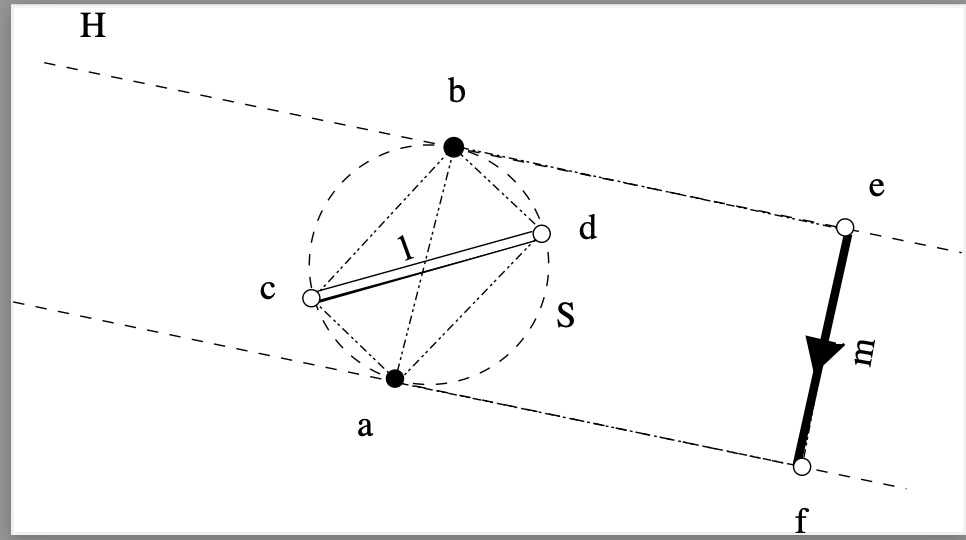}
\medskip

 {\footnotesize{The plane $H_{\ell}$ with $\ell=  \mathtt v_j-\mathtt v_i$ and the sphere $S_{\ell}$  with $\ell= -\mathtt v_i-\mathtt v_j$. The points $\mathtt v_i=a,\mathtt v_j=b,e,f$ form  the vertices of a rectangle. Same for  the points $a,c,b,d$}}\label{fig1}

  $$\xymatrix{ &\mathtt v_i\ar@{->}[dr]&&\\
 p\ar@{-}[ur] \ar@{->}[dr]^{\mathtt v_i-\mathtt v_j}&&\mathtt v_j \\
 & q\ar@{-}[ur]&&  } \qquad\xymatrix{ &\mathtt v_i\ar@{-}[dr]&&\\
 p\ar@{-}[ur] \ar@{-}[dr] \ar@{=}[rr]^{\mathtt v_i-\mathtt v_j}&&q \\
 & \mathtt v_j\ar@{-}[ur]&&  } $$
 Thus the set $S$  determines finitely many spheres and finitely many pairs of parallel hyperplanes, which have a complicated geometric pattern of intersections. 
 
 Points which are not in any of these finitely many     spheres or hyperplanes       will not be connected to any other point in the graph, that is they are isolated.  The possibility for a point to be connected with many other points depends roughly   in how many of these hypersurfaces the point lies.  It should be intuitively clear that the complicated geometry of this configuration of      spheres and  hyperplanes depends strongly on the choice of $S$.\medskip

 EXAMPLE:  $S$  is given by 4  points in the plane marked \quad $\textcolor{Red}{\bullet}$  $$\xymatrix{&.&.&.&.&.&.&.&.&.\\
 &.&.&\mathtt v_1\textcolor{Red}{\bullet}&.&.&.&.&.&.\\
 &.&.&.&.&.&.&.&\mathtt v_2\textcolor{Red}{\bullet}&.\\
 &.&.&.&.&.&.&.&.&.\\
 &.&.&.&.&.&.&\mathtt v_3\textcolor{Red}{\bullet}&.&.\\
 & .&.&\mathtt v_4\textcolor{Red}{\bullet}&.&.&.&.&.&.
 } $$ EXAMPLE:  points connected by edges  $$\xymatrix{ &.&.&.&\textcolor{Green}{\bullet}\ar@{-}[ddl]\ar@{.}[rrd]&.&.&.&.&.\\
 &.&.&\textcolor{Red}{\bullet}\ar@{.}[lldd]&.&.&\textcolor{Green}{\bullet}\ar@{-}[ddl]\ar@{-}[dddd]\ar@{.}[rrd]&.&.&.\\
 &.&.&\textcolor{Green}{\bullet}\ar@{.}[rrd]&.&.&.&.&\textcolor{Red}{\bullet}&.\\
 &\textcolor{Green}{\bullet} {\ar@{=}[rrrr]}\ar@{.}[rrdd]&.&.&.&\textcolor{Green}{\bullet}\ar@{.}[lldd]\ar@{.}[rrd]\ar@{.}[lluu]&.&.&.&.\\
 &.&.&.&.&.&.&\textcolor{Red}{\bullet}\ar@{.}[ruu]&.&.\\
 & .&.&\textcolor{Red}{\bullet}&.&.&\textcolor{Green}{\bullet}&.&.&.
} $$
The graph depends strongly  on the choice of $S$  and we want to see its form under a {\em generic choice} of $S$.  Recall some terminology
\begin{definition}\label{term}
A path in a 
 graph $\Gamma$  is a sequence of vertices $v_1,v_2,\cdots, v_k$  such  that $v_i$ and $v_{i+1}$ are connected by an edge.
 
 A path is {\em simple}  if the $v_i$ are all distinct.
 
 A path is a {\em circuit} if $v_1=v_k$.  It is a {\em simple circuit}  if $v_1,\cdots, v_{k-1}$ is a simple path.
 
 A graph is connected if any two vertices of $\Gamma$ are connected by a path.
 
 A connected graph is a {\em tree}  if it does not have circuits or equivalently two vertices are connected by a unique simple path.
 
 Given any set of vertices $U$ of $\Gamma$  the graph $\Gamma|U$with these vertices and all the edges in $\Gamma$ joining two of them is the {\em full subgraph} generated by $U$.
\end{definition}
Each graph decomposes into its connected components and our goal is to study the connected components of $\Gama$ and prove Theorem \ref{main}.
\begin{theorem} \label{main}
For generic choices of $S$  the set  $S$ is a connected component of the graph $\Gamma_S$, called {\em the special component}.

 The other connected components of the graph $\Gamma_S$,    are formed by  {\em affinely independent points}.
\smallskip

In particular each non special component has at most $n+1$ points.
\end{theorem} The proof of Theorem \ref{main} is quite complex,
it requires some non trivial algebraic geometry, invariant theory  and a  very   long  and hard combinatorial analysis  which will be presented in Part 2 starting from \S \ref{drg}.
\medskip

In this paper  {\em generic}  is in the sense of algebraic geometry.  We think of $S$ as a point in $\R^{nm}$ and then we   want to find    {\em optimal}  constraints on the tangential sites $S$, given by a finite list of polynomial inequalities on the coordinates of $S$. 

If $S$ satisfies these inequalities we say that it is {\em generic} and then, hopefully these      constraints    make the graph  {\em as simple as possible}.  

These constraints will be discovered and constructed stepwise as we go along the proof.
\begin{remark}\label{suic}
\begin{itemize}
\item Several polynomial inequalities are equivalent to a unique polynomial inequality.
\item  We will have linear, quadratic,  and determinantal inequalities of degree $n,\ n+1.$
\item The  number of inequalities is finite but depends on $n,m$.
\item Most choices of $S$, even if restricted to be integral,  satisfy these inequalities.
\end{itemize}
\end{remark}

Notice that two  vectors $\mathtt v_i,\mathtt v_j\in S$  are connected by both a black and a red edge since they are vertices of a degenerate rectangle and satisfy the two equations \eqref{lan},   \eqref{lar}.
 
\begin{remark}\label{othg}
When we restrict to $S\in (\Z^n)^m$  and the arithmetic graph  one can use a stronger notion of being generic  by imposing arithmetic constraints. 

In this way one may get stronger results, as in the first paper on this subject by Geng, You, Xu \cite{GYX}, who give arithmetic conditions for $n=2$ which insure that we have components with at most 2 vertices, rather than 3 as given by our geometric Theorem. \smallskip

There is also, in this case, a weaker notion of  being generic that is that the density of the possible $S\in (\Z^n)^m= \Z^{n m}$   in the sets $B_k:=\{(a_1,\cdots, a_{nm})\mid  a_i\in\Z,\ |a_i|\leq k\}$ tends to 1 as $k\to\infty$. This is automatically true if $S$  is generic in the sense of algebraic geometry.
\end{remark} The first simple constraint is
\begin{constraint}\label{c0}
We assume that the vectors in $S$  linearly span $\R^n$.
\end{constraint}
\subsection{The special component}

The next constraint  we want  serves  to ensure that no vector $p\notin S,\ p\in\R^n$ is connected by an edge to $S$, that is $S$ is a component of the graph. 

 For this it is sufficient to assume that  any 3  vectors $\mathtt v_i,\mathtt v_j,\mathtt v_h\in S$   are not vertices of a rectangle. 
 
 This means that the triangle of vertices $\mathtt v_i,\mathtt v_j,\mathtt v_h$   has no right angle i.e. of $\pi/2$. 

\begin{constraint}\label{c1}
This is insured by 3 inequalities  $(\mathtt v_a-\mathtt v_b,\mathtt v_a-\mathtt v_c)\neq 0$ on the scalar products  of the 3 vectors sides of the triangle, we also impose $(\mathtt v_i,\mathtt v_j)\neq 0,\ \forall i,j$.
\end{constraint}    \smallskip

\begin{remark}\label{laspe}
Under the previous constraint $S$ is a component. We say that $S$ is {\em complete} and call $S$ the {\em special component}.
\end{remark}

  \begin{example}
$q=1,\ n=2,m=4$.  \quad Four vectors $\mathtt v_1,\mathtt v_2,\mathtt v_3,\mathtt v_4$  in the plane   do not satisfy Constraint 1)  if they form a picture of type
$$\xymatrix{{\circ}\, \mathtt v_1&&   &{\circ}\, \mathtt v_4\\&{\circ}\,\mathtt v_2&   &{\circ}\,\mathtt v_3 }$$  
 The point $x$ is connected to $S$ by 3 edges.$$\xymatrix{{\circ}\, \mathtt v_1&x&   &{\circ}\, \mathtt v_4\\&{\circ}\,\mathtt v_2&   &{\circ}\,\mathtt v_3 }$$ 
\end{example}

 \subsection{Combinatorial graphs}
  By fixing an element $x$  in a component,  called {\em the root},  the component is described by a marked graph  of this type
\begin{example}   $$
    \xymatrix{ &x-\mathtt v_1+\mathtt v_3\ar@{<- }[d] _{2,  1}& &\\  &x-\mathtt v_2+\mathtt v_3& &\\  x\ar@{ ->}[ru] _{3,  2}\ar@{<-}[ruu] ^{  1,3}\ar@{=}[rr]_{1,  2} && -x+\mathtt v_1+\mathtt v_2 \ar@{=}[luu]_{2,  3} \ar@{=}[lu] ^{1,  3} &   } $$\end{example}
 which encodes the linear relations  explained in Remark \ref{edg}.\smallskip
 
  This graph is completely recovered from the following combinatorial graph with two colors on vertices and the $\mathtt v_i$. A formal definition is \ref{comg}.
  \begin{equation}\label{prie}
 \xymatrix{ &  \bullet\ar@{<- }[d] _{2,  1}& &\\  &  \bullet & &\\  x\ar@{ ->}[ru] _{3,  2}\ar@{<-}[ruu] ^{  1,3}\ar@{=}[rr]_{1,  2} &&    \textcolor{Red}{\bullet}   \ar@{=}[luu]_{2,  3} \ar@{=}[lu] ^{1,  3} &   }  
\end{equation}
The color of a vertex is black   if the vertex is reached from $x$ by a path containing an even number of red edges   and red otherwise. At this point it is not clear that the color is well defined, since the vertex can be reached by different paths.

 We will see in \S \ref{colo} that under Constraints \ref{c3},\ \ref{c4}  the color is well defined.

The equations that $x$ has to satisfy  for this to be part of the rectangle graph  are obtained from those defining the various rectangles eliminating the variable vertices different from the root by the linear equations. In this example they can be organised as follows, where $(u,v)$ denotes the usual scalar product in $\R^n$:  $$  \begin{matrix}&\\&\\&\\
  (x,  \mathtt v_2-\mathtt v_3)=|\mathtt v_2|^2-  (\mathtt v_2,  \mathtt v_3) \\&\\
|x|^2-  (x,  \mathtt v_1+\mathtt v_2)=-  (\mathtt v_1,  \mathtt v_2) \\&\\
  (x,  \mathtt v_1-\mathtt v_3)=|\mathtt v_1|^2-  (\mathtt v_2,  \mathtt v_3) 
\end{matrix}
 $$  In general one has a similar list of linear and quadratic constraints on $x$, given by  Formulas \eqref{bacos}, each for a vertex of the graph different from $x$. 
 
 The equation is  linear  if the vertex is reached from $x$ by a path containing an even number of red edges (a black vertex) and quadratic otherwise (a red vertex).
  \begin{proposition}\label{eli}
By eliminating  the intermediate steps  the equations defining the various rectangles give rise for each coloured vertex (different from the root) to\begin{enumerate} \item  Each vertex $p$  is of the form $p=a+x$ if black, or $p=a-x$ if red,  with $a$ a linear combination with integer coefficients of the $\mathtt v_i$.
\item For a black vertex we have a linear equation for $x$ of the form $(x,a)=b$  with $a$ a linear combination with integer coefficients of the $\mathtt v_i$ and  $b$ a linear combination with integer coefficients of the $|\mathtt v_i|^2,\ (\mathtt v_i,\mathtt v_j)$. 
\item For a red vertex  we have  a quadratic  equation for $x$ of the form $|x|^2+(x,a)=b$  with $a$ a linear combination with integer coefficients of the $\mathtt v_i$ and  $b$ a linear combination with integer coefficients of the $|\mathtt v_i|^2,\ (\mathtt v_i,\mathtt v_j)$. \end{enumerate}  
\end{proposition}\begin{proof}
This is a simple induction by choosing a path from the root to the vertex, the explicit Formulas are \eqref{bacos}.

A priori a different path  could give a different expression for the vertex, this as we will see in \S \ref{colo}  is excluded by the constraints  \ref{c3},\ \ref{c4}.\end{proof}
Thus the first problem is to understand the exact form of these equations. This will be explained in Formula \eqref{bacos}, for this we need some algebra.

\subsection{ The Cayley graph} The conditions for 2 points to be vertices of a rectangle are linear and quadratic. We first describe an efficient way to keep track of the linear equations, which are expressed in Remark \ref{edg} and afterwards   we will show how  to define a function {\em quadratic energy} with which to express the linear and quadratic equations (see  \eqref{bacos}).\smallskip

How to describe the possible combinatorial graphs appearing in the geometric graph?

This is done through the idea of  Cayley graph. Cayley graphs are a useful tool of group theory  to visualise monomial relations among group elements. 

 The formal definition is the following.

 \smallskip

 Let $G$ be a group and $X=X^{-1}\subset G$ a subset  (by $X^{-1}$ we denote $\{g^{-1}.\mid g\in X\}$).
   Consider an action $G\times A\to A$ of $G$ on a set $A$,   we then define.
 \begin{definition}\label{cg}[Cayley graph] The graph $A_X$ has as  vertices   the elements of $A$ and,   given $a,  b\in A$ we join them by an oriented edge $\xymatrix{a\ar@{->}[r]^{x} &b }      $,   marked $x$,   if $b=xa,  \ x\in X$.
 \end{definition} The condition     $X=X^{-1}$ is used so  that $\xymatrix{a\ar@{->}[r]^{x} &b }  \iff  \xymatrix{a\ar@{<- }[r]^{\quad x^{-1}} &b }           $.
 \smallskip

 Cayley graphs are very useful in group theory.  In particular when $G$ acts on itself by multiplication and its Cayley graph is denoted $G_X$. 
 
 Different paths in the Cayley graph give relations  among the elements $X$. The graph $G_X$  is connected if and only if $X$ generates $G$.
  \medskip

  \includegraphics[height=4cm]{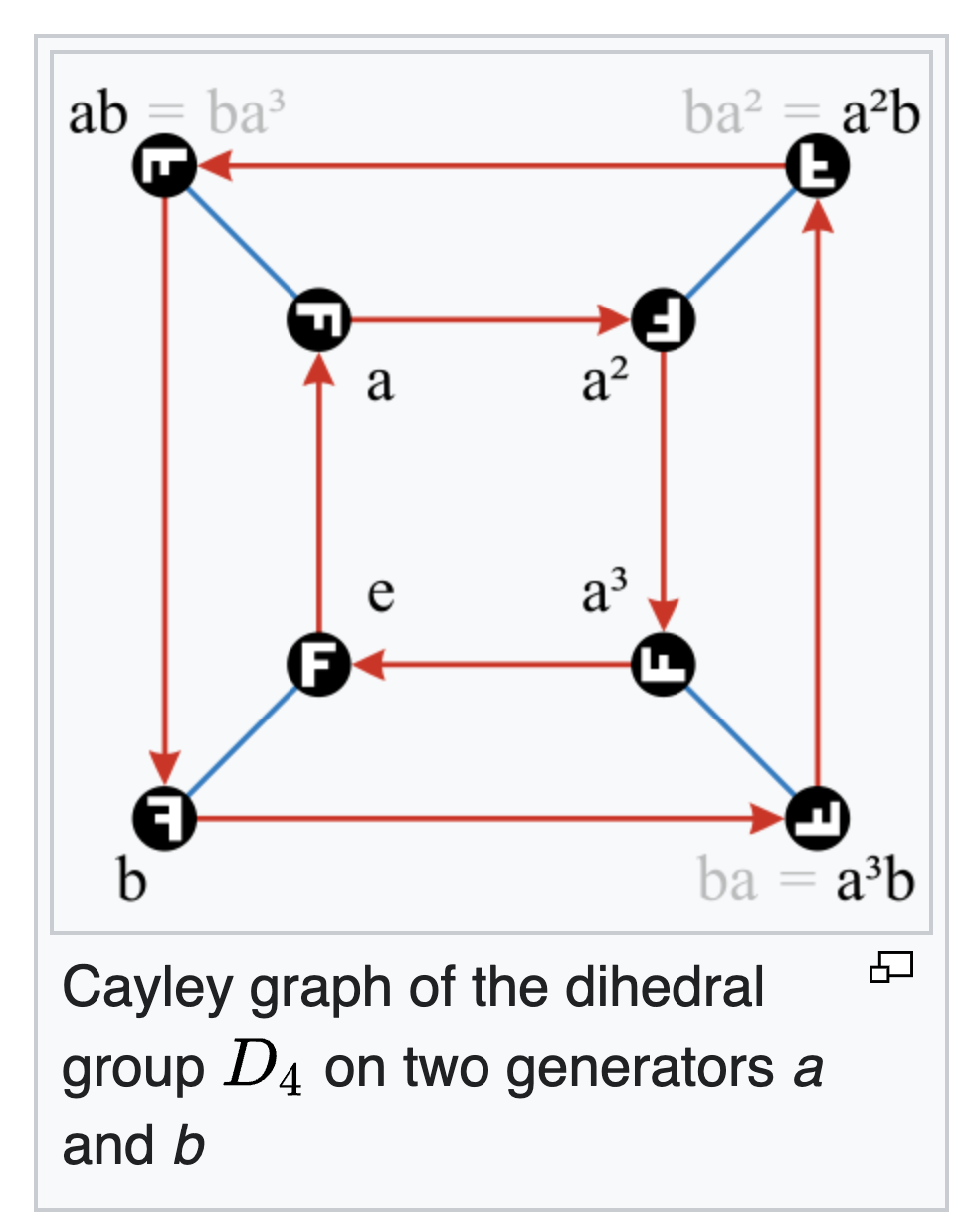}  
\vskip-3cm\hskip 4cm The 8  symmetries of a square. $e$ is the identity,  $a$ is the 

\hskip 4cm rotation by  $\pi/2$ and $b$ the reflection. 

\hskip 4cm $a^4=e,\ b^2=e,\ ab=ba^3$.
\vskip3cm
 
  \begin{remark}\label{isocg}
Right multiplication  by an element $g\in G$  gives an isomorphism of the Cayley graph $G_X$.
\end{remark}
In our setting the relevant group $G$  is   the group of transformations of $ \Z^m$ (or  $ \R^m$)   generated by translations $a:x\mapsto x+a,\ a\in \Z^m$ and {\em sign change} $\tau:x\mapsto -x$. 
 $$\text{We have}\quad G:={\Z^m}\rtimes\Z/  (2) =\Z^m \cup \Z^m\tau\qquad\text{is a semidirect product} $$ and the product rule is $a\tau=-\tau a,\ \forall a\in\Z^m$ (notice that this implies $(a\tau)^2=0$). 
 
 Sowe have  the composition Formulas (denote by $\circ$ the group composition)
 \begin{equation}\label{cofu}
a,b\in\Z^m,\quad a\circ b=a+b, \quad a\tau\circ b=(a-b)\tau,\ \tau^2=0.
\end{equation}
  \smallskip

  In order to express in a compact form the equations of compatibility  we need to extend our group to real  linear combinations of the $e_i$  identified to   $\R^m$: \begin{equation}\label{gr}
G ={\Z^m}\rtimes\Z/  (2)\subset G_{\R} ={\R^m}\rtimes\Z/  (2) =\R^m \cup \R^m\tau 
\end{equation} which acts  on itself and on $\R^m$ as $G$ does.\smallskip

Having chosen $S\subset  \R^n$ the groups $G,\ G_\R$ act  also geometrically on $\R^n$  by defining $$\pi:\R^m\to \R^n,\quad \pi(\sum_ia_ie_i ):=\sum_ia_i\mathtt v_i .$$
We then define  an action of $  G_\R$ on $\R^n$ by setting, for    $g\in G_\R,\  x\in\R^n$:
  \begin{equation}\label{lazz}
g=a\in\R^m,\  g\cdot x:=-\pi(a)+  x,\footnote{the choice of the minus sign is due to conservation laws in the NLS}\quad \tau x:= -  x.
\end{equation} In particular 
  \begin{equation}\label{pig}
 g\cdot   0=-\pi(a),\ g=a,\ g=a\tau.
\end{equation}
  
\subsection{The case of $S$} We have  by this definition
  $$(e_i-e_j )x=\mathtt v_j-\mathtt v_i+x,\  (-e_i-e_j )\tau x= \mathtt v_i+\mathtt v_j -x$$ which are a possible black and a red edge, see Remark \ref{edg}. Therefore we can also identify the edges as these elements of $G$.\smallskip

\begin{definition}\label{cg1}  We denote by
\begin{equation}\label{lix}
X=X_0\cup X_2,\  X_0:=\{(e_i-e_j )\} ,\ X_2 :=\{(-e_i-e_j )\tau\},\  \forall i\, \neq j\in  \{1,2,\cdots,m\}.
\end{equation}
We consider the Cayley graphs $G_X  \subset G_{X,\R} $ generated by these elements, in $G$  and $G_\R$ respectively, and $\R^n_{ X}$ generated by the action of $G_\R$ on $\R^n$.
\end{definition}  \begin{proposition}\label{coo}
If we have a sequence of points $p_1,p_2,\cdots ,p_k\in\R^n$ with $p_i,\ p_{i+1}$ connected by some edge $ \ell_i, $  (a path) we have\begin{equation}\label{gcom}
p_k=g\cdot p_1,\quad g=\ell_{k-1}\circ\ell_{k-2}\circ\cdots \circ\ell_{2}\circ\ell_{ 1}.
\end{equation}
\end{proposition}\begin{proof} By definition $p$ is connected to $q$ by an edge $\ell$ if  $q=\ell p$, then the proof is by induction.\end{proof}

\begin{remark}\label{cogg} The geometric graph $\Gama$  is thus a subgraph of    the Cayley graph $\R^n_{ X}$  defined by imposing the quadratic equations to the edges.\smallskip

Under the orbit maps  $\rho_x:  G_{X,\R}\to \R^n,\ \rho_x(g)=g\cdot x$, the graph $G_{X,\R} $ maps surjectively  to the Cayley graph $\R^n_{ X}$.\smallskip

We will see in Example   \ref{ecv},    that this map is not injective but a {\em covering} of the graphs.

\begin{itemize}
\item In fact for all $g=\sum_im_ie_i$ with $ \sum_im_i\mathtt v_i=0$ we have $g\cdot x=x$ for all $x$. 

\item For all $g=(\sum_im_ie_i)\tau$ we have $g\cdot x=x$ if and only if $2x=- \sum_im_i\mathtt v_i$. 

\end{itemize}

So the stabilizer $H_x$ of $x$  in $G$  is non 0, as soon as $m>n$. 

The group  $H_x$  is either the kernel of the map $\pi$  or, in case   $2x= - \sum_im_i\mathtt v_i.\ m_i\in\Z$ it also contains another coset  of this Kernel inside $\Z^m\tau$. \smallskip

 We may identify the orbit $G_{X,\R}x=  G_{X,\R}/H_x $. This is a quotient also as graphs and also as the topological spaces associated to the graphs.\smallskip

For $a\in \R^m$  the stabilizer in $G$  is trivial unless $a\in\frac12\Z^m$  when it has 2 elements  $1,\ 2a\tau$.\end{remark}
\subsubsection{Orbit maps}
Let us make a brief general digression, it will be used in \S \ref{colo}. 

 If $G$  acts on a set $A$  then $A$ decomposes into its $G$ orbits. For a given $a\in A$ let  $G_a:=\{g\in G\mid g\cdot a=a\}.$  This is a subgroup of $G$ the {\em stabilizer of $a$}. \smallskip

The orbit $G\cdot a$   is identified with the set $G/G_a$  of its {\em left cosets  $gG_a$.}\smallskip

Now a set $X=X^{-1}$  defines two Cayley graphs     $G_X, A_X$  in $G$ and $A$ respectively, and the orbit map  $\rho:g\mapsto g\cdot a$  is a map of graphs.

Take a  subset $U\subset A$  containing $a$ and such that the full subgraph $\Lambda\subset A_X$  (Definition \ref{term})   on the vertices $U$ is connected.  Consider the  full subgraph   $\rho^{-1}(\Lambda)\subset G_X$ formed by the elements $\rho^{-1}(U)$. 

  \begin{lemma}\label{mcc}
Under the orbit map  $\rho^{-1}(U)$  maps to $\Lambda$ and  
  each connected component of  the graph $\rho^{-1}(\Lambda)$  maps onto  $U$.
\end{lemma}
\begin{proof}
Take any  $h\in \rho^{-1}(U)$ and let $b=\rho(h)\in U$.  If  $c\in U$ since the graph $\Lambda $    on the vertices $U$ is connected there is a path from $b$  to $c$  with edges $\ell_i\in X$. The same sequence of edges defines a path  from $h$  to some element $k$  in the connected component of $h$  in $\rho^{-1}(\Lambda)$ lifting the given path so $\rho(k)=c$  and the claim follows.
\end{proof}
Denote by $C_1$ the connected component  of the identity $1\in G$ of  $\rho^{-1}(\Lambda)$.

\begin{lemma}\label{leal} Take a connected component $C$  of $\rho^{-1}(\Lambda)$ and an element $g_0\in C$ with $\rho(g_0)=1$ (Lemma \ref{mcc}).  Then $C=C_1g_0$. 
\end{lemma}
\begin{proof}
If $a=gg_0,\  g\in C_1$  we have $\rho(a)=\rho(g)\in U$.  Then any path from $g_0$ to $a$ in $C$ corresponds to a path from $1$ to $g$ in $C_1$.
\end{proof}
\begin{proposition}\label{bigi}
The orbit map  $\rho$  induces for each connected component of  $\rho^{-1}(\Lambda)$ an isomorphism to    $\Lambda$, if and only if the connected component $C_1$ of the identity $1\in G$ of  $\rho^{-1}(\Lambda)$  intersects $G_a$ only in 1.
\end{proposition}
\begin{proof} By the previous Lemma it is enough to treat the connected component $C_1$.
If there is an element $g\in G_a,\ g\neq 1$  in the connected component $C_1$  of 1  then $g\cdot a=1\cdot a=a$ and so the map is not injective.   

Conversely if given $h,k\in C_1$ we have $h\cdot a=k\cdot a$ then $h^{-1}k\in G_a$.  We  need to show that $h^{-1}k\in C_1$. By definition of  $\rho^{-1}(\Lambda)$ to say that $h\in C_1$  means that there is a sequence of edges $\ell_i\in X,\ i=1,\cdots,p$  so that, setting by induction $h_1=1,\  h_{i+1}=\ell_i h_i$  we have that  $h=h_p$  and  for each $i$ the two elements $h_{i-1}a,\  h_ia=\ell_ih_{i-1}a\in  U$  are connected by the edge $\ell_i$ in $\Lambda$.  We have $$ h^{-1} =\ell_1^{-1}\ell_2^{-1}\cdots \ell_p^{-1}.$$ Thus going back from $h\cdot a=k\cdot a$ to $a$  with the edges $\ell_i^{-1}$  we just walk back to $a$  remaining in   $\Lambda$  this means that  $h^{-1}k\in C_1$. 
\end{proof} If there is an element $g\neq 1,\ g\in C_1\cap  G_a$  we have $C_1g=C_1$. So  $C_1\cap G_a$  is a subgroup $H$ of $G_a$  acting on $C_1$  and naturally $\Lambda= C_1/H$.

\subsubsection{The subgroup $G_2$}

   Let $G_2$  be the subgroup of the group $G$  generated by the elements  $(e_i-e_j ), (-e_i-e_j)\tau$.\medskip
  
  Given $a=\sum_i\nu_ie_i  $ set $
 \eta(a):=\sum_i\nu_i \in \Z$. \footnote{Sometimes one refers to $\eta(a) $ as the {\em mass} of $a$.}  We have
$$\eta((e_i-e_j )a)=\eta(e_i-e_j+a)=\eta(a),$$$$ \eta((-e_i-e_j )\tau a)=\eta(-e_i-e_j-a)=-2-\eta(  a)$$
One easily verifies that:
 $$G_2:=G_{2,+}\cup G_{2,-},\   G_{2,-}= G_{2,+}\tau $$
 $$G_{2,+}:=\{a\in\Z^m\mid\eta(a)=0,\ G_{2,-}:=\{a\tau,\ a\in\Z^m\mid\eta(a)=-2\}.$$
Of course $G_{2,+}$ is a subgroup of index 2 of $G_2$.  In particular $G_2$  can be identified to  the orbit of 0 under $G_2$ in $\Z^m$
\begin{equation}\label{orbG}
\Z^m_2:=G_{2}\cdot 0=\{a\in \Z^m, \eta(a)=0, -2\}.
\end{equation} 
We call {\em black} the points $a\in \Z^m_2$ with $\eta(a)=0$  and {\em red}  the ones   with $\eta(a)=-2.$  \medskip

The composition law  of two such integral vectors as group elements is:
\begin{equation}\label{claw}
a\circ b=a+(\eta(a)+1)b,\ \ a\circ b=a+b\ \text{if}\ \eta(a)=0,\  a\circ b=a-b\ \text{if}\ \eta(a)=-2.
\end{equation} 
It is also convenient to write an element of $G_2$ as the pair $(a,\eta(a)+1),\ a\in\Z^m_2$ and the ones in $\Z^m_2$ as  pairs $(a,\pm 1).$ 
\medskip

\begin{remark}\label{cc}
The group  $G_2$ is a connected component of  $G_X$ and $G_{X,\R} $,  and the other components are its  right cosets  $G_2g,\ g\in G_{X,\R}  $.\smallskip

The connected components of $\R^n_{X}$ are the $G_2$ orbits.
\end{remark}

As for the graph in $\Z^n$ or in $\R^n$,
a path of edges  starting from some $x$  reaches a point $y$ obtained from $x$ by applying the corresponding product of elements, by \eqref{claw}.

\begin{equation}\label{lieq}
y=\pm x+\sum_in_i\mathtt v_i,\ n_i\in\Z.
\end{equation} 
\begin{proposition}\label{lieq1}
Formula  \eqref{lieq} expresses the linear equations for the vertices of $\Gama$ in  Proposition \ref{eli}.
\end{proposition}
\begin{remark}\label{edg1}\strut  We can define, using $S$,  a  subgraph $\Gamat$ of the Cayley graph $\R^m_X$  as in  Remark \ref{edg}, formed by edges {\em compatible with $S$}:
 
\begin{itemize}\item An oriented  black edge $\xymatrix{  p\ar@{->}[r]^{e_i-e_j} &q}$   connecting  two points $p , \ q=e_i-e_j+p \in \R^m$ is compatible with $S$ if $-\pi(p),\ -\pi(q)$   are  {\em adjacent} in the rectangle with vertices     $-\pi(p), -\pi(q),\mathtt v_j,\mathtt v_i$  hence $$-\pi(q)=-(e_i-e_j)\pi(p)=-\pi(p)+\mathtt v_j-\mathtt v_i.$$ 

\item A red edge $\xymatrix{  p\ar@{=}[r]^{(-e_i-e_j)\tau} &q}$  which connects  two points $p, q=-e_i-e_j-p\in \R^m  $ is compatible with $S$ if $-\pi(p),\ -\pi(q)$   are {\em opposite} in the rectangle with vertices     $-\pi(p), \mathtt v_j,-\pi(q),\mathtt v_i$   hence $$-\pi(q)=-(-e_i-e_j)\tau \pi(p)=\pi(p)+\mathtt v_j+\mathtt v_i$$

\end{itemize}
\end{remark}   With this notation it is important  to make sure that two different  combinatorial edges which appear in the Cayley graph do not determine the same  the geometric edge. 

This is insured by the next constraint    
\begin{constraint}\label{c2}
If $e_i-e_j\neq e_h-e_k$ we require $\mathtt v_i-\mathtt v_j\neq \mathtt v_h-\mathtt v_k$. Similarly if  $e_i+e_j\neq e_h+e_k$ we require $\mathtt v_i+\mathtt v_j\neq \mathtt v_h+\mathtt v_k$.
\end{constraint}
In fact later we shall  use the further constraint
\begin{constraint}\label{c3}
$\sum_{i=1}^m\nu_i\mathtt v_i\neq 0,\ \forall \nu_i\in\Z,\ \mid \sum_{i=1}^m|\nu_i|\leq 4(n+1)$.
\end{constraint}

\begin{definition}\label{comg}
A {\em combinatorial graph} is a finite full subgraph (Definition \ref{term}) of the graph $G_X$ in  $G_2\equiv G_2\cdot 0$  containing 0.
\end{definition}

\begin{example}[Combinatorial graph]\label{es1}
$$ \xymatrix{  && \textcolor{Red}{\bullet} -e_2-e_1 \ar@{->}[r]_{ e_1-e_3} &  \textcolor{Red}{\bullet} -e_2-e_3    \ar@{=}[dl]^{-e_2-e_3}& &&\\ &e_4-e_5  \bullet   \ar@{<-}[r] _{e_4-e_5}&0 \ar@{=}[u]^{-e_2-e_1}  \ar@{->}[r]_{e_4-e_2} &\bullet  e_4-e_2}$$   
\end{example}      
     
    So the previous  example applied to some $x\in
R^n$ gives:
\begin{example}[Geometric Avatar]\label{es11}
$$ \xymatrix{  && \mathtt v_2+\mathtt v_1 -x \ar@{->}[r]_{ \mathtt v_3-\mathtt v_1} & \mathtt v_2+\mathtt v_3-x \ar@{=}[dl]^{\mathtt v_2+\mathtt v_3}& &&\\ &  \mathtt v_5-\mathtt v_4+x \ar@{<-}[r] _{\mathtt v_5-\mathtt v_4}&x \ar@{=}[u]^{ \mathtt v_2+\mathtt v_1}  \ar@{->}[r]_{\mathtt v_2-\mathtt v_4}
 & \mathtt v_2-\mathtt v_4+x}$$ 
\end{example} If this graph is contained in a component  of $\Gamma_S$ we  say  that  {\em  it is compatible with $S$}.

The condition    is that  the 4 vertices satisfy  4 linear and 4 quadratic equations $$ a=\mathtt v_5-\mathtt v_4+x,\ e = \mathtt v_1+\mathtt v_2-x,\  c=\mathtt v_2+\mathtt v_3-x,\  d=\mathtt v_2-\mathtt v_4+x $$$$ |a|^2-|x|^2=|\mathtt v_5|^2-|\mathtt v_4|^2,\ |e|^2+|x|^2=|\mathtt v_1|^2+|\mathtt v_2|^2,\ldots .$$ One can eliminate, using the linear equations,  all vertices different from the root and finally obtain a system of linear and quadratic equations for $x$. Our next task is to understand these equations in general, see \eqref{bacos}.
      \subsection{The quadratic energy constraints\label{gerea}} In order to discuss, in Proposition \ref{main1}, the quadratic equations of Proposition \ref{eli} we need to use the Cayley graph in $\R^m$ and introduce a quadratic function on $\R^m$.\medskip
      
        Denote $a\in\R^m$  by $(a,1)$  and   $a\tau,\ a\in\R^m$  by $(a,-1)$.\smallskip

We want to formalize the proof of Proposition \ref{main1} as follows.\medskip

We consider $\R^n$   with the standard scalar product $(a,b)$.  \smallskip

\begin{enumerate}\item Given a list $S$ of  $m$ vectors $\mathtt v_i\in\R^n$, we  have defined  the  linear map \begin{equation}\label{ilpi}
\qquad \qquad \pi:\R^m\to\R^n,\quad e_i\mapsto \mathtt v_i . 
\end{equation}
\item 
Let    $S^2[{\R^m}]:=\{\sum_{i,j=1}^ma_{i,j}e_ie_j\},\  a_{i,j}\in\R$ be   the  polynomials of degree $2$  in the variables $e_i$ with real coefficients. 

We extend  the map $\pi$ to a linear map of $S^2[{\R^m}]$ to quadratic  polynomials on $\R^n$, and introduce a linear map $L^{(2)}: a\mapsto a^{(2)}\in S^2({\R^m})$ as:
  \begin{equation}\label{ipi}  \pi(e_i) = \mathtt v_i,\ 
\pi(e_ie_j):=(\mathtt v_i,\mathtt v_j),  \ L^{(2)}:{\R^m}\to S^2({\R^m}),\   a= \sum a_i e_i \mapsto a^{(2)}:= \sum a_i e_i^2 .
\end{equation}

\item We have $   \pi(AB)=(\pi(A),\pi(B)),\forall A,B\in{\R^m}.$
\end{enumerate}

\begin{remark}
Notice that we have   $a^{(2)}=a^2$ if and only if  $a=0$   or $a=e_i$, for  one of the variables $e_i$.
\end{remark} 
 
\begin{definition}
Given an element $u= (a,\sigma)=(\sum_ia_ie_i,\sigma) \in G_\R,\ \sigma=\pm 1 $ set 
 
\begin{equation}\label{ricon} 
 C(u):= \frac{\sigma }{2}(a^2+a^{(2)}),\quad    K(u) : =\pi(C(u) )=\boxed{\frac{\sigma }{2} (|\sum_ia_i\mathtt v_i|^2+\sum_ia_i|\mathtt v_i|^2)}.
\end{equation}
 \smallskip
 
  \end{definition}
 \begin{remark}
Notice that if $a\in\Z^m$  then  $C(u)$ has integer coefficients (for instance we have $C(e_1+e_2)=e_1^2+e_2^2+e_1e_2$) so $K(u)$  is a quadratic polynomial in the coordinates of the vectors $\mathtt v_i$ with integer coefficients.
\end{remark} 
In particular we have
 \begin{equation}\label{ilK}
K(e_i-e_j)=\frac 1{2}( |\mathtt v_i - \mathtt v_j|^2+|\mathtt v_i|^2-|\mathtt v_j|^2)= |\mathtt v_i|^2-(\mathtt v_i,\mathtt v_j)=(\mathtt v_i,\mathtt v_i -\mathtt v_j)
\end{equation}
 \begin{equation}\label{ilK1}
K((-e_i-e_j)\tau)=-\frac 1{2}( |\mathtt v_i + \mathtt v_j|^2-|\mathtt v_i|^2-|\mathtt v_j|^2)=  -(\mathtt v_i,\mathtt v_j)
\end{equation}

These Formulas coincide with the right hand side of formulas  \eqref{iperp} and \eqref{sfera}.\medskip

  \subsubsection{Composition}
For   $u= (a,\sigma) $ and 
   $ g=(b,  \rho)$ consider $g\cdot u=( b +\rho a,\rho \sigma).$ We have 
   $$C(g\cdot u)= \frac{\sigma\rho }{2}\left(( b +\rho a)^2+( b +\rho a)^{(2)}\right)= \frac{\sigma\rho }{2}\left(  b^2 +  b^{(2)} +2\rho ab+a^2 +\rho a^{(2)}\right)$$
$$ =\frac{\sigma\rho }{2}\left(  b^2 +  b^{(2)}\right) + \sigma ab+ \frac{\sigma }{2}\left( \rho a^2 + a^{(2)}\right) =\frac{\sigma\rho }{2}\left(  b^2 +  b^{(2)}\right) + \sigma ab+ \frac{\sigma }{2}\left( (\rho -1)a^2 +a^2+ a^{(2)}\right). $$
Therefore:\begin{proposition}\label{comC} With the previous notations:
\begin{equation}\label{vee}
C(g\cdot u)=\sigma   C(g)+ C(u)+ (\rho-1) \frac{\sigma }{2}  a^2+ \sigma    ab.
\end{equation}
$$\implies  K(g\cdot u)=\sigma   K(g)+ K(u)+ (\rho-1) \frac{\sigma }{2}  |\pi(a)|^2+ \sigma   (\pi(a) ,\pi(b)).$$
\end{proposition}
 From \eqref{vee} we see that $K(g\cdot u)=  K(u)$ if and only if:
\begin{equation}\label{vee1}
\begin{cases}
 i)\quad K(g)= -   (\pi(a),\pi(b)),\  \rho=1\\  i i)\quad  K(g)=  |\pi(a)|^2-   (\pi(a),\pi(b)),\  \rho=-1
\end{cases}.
\end{equation}$K$ is called the {\em energy function  } on $G_\R$.\footnote{In the theory of the NLS this appears as  a conservation law.}\medskip

With the notations of Remark \ref{edg} we have the fundamental reason to introduce the function $K(u)$:
\begin{theorem}\label{fo}
Two points  $u=(a,\sigma),\ v=\ell\cdot u\in G_\R,\ \ell\in X$  have $K(u)=K(v)$  if and only if  $p:=u\cdot 0=-\pi(a),\ q:=v\cdot 0 $ are connected by the edge marked    $\ell$ compatible with $S$.
\end{theorem}
 \begin{proof} Since   $q=v\cdot 0=\ell\cdot u\cdot 0$ we have $q=\ell\cdot p$.  Now the compatibility with $S$ is given:\smallskip

\begin{enumerate}\item If $\ell=e_i-e_j,\ a\in\R^m$  we have $K(e_i-e_j) = |\mathtt v_i|^2-(\mathtt v_i,\mathtt v_j)$.
The condition  $K(u)=K(v)$ is from Formula \eqref{vee1}  i)  and \eqref{ilK} applied to $g=\ell=e_i-e_j$
$$  |\mathtt v_i|^2-(\mathtt v_i,\mathtt v_j)=-  (\mathtt v_i-\mathtt v_j, \pi(a)) =-  (\mathtt v_j-\mathtt v_i, p).$$
This means that the two points  $u\cdot 0=-\pi(a),\ \ell\cdot u\cdot 0=-\pi(e_i-e_j+  a)=-\mathtt v_i+\mathtt v_j+u\cdot 0$ are the vertices of a black edge marked  by $\mathtt v_j-\mathtt v_i$,  compatible with $S$,  Formula \eqref{lan} and  Remark \ref{edg1}.\smallskip

\item  If $\ell=-e_i-e_j,\ a\in\R^m$  we have
$K((-e_i-e_j)\tau)=  -(\mathtt v_i,\mathtt v_j)$.   The condition is  from Formula \eqref{vee1}  ii)  and \eqref{ilK1}  applied to $g=\ell=-e_i-e_j$
$$  -(\mathtt v_i,\mathtt v_j)=  |\pi(a)|^2+  (\pi(a), \mathtt v_i+\mathtt v_j)= |p|^2- (p, \mathtt v_i+\mathtt v_j).$$
This means that the two points  $u\cdot 0=-\pi(a),\ \ell\cdot u\cdot 0=-\pi(-e_i-e_j-  a)=\mathtt v_j+\mathtt v_i -u\cdot 0$ are the vertices of a red edge marked  by $-\mathtt v_i-\mathtt v_j$, Formula   \eqref{lar} and  Remark \ref{edg1}. \end{enumerate}\end{proof}

Observe that for $g\in G_\R$  we have  $K(g\tau) =-K(g)$.\smallskip

{\bf Warning}\quad The function $K(u)$ is defined only on $G_\R$ and not on $\R^n$ where the geometric graph  $\Gamma_S$ lives.  But we have the following:
\begin{proposition}\label{indp}
For  $q=(q_1,\cdots,q_m) \in \R^m$  with  $\pi(q)=0$  set $\phi(q):= \frac12(\sum_iq_i|\mathtt v_i|^2).$ 

Then $K(q)= \phi(q)$ and, for any $p\in \R^m$ we have  \begin{equation}\label{trasl}
K(p+q)=K(p)+K(q)=K(p)+\phi(q)\end{equation}\end{proposition}  \begin{proof}
Apply Formula \eqref{vee}  to $p,q\in \R^m$  with $\pi(q)=0$.   If $q=(q_1,\cdots,q_m)$  we have 
$$K(q)=\frac12(\sum_iq_i|\mathtt v_i|^2),\quad K(p+q)= K(p)+K(q)=K(p)+ \frac12(\sum_iq_i|\mathtt v_i|^2).$$  
\end{proof}
\begin{definition}\label{ilG} We  define $\Lambda_{S,\R}$  (resp.    $\Lambda_{S,\Z}$) to be   the subgraph of the Cayley graph $G_{X,\R}$   (resp. $G_X$) in which we only keep as edges the ones which preserve the energy function $K$.\smallskip

For each $a\in\R$  we denote by $G_{X,\R}^a$ the subgraph of $G_{X,\R}$ formed by the vertices $p\in G_{X,\R}$ with $K(p)=a$.\smallskip

By  definition  $G_{X,\R}^a$    is a full subgraph of $\Lambda_{S,\R}$ which is the union of the $G_{X,\R}^a,\ a\in \R$. 
\end{definition}
\begin{corollary}\label{mapg}[Of Theorem \ref{fo}]
Under the orbit map  $g\mapsto g\cdot 0,\ \in\R^n$ the graph  $\Lambda_{S,\R}$ maps to the geometric graph $\Gamma_S$ as a surjective graph morphism.
\end{corollary}

Moreover the subgraph  $\Gamat\subset \R^m_X$ (see \ref{edg1})     in $\R^m $    is  obtained     keeping only the edges which preserve the function $K$.
 
\subsubsection{Connected components} Take $x\in\R^n\setminus S$  we want to study the connected component $C_x$ of the graph $\Gama$  containing $x$, using the ideas of the previous section.  

The first remark is that $C_x\subset G_2\cdot x$  by  Remark \ref{edg1} and the orbit $G_2\cdot x$ is a connected component $A_X$ of the Cayley graph $\R^n_X$ isomorphic to the coset space $A:=G_2/H_x,\ H_x:=\{g\in G_2\mid gx=x\}$.

Choose $p\in \R^m$  so that $p\cdot 0=-\pi(p)=x$ and let $a:=K(p)$, this depends on the choice of $p$ and it exists by Constraint \ref{c0}.

We have $g\cdot x=-\pi(g\cdot p)$  and thus we can define  the subgraph of  the Cayley graph of $(G_2)_X$ in which we only keep   the elements $g\in G_2$  with $K(g\cdot p)=a=K(p)$. 
\begin{lemma}\label{indp1}
The previous graph depends only on $x$ and not on $p$  with $p\cdot 0=x$.
\end{lemma}  
\begin{proof} 
Let $r=p+q\in\R^m$ be such that $r\cdot 0=r\cdot 0=x$ we have thus $  \pi(q)=0$  and  let  $h,k\in G_2$ be such that $k=\ell h,\    K(k\cdot p)=K(h\cdot p)$ we have by proposition \ref{indp}  $$K(k\cdot r) =K(k\cdot p)+\phi(q),\  K(h \cdot r) =K(h\cdot p)+\phi(q)\implies  K(k\cdot r) = K(h \cdot r) $$ and so the graphs are the same.
\end{proof}
\begin{definition}\label{g2x}
Let us denote this graph $G_2^x$, by definition this is a full subgraph of the Cayley graph $(G_2)_X$.  \smallskip

By $\widetilde C_x$  denote  the connected component  of  0  in the graph  $G_2^x$\end{definition}
By definition two  elements $h,k\in G_2$  are joined  by an edge $\ell$  in  $G_2^x$ if and only if  $k=\ell h$ and  $K(h\cdot p)= K(k\cdot p)=K(p)$. By  Theorem \ref{fo} if we consider the orbit map  $\rho:G_2\to  G_2\cdot x, \ \rho(g):=g\cdot x=-\pi(g\cdot p)+x$ the previous condition is that the two points $h\cdot x, \ k\cdot x$  are joined by $\ell$  in $\Gama$.  

In particular   the connected component  of  0  in the graph  $G_2^x$, $\widetilde C_x$ under the orbit map $\rho$   maps to $C_x$.
\begin{theorem}\label{fuga} The map $\rho:\widetilde C_x\to C_x$ is surjective  and, if it is also injective  
the graph $C_x$ is a full subgraph  of the Cayley graph $A_X$.
\end{theorem}
\begin{proof}
This follows from Corollary \ref{mapg} and the previous discussion.
\end{proof}
Our next goal is to prove that  \begin{theorem}\label{noinf}
Under further constraints on $S$,  for all $x\in\R^n$,  we have that $\widetilde C_x$  is finite and isomorphic to  $C_x$.
\end{theorem}
\smallskip

We will use  the conditions of Proposition \ref{bigi}, to prove this in Proposition \ref{main1} by introducing further constraints,  but first

%
%
%
\begin{definition}\label{Lp}
   The graph  $\widetilde C_x\subset G_2^x$ is called {\em the combinatorial graph} associated to $x$.

We say that a combinatorial graph  $\Gamma$ (Definition \ref{comg})  has a geometric realization in $\Gama$  if there is a $x\in \R^n$ so that 
$\Gamma \subset G_2^x.$

\end{definition} By our discussion to say that  $\Gamma$   has a geometric realization in $\Gama$ for some $x\in\R^n$ means that $\Gamma\cdot x\subset C_x$, with $C_x$   the connected component   of the graph $\Gama$  containing $x$.
  \begin{remark}\label{fus}

If a  subgraph  $\Gamma$ of the Cayley graph of $G_2$ has a geometric realization  then so has any of its translates  $\Gamma\cdot g^{-1},\ g\in \Gamma.$

By choosing a $g\in \Gamma$  we have (in different ways)  graphs $  \Gamma g^{-1}$ containing 0 (in the position where first was $g$), that is combinatorial graphs, which have a geometric realization.  What changes is the {\em root}  of the connected graph.  We say that two such graphs are {\em equivalent} see Example \ref{es1} and \eqref{es11}, \eqref{es12}.
\end{remark} In Example $\Gamma$ given by \ref{es1} if we choose for $g=e_4-e_2$ then  $  \Gamma g^{-1}$  becomes:
 \begin{equation}\label{es12}
 \xymatrix{  &&\! \! \! \! \! \! \! \! \! \! \! \! \! \! \! \! \! \! \! \! \! \! -2e_2-e_1+e_4 \textcolor{Red}{\bullet}  \ar@{->}[r]_{ e_1-e_3} &  \textcolor{Red}{\bullet} -2e_2-e_3  +e_4  \ar@{=}[dl]^{-e_2-e_3}& &&\\ &e_2-e_5  \bullet   \ar@{<-}[r] _{e_4-e_5}&e_2-e_4 \ar@{=}[u]^{-e_2-e_1}  \ar@{<-}[r]_{e_2-e_4} &   0}.
\end{equation} 
\begin{remark}\label{fing}
Given any integer $k$  there are only finitely many  combinatorial graphs  with at most $k$ vertices.  Our strategy will be to impose constraints  which will exclude some of these combinatorial graphs  to  appear in the geometric graph.
\end{remark}
\subsection{The geometric avatars\label{colo}}
We now pass to the {\em geometric avatars} of   $G_2^x$ in $\R^n$  so let $x:= p\cdot 0=-\pi(p)$.\medskip

By definition   under the action map  $G_2^x\cdot x$ an edge of  $G_2^x$ gives rise to a corresponding edge in the geometric graph and the image of $G_2^x$ lies the component $C_x$ of the geometric graph $\Gama$  containing $x$.  \smallskip

Conversely, given   the component $C_x$ of the geometric graph $\Gama$  containing $x\notin S$ we can {\em lift it}  in the sense that:
\begin{lemma}\label{lift} The map  $g\to g\cdot x$ from  $\widetilde C_x$  to $C_x$ is surjective.

\end{lemma} 
\begin{proof}
This is essentially trivial.  If one has a vertex  $q\in C_x$  of the form $q=g\cdot p,\ g\in\widetilde C_x$  and $q'\in C_x$ with $q'=\ell q,\ \ell\in X$ then by Theorem \ref{fo}, 
 $\ell\cdot g\in \widetilde C_x$.\end{proof}  In general the map  $\gamma:\widetilde C_x\to C_x$      is a       covering of graphs. We easily see that  the two graphs are isomorphic if and only if every circuit in $C_x$ is also a circuit in $\widetilde C_x$.   This is essentially the content of Proposition \ref{bigi}  which we reformulate with a different proof in our special setting.

In general  $\gamma$ is not injective, due to the fact that circuits   in  $C_x$ may unravel into larger circuits of infinite strings of edges in $G_2^x$.  The following are two examples.

    \medskip

 There can be two cases: 1.\quad  the circuit in $C_x$ contains an even number of {\em red edges}.  2. \quad the circuit in $C_x$ contains an odd number of {\em red edges}.

\begin{example}\label{ecv}[Case 1]
\quad suppose that the geometric graph contains a subgraph
 $$ \xymatrix{  &\ar@{->}[dl]_{\mathtt v_2-\mathtt v_4}k_3\ar@{<-}[dr]^{ \mathtt v_2-\mathtt v_3\quad}& \\ k_1\ar@{->}[rr]^{ \mathtt v_2-\mathtt v_1\quad}& & k_2  },$$ this happens if  we have the linear relation $0=\mathtt v_1-3\mathtt v_2+\mathtt v_3+\mathtt v_4$.  
 
  So for $g=e_1-3e_2+e_3+e_4\neq 0$  we have $g\cdot  k_1=k_1$. This graph is in $\Gamma_S$    provided that
 $$ \mathtt v_1-3\mathtt v_2+\mathtt v_3+\mathtt v_4=0 \,, \quad \left\{ \begin{aligned} 2(k_1, \mathtt v_2-\mathtt v_1)= |\mathtt v_2-\mathtt v_1|^2 + |\mathtt v_2|^2-|\mathtt v_1|^2 \\
 2(k_1, \mathtt v_4-\mathtt v_2)= |\mathtt v_4-\mathtt v_2|^2 +|\mathtt v_4|^2 -|\mathtt v_2|^2\end{aligned}\right.$$ 
 By hypothesis 
 $ \pi(g)=0$,  so that we have $ \alpha g\cdot k_1=k_1$ for all integer $\alpha$. This implies that the connected component of $\widetilde C_{k_1}$ has infinitely many vertices:
  $$\xymatrix{ \cdots  0\ar@{->}[r]^{e_1-e_2}& (  e_1- e_2)  \ar@{->}[r]^{e_3-e_2}&  ( e_1-2e_2+e_3)\ar@{->}[r]^{e_4-e_2}&   ( g) \ar@{->}[d]^{e_1-e_2}\\ \cdots \ar@{<-}[r]^{e_1-e_2}& (2g)  \ar@{<-}[r]^{e_4-e_2\qquad}&  ( g+ e_1-2e_2+e_3)\ar@{<-}[r]^{\quad e_3-e_2}&  (g+e_1-e_2) } $$ 
To avoid this pathology we simply require that $\mathtt v_1-3\mathtt v_2+\mathtt v_3+\mathtt v_4\neq 0$ so that this geometric graph does not have a realization.\smallskip

Of course since $m>n$ in general  we cannot impose that the $\mathtt v_i$ are linearly independent. So we need to show that imposing a finite number of constraints of linear independence plus other non linear constraints we can assume that all geometric  components satisfy these linear constraints.
\end{example}  
\begin{example}\label{ecv1}[Case 2] Suppose that the geometric graph contains a graph
 $$ \xymatrix{  &\ar@{->}[dl]_{\mathtt v_2-\mathtt v_4}k_3\ar@{=}[dr]^{ \mathtt v_2+\mathtt v_3\quad}& \\ k_1\ar@{->}[rr]^{ \mathtt v_2-\mathtt v_1\quad}& & k_2  },$$ which is the case provided that
 $$k_2+k_3=k_1+\mathtt v_2-\mathtt v_1+k_1+\mathtt v_4-\mathtt v_2= \mathtt v_2+\mathtt v_3 ,  $$ $$ 2k_1= \mathtt v_1+\mathtt v_2+\mathtt v_3-\mathtt v_4 \,, \quad \left\{ \begin{aligned} 2(k_1, \mathtt v_2-\mathtt v_1)= |\mathtt v_2-\mathtt v_1|^2 + |\mathtt v_2|^2-|\mathtt v_1|^2 \\
 2(k_1, \mathtt v_4-\mathtt v_2)= |\mathtt v_4-\mathtt v_2|^2 +|\mathtt v_4|^2 -|\mathtt v_2|^2\end{aligned}\right.$$
We substitute $2k_1$ in one of the linear equations and obtain that this geometric graph does not have realization if
$$  (\mathtt v_1+\mathtt v_2+\mathtt v_3-\mathtt v_4, \mathtt v_4-\mathtt v_2)\neq  |\mathtt v_4-\mathtt v_2|^2 +|\mathtt v_4|^2 -|\mathtt v_2|^2.$$ 
\end{example}

To repeat this reasonings in the general case we need the following trivial fact:
 \begin{lemma}\label{fco}
If $a=\sum_in_ie_i\in{\Z^m}$ resp. $(a,\tau)$ is a product of $d$  elements in $X $  we have that $\sum_i|n_i|\leq 2d  $.
\end{lemma}

It should be clear at this point that in order to {\em lift} the components of  $\Gama$ with at most $d$ vertices (part of Theorem \ref{noinf}) we must impose as many linear/quadratic inequalities on $S$ as the number of circuits which may appear in a component. Thus if we wish to impose only a finite number of constraints we cannot lift arbitrarily large components. 

Our strategy is the following: first we fix $d= 2n+2$ and impose constraints to ensure that all components with at most $d$ vertices can be lifted. Then we show, in \S \ref{detvar},  that there are no compatible graphs in  $\Gama$ with $d$ or more vertices. 

  This finally implies that the mapping $-\pi$ gives an isomorphism from each connected component of $\Lambda_S$ to its image in  $\Gama$.

\bigskip

By Constraint \ref{c3}   
  $\sum_i\ell_i\mathtt v_i\neq 0$, for all choices of the $\ell_i$ such that $\sum_i\ell_i=0,\ \sum_i|\ell_i|\leq 4(n+1)$ and  $\sum_i\ell_ie_i\neq 0$.

 \begin{proposition}\label{main1} Assume that the component $C_x$ of the geometric graph $\Gama$  containing $x=p\cdot 0$   has $d\leq 2n+2$ vertices.  Then the mapping     $\gamma:g\to g\cdot x$ from  $\widetilde C_x$  to $C_x$ is an isomorphism under Constraint \ref{c3} and the next Constraint \ref{c4}.\end{proposition} 
 \begin{proof}
By Lemma \ref{lift} we need to show that the map is injective.   We first construct a map $\lambda:C_x\to \widetilde C_x$ so that $\gamma\circ \lambda =1$  the identity of $C_x$.\smallskip

Take a vertex $q\in C_x$ and a simple path  from $x$  to  $q$  (Def. \ref{term}) which thus has $\leq 2n+2$  steps. By Formula \eqref{gcom}  $q=g\cdot x$  with $g$ a product  of $\leq 2n+2$ edges and by Lemma \ref{fco}  $g=c,  \ (c,\tau),\ c:=\sum_in_ie_i\in{\Z^m}$    we have that $\sum_i|n_i|\leq 2n+2  $.

If $q=h\cdot x$  is reached by a different  path we have $h=b,  \ (b,\tau),\ b:=\sum_ir_ie_i\in{\Z^m}$    we have that $\sum_i|r_i|\leq 2n+2 . $
 
 We thus have $x=h^{-1} gx$   with $h^{-1} g=a,$ or $ h^{-1} g=a\tau,\  a:=\sum_in_ie_i\in{\Z^m}$ and  $\sum_i|n_i|\leq 4(n+1) . $ 
 
 We need to prove that $h=g$  that is $h^{-1} g=1$.\smallskip

 If $h^{-1} g=a$ is black,   then  $x=h^{-1} gx$  implies  that  $\pi(a)=\sum_in_i\mathtt v_i= 0$. By   Constraint \ref{c3}   $a=0$  and $h=g$.  

\smallskip

So assume that  $h^{-1} g=a\tau$ is red,  case 2,   $  \ a=\sum_is_ie_i,\ \sum_is_i=-2$. 
 $ h^{-1} gx=x$  if and only if, by Formula \eqref{lazz} $\pi(a)=-\sum_is_i\mathtt v_i=2x$.\smallskip

This implies that $x=-1/2\sum_in_i\mathtt v_i$ satisfies a  relation of type   \eqref{sfera}
\begin{equation}\label{pririn}  |\sum_h n_h\mathtt v_h|^2+2(\sum_h n_h\mathtt v_h,\mathtt v_i+\mathtt v_j)=-4 (  \mathtt v_i,\mathtt v_j).
\end{equation}  Let us first see what happens   if this formula  vanishes identically  as polynomial in the $\mathtt v_i$.  \smallskip

Then $n_h=0$ for $ h\neq i,j$  and so   $x=-1/2 (n_i\mathtt v_i+n_j\mathtt v_j)$ and
$$n_i^2=-2n_i, \ n_j^2=-2n_j\implies n_i,n_j=0,-2  .$$  If $n_i=n_j=-2$  we have 
$$4(\mathtt v_i+\mathtt v_j)^2 -4(\mathtt v_i+\mathtt v_j)^2=-4 (  \mathtt v_i,\mathtt v_j)$$ which implies $(  \mathtt v_i,\mathtt v_j)=0$ which we have excluded in Constraint \ref{c1} otherwise $x=\mathtt v_i,\mathtt v_j\in S$  contrary to our choice of a component outside the special one.\smallskip

Therefore we can  impose as constraints:
 \begin{constraint}\label{c4}
We assume that   for all choices of the $n_i$ such that $\sum_in_i=-2,\  \sum_i|n_i|\leq 4(n+1)$ all   equations \eqref{pririn} are not satisfied. 

\end{constraint}

  Thus under these new finitely many constraints we have a canonical lift of $C_x$  inside $G_2^x$. We need to prove that it is surjective to the connected component of 0.  
  
  By induction it is enough to prove that,  a vertex $v\in \widetilde C_x$ connected   by an edge $\ell$ to a vertex   $u=\lambda(q),\ q=\gamma(u)$ is  in the image of $\lambda$.\smallskip
  
   Since the map $\gamma:g\to g\cdot x$ is a morphism of graphs $t:=\gamma(v)$ is connected to  $\gamma(u)=q$  by the same  edge $\ell$.  Consider the path from $x$ to  $t$  which passes first through $q$  and then the edge $\ell$.  If this is a simple path  then by the previous discussion $v=\lambda(t)$  is the lift of $t$.
  
  Otherwise we have a simple path from $x$  to $t$ and then from $t$ to $q$ by $\ell^{-1}$ giving a simple path from $x$ to $q$. Then 
  $$v=\ell^{-1} u= \ell^{-1} \lambda(q)=  \lambda(t).$$\end{proof}
 
   \section{The equations defining a connected subgraph of $\Gama$.\label{eqgama}}  
   
   Take   a  connected subgraph $A$ of $\Gama$ which can be lifted (in particular this will be the case if $A$ has at most $2n+2$ vertices by the previous constraints).
   
    Choose a root $x\in A $, we lift $x=-\pi(a),\ a\in\R^m$, this  lifts $A$ to the component $\mathcal A_{ a }$ through $a$ in $\Lambda_S$. 
    
    
    For  each $h\in A$ we have an element $g_h \in G$  obtained by lifting a path in $A$ from $x$ to $h$ and such that $h=g_h x$. We set   
   \begin{equation}\label{lasig}
g_h :=( L(h),\sigma(h)),\quad L(h)\in\Z^m,\ \sigma(h)\in\{1,\tau\}\implies h=-\pi(L(h))+\sigma(h)x.
\end{equation}

We then can deduce the defining equations that is:
\begin{theorem}
 For each $h\in A$  we  have:
\begin{equation}\label{bacos}
\begin{cases}
  (x,\pi( L(h)))=  K(g_h )\quad \text{if}\  \sigma(h)=1\\
 |x|^2+(x,\pi( L(h)))= K(g_h )\quad \text{if}\  \sigma(h)=\tau\end{cases}
.\end{equation} 
 \end{theorem}
 \begin{proof} By Theorem \ref{fo} $K(g_h a)=K(a)$  for all $h$ and $x=-\pi(a)$. Formula \eqref{bacos} follows then 
from  Formula \eqref{vee1}.

 To be explicit if $L(h)=\sum_im_ie_i$ by \eqref{ricon}:
\begin{equation}  
\pi(g_h)= \sum_im_i\mathtt v_i,\quad  K(g_h )=   \sigma(h)\frac12 (|\sum_im_i\mathtt v_i|^2+\sum_im_i|\mathtt v_i|^2).
\end{equation} \end{proof} 
Observe that 
$$|x|^2+(x,\pi(g_h))=|x+\frac{\pi(g_h)}2|^2-|\frac{\pi(g_h)}2|^2=|x-\frac{\sum_im_i\mathtt v_i}2|^2- \frac{|\sum_im_i\mathtt v_i |^2}4 $$ so the equation becomes
\begin{equation}\label{cicc}
|2x- \sum_im_i\mathtt v_i |^2=-  |\sum_im_i\mathtt v_i |^2  - 2  \sum_im_i|\mathtt v_i|^2 =-\sum_i(m_i^2+2m_i)|\mathtt v_i|^2-\sum_{i<j}2m_im_j (\mathtt v_i,\mathtt v_j).
\end{equation}Observe that these equations do not depend upon the choice of  $a$ with  $x=-\pi(a),\ a\in\R^m$.  
   We think of this system of equations as   associated to the graph.
\begin{proposition}\label{coset}
The equations on $x$ given in Formula \eqref{bacos} are a complete set of conditions  for the existence of a graph $A$  inside some connected component (which could also properly contain $A$) of $\Gama$ containing the point $x$. 

\end{proposition}
The reader should notice that these equations are completely analogous to the ones of Definition \ref{lasfe}, given only for edges. Using the notations of Formula \eqref{orbG}  we set:
 \begin{definition}\label{cogrA}
Let $\GA\subset G_X\subset \Z^m_2$ be the  graph with  vertices the elements $ L(h) $ (and $  0 $), this is called the {\em combinatorial graph} associated to $A$ and the {\em root} $x$. \end{definition} 
 
\begin{remark}
Notice that the map which associates to each $h\in A$ the element $g_h=(L(h),\sigma(h))$ is well defined only if $A$ can be lifted.  \end{remark}

\begin{definition}
We call the set of complete  subgraphs of $G_X$ which contain $0$ and have at most $2n+2$ vertices the set of {\em possible combinatorial graphs}.  
\smallskip

We say that a possible combinatorial graph $\GA$ has a geometric realization (in $\Gama$) if  the equations \eqref{bacos}, associated to the graph have real solutions outside $S$. 
\end{definition} 
 \begin{remark}  First of all there are only finitely many possible combinatorial graphs.

Notice that in a possible combinatorial graph one may deduce the color of each vertex by computing its mass. Indeed all vertices $(a,+)$ must have $\eta(a)=0$ while $(a,-)$ corresponds to $\eta(a)=-2$.
\end{remark}
We have reduced our problem to that of understanding which possible combinatorial graphs have a geometric realization. 

For given $S$ and graph $\GA$  this amounts to checking wether the equations associated to the graph have a real solution outside the special component. 

 \begin{remark}
It should be clear that if $\GA$ has a geometric realization then so has any other equivalent possible combinatorial graph. Moreover the two identify the same subgraphs of $\Gama$ with a different choice of the root.
\end{remark}
 \section{Relations}
 \subsection{Basic definitions}\label{noveuno}
  We want to study the geometric realizations of a combinatorial graph $\GA\subset G_2$ in dimension exactly $n$ depending on the choices of the tangential sites $S$. 
  
  By definition $0\in \GA$  will be also called the root.

To $\GA$  are associated the equations \eqref{bacos}  for $x\in\R^n$ which express  the conditions that $\GA$  has a geometric realization with root $x$. 
\begin{definition}\label{rev}
We call  the set $R_{\GA}$ of points $(x,\mathtt v_1,\dots,\mathtt v_m)\in \R^{(m+1)n}$ which satisfy all the equations \eqref{bacos} associated to $\GA$ the variety    {\em of realizations of the graph}. \smallskip

 Call $\theta:R_\GA\to \R^{mn}$ the projection map $(x,\mathtt v_1,\dots,\mathtt v_m) \to (\mathtt v_1,\dots,\mathtt v_m)$. \smallskip

We say that the graph $R_{\GA}$ has no generic realization if  $\theta(R_{\GA})$ is contained in a proper subvariety, that is there is a non zero polynomial $f(\mathtt v_1,\cdots,\mathtt v_m)$  in the coordinates of the vectors $\mathtt v_i$  which vanishes on $R_{\GA}$.  The polynomial $f$ is also called an {\em avoidable resonance}.
\end{definition}
Our strategy is to describe all combinatorial graphs $\GA$ which  have an   avoidable resonance  $f_{\GA}(\mathtt v_1,\cdots,\mathtt v_m)$.  We then impose all these avoidable resonances as constraints.  

As a result we have that all these combinatorial graphs will not appear in $\Gama$  for $S$  generic, that is $f_{\GA}(\mathtt v_1,\cdots,\mathtt v_m)\neq 0.$\smallskip

As in Formula \ref{orbG}  we identify the vertices of a combinatorial graph with integer vectors $a_i$ with $\eta(a_i)\in\{0,-2\}$. We always refer at the vertices different from the root 0.
 \begin{definition}\label{ranghi}\strut
 
 \begin{itemize}
 \item If  $\GA$ has $k $ vertices plus the root 0, it is  said to be of {\em dimension} $k$.

\item  The dimension of the  lattice    generated by the vertices of $\GA$ is
the {\em rank}, ${\rm rk}\, \GA$,  of the graph $\GA$. The dimension of the lattice generated by the black vertices $(a,+)$ (resp. red)  is called the black (resp. red) rank of $\GA$.
\item If the rank of $\GA$ is strictly less than the dimension of $\GA$ we say that $\GA$ {\em is degenerate}.
\end{itemize}
\end{definition} Our main Theorem  \ref{main}  then follows from the following  basic but quite technical
\begin{proposition}\label{LaMa}
For each dimension $n$ for a generic choice of the set $S$  no degenerate graph appears in $\Gama$.  

Moreover the only non degenerate graphs which  appear  in $\Gama$ have rank $\leq n$.
\end{proposition}
The proof of this Proposition  will take the rest of this paper. 

By Proposition   \ref{LaMa} only  non degenerate graphs which   have rank $\leq n$ may appear  in $\Gama$, then  Theorem \ref{main} follows from this and the following.
\begin{proposition}\label{avri}
In a graph with $\leq n$ linearly independent vertices (plus the root) a generic realization is formed by affinely independent points,
\end{proposition}
\begin{proof} Take one of these graphs  which has  $h+k\leq n$ linearly independent vertices $$a_1,\cdots,a_h, b_1,\cdots b_k\quad\text{with}\quad \eta(a_i)=0,\ \eta(b_i)=-2.$$  Given a  geometric realization of this graph, starting from a root $x$  the remaining vertices are  $$a_i\cdot x=-\pi(a_i)+x,\quad  b_i\cdot x=-\pi(b_i)-x=(-\pi(b_i)-2x)+x  .$$

   We need to prove that, for generic choices of $\mathtt v_i$,  the vectors $-\pi(a_i),\ -\pi(b_i)-2x$ are linearly independent.  This means that  some determinant  of  a maximal minor of the $h+k\times n$ matrix with columns these elements is different from 0. 
   
   Change the basis of  $\R^m$  from $e_i$ to some $f_j$ so that the first $h+k$ elements $f_j$ coincide with  $a_1,\cdots,a_h, b_1,\cdots b_k$.\smallskip
   
   If there are no $b_i$  then this matrix  can be thought of as a matrix of variables so all  determinants of maximal minors are polynomials  different from 0 in the $\mathtt v_i$ and can thus be imposed in the list of avoidable resonances, that is inequalities satisfied by the $\mathtt v_i$.  
   
   If on the other hand we have some $b_i$  the determinants of maximal minors are polynomials which, besides the coordinates of the $\mathtt v_i$ contain also the coordinates $x_i$ of the root.  So we need to approach the problem in a different way.\smallskip

   Let $c_i=b_i-b_1,\ i=2,\cdots, k$. The list of vectors to be proved to be linearly independent is now
\begin{equation}\label{ivve}
u_1=\pi(a_1),\cdots,  u_h=\pi(a_h),\ u_{h+2}=\pi(c_2),\cdots,u_{h+k}=\pi(c_k) ,\ \pi(b_1)+2x=u_{h+1}+2x.
\end{equation} The coordinates of the $h+k$ vectors  $u_i  $ are algebraically independent linear functions in the coordinates of the $\mathtt v_i$ so we can treat them as independent variables.  

Consider the matrix $A$ of scalar products  of the  $h+k$ vectors   of Formula  \eqref{ivve}.  In particular compute 
$$ (u_i,\pi(b_1)+2x)=(u_i,\pi(b_1))+2(u_i,x),\ i\neq h=1, $$$$\ (\pi(b_1)+2x,\pi(b_1)+2x)=  (\pi(b_1) ,\pi(b_1) )+4(\pi(b_1),x)+4|x|^2.$$
From equations \eqref{bacos} and the definition of $K(a)$  which is a quadratic polynomial  in the coordinates of $\mathtt v_i$ with integer coefficients  the terms  $(u_i,x)$  and $(\pi(b_1),x)+ |x|^2$ are quadratic polynomials  in the coordinates of $\mathtt v_i$ with integer coefficients.  Making this substitution we have that the matrix $A$ of scalar products has entries  quadratic polynomials in the coordinates of the vectors $\mathtt v_i$ with integer coefficients.

 If the determinant $\Delta$ of $A$  is a non zero polynomial  we take this as avoidable resonance so under these constraints the vectors   are linearly independent and thus the points of the corresponding component are affinely independent.
 
 In order to prove that $\Delta$  is non zero we can   pass modulo 2  where the terms deduced by substitution of  $2(u_i,x),\  4(\pi(b_1),x)+4|x|^2 $ disappear.  $\Delta$   becomes the determinant of scalar products  of  the vectors $u_i$ with indeterminate coefficients  which is clearly different from 0 and we impose it as avoidable resonance.  

  \end{proof}
Take a connected component $A$ of $\Gama$ and choose  a root $ x \in A.$  Assume that $A$ can be lifted. Let $\GA=\{g_a ,\ a\in A\}$ be the combinatorial  graph of which $A$ is a geometric realization. 
\begin{lemma}
The rank of $\GA$ does not depend on the choice of the root but only on $A$. 
\end{lemma}
\begin{proof}  We can stress the role of the root in the notation $g_{a,x } =(L_x(a),\sigma_x(a))$. 

We change the root from $x$ to another $y=g_{y,x}x$,  and  have $a=g_{a,x}x=g_{a,x}g_{y,x}^{-1}y$. 

 Then   $g_{y,x }^{-1} =(-\sigma_x(y)L_x(y),\sigma_x(y))$ and
 $$(L_x(a),\sigma_x(a) )=g_{ a,x}=g_{ a,y}g_{y,x}^{-1},=(L_y(a),\sigma_y(a))(L_x(y),\sigma_x(y)),$$
\begin{equation}
\label{chro} \implies L_x(a)=L_y(a)+\sigma_y(a)L_x(y),\  \sigma_x(a)=\sigma_y(a)\sigma_x(y).  
\end{equation}   In particular $L_y(x)=- \sigma_x(y)L_x(y)$. This shows that   the notion of rank is independent of the root.\end{proof}

Notice that when we change the root in $A$ we have a simple way of changing the colors and the ranks of the vertices of $\GA$ that we leave to the reader.

\subsection{Degenerate graphs}

If $\GA$ is a degenerate  graph then there are non trivial relations,   $\sum_an_a a=0,\ n_a\in\Z$  where the sum runs among the vertices $a\in \GA$ different from 0.
\begin{remark}
\label{maxt} It is  useful to choose a maximal tree $T$ in $\GA$. 

This is a tree which contains all vertices of $\GA$. For each choice of $T$ there is   a triangular change of coordinates from the vertices to the    edges  of $T$. Hence the relation can be also expressed as a relation between these   edges.
\end{remark}
 
In the next discussion we treat the $\mathtt v_i$ as  {\em vector  variables} and we seek solutions of our equations  as functions of the $\mathtt v_i$.

 We must have, by linearity of the map $a\mapsto  a^{(2)}$, for every relation  $\sum_in_i a_i=0,\ n_i\in\Z$ that $0=\sum_in_i a_i^{(2)},$ where we recall that if $a=\sum m_i e_i$ we have that $a^{(2)}= \sum m_i e_i^2$. \smallskip
 
  Finally we have $ \ 0=\sum_in_i \pi(a_i)$ as linear polynomial in the $\mathtt v_i$  and $\sum_in_i \eta(a_i)=0$. \medskip
 
 Recalling that $\eta(a)=0,-2$ (resp. if $a$ is black or red), we have :\begin{equation}
0=\sum_{i\, |\, \eta(a_i)=-2}n_i.
\end{equation}
Applying Formula \eqref{bacos}  we deduce that, in order to ensure that the equations of $\GA$ are compatible,  we must have
\begin{equation}\label{riso}
\sum_in_iK(a_i )=\sum_in_i (x,\pi(a_i))+[\sum_{i \, |\, \eta(a_i)=-2}n_i]|x|^2  =(x,\sum_in_i \pi(a_i))=0.
\end{equation}   
  \begin{lemma}\label{CK}
If $\sum_in_iC(a_i) $ is non zero then $\sum_in_iK(a_i ) =\pi(\sum_in_iC(a_i) )$ is a non zero polynomial  in the coordinates of the vectors $\mathtt v_i$ for all dimensions $n$.
\end{lemma}
 \begin{proof}It is clear that it is enough to prove this for $n=1$, by specializing the $\mathtt v_i$  to vectors in which only the first coordinate is not zero.
 
The expression $\sum_in_i {K(a_i)} =\pi(\sum_in_iC(a_i) )$ is a linear combination with integer coefficients of the scalar products $(\mathtt v_i,\mathtt v_j)$.
In dimension $n=1$  we have that the $\mathtt v_i$  are   variables and $(\mathtt v_i,\mathtt v_j)= \mathtt v_i \mathtt v_j$,   so in practice this is just a variable substitution $e_i\mapsto \mathtt v_i$.
 \end{proof}

 Let  $\GA$ be a   combinatorial graph   $\GA$ with a relation $\sum_an_a a= 0$: 
  \begin{lemma}\label{avr}
 If $\sum_an_aC(a) \neq 0$ the graph $\GA$    has no geometric realization for a generic choice of the $S:=\{\mathtt v_i\}$.
\end{lemma}
\begin{proof}
If the graph has a realization then  $\sum_in_iK(a_i )=0$ but this polynomial  is not identically zero by Lemma \ref{CK}, so we can impose it as one of the  constraints on $S$.\end{proof}
As alreadhy explained we restrict to impose these conditions to graphs with at most $2n+2$ vertices, so we have a finite number of constraints.
\begin{example}\label{pirir}
Consider the   degenerate combinatorial graph
 $$ 
 \GA= \xymatrix{e_1-e_2 \ar@{<-}[r]^{e_1-e_2}& 0     \ar@{=}[d]_{-e_1-e_2}
  \ar@{=}[r]^{-e_1-e_3}
  &
  {-e_1-e_3}\ar@{->}[r]^{e_1-e_3} & {-2e_3} \\&{-e_1-e_2}&& } 
 $$
The relation is $  (e_1-e_2)+2(-e_1-e_3)-(-2e_3)-(-e_1-e_2)=0  $.

We may write the value of $C(a)$
 of each vertex $a$, we get
 $$  \xymatrix{e_1^2-e_1e_2 \ar@{-}[r]& 0     \ar@{=}[d]\ar@{=}[r]&{-e_1e_3}\ar@{-}[r] & {-e_3^2} \\&{-e_1e_2}&& }$$
we have
$$ \sum_a n_a C(a)=
\qquad e_1^2-e_1e_2-2e_1e_3+e_3^2+e_1e_2 \neq 0
$$
so the equations of this graph are incompatible if $\pi( e_1^2  -2e_1e_3+e_3^2 )=|\mathtt v_1-\mathtt v_3|^2\neq 0$. This is an avoidable resonance.
\end{example}

We arrive now at the main Theorem of the section:
\begin{theorem}\label{ridma}
Given  a possible combinatorial graph  of rank $k$ for a  given color, then either it has  exactly $k $  vertices of that color or it produces an avoidable resonance.  \end{theorem}
\begin{proof}    
Assume   that  we can choose   $k+1$ vertices $(a_0,a_1,\ldots,a_k) $, different from the root of the given color $\sigma=\pm 1$  so that we have a non trivial relation   $\sum_in_i  a_i =0 $ with $n_0\neq 0$ and the vertices $a_i,\ i=1,\ldots,k$ are linearly independent.  
We compute the resonance relation and need to show that it is different from 0:
$$2\sigma  \sum_in_i  C(a_i) = \sum_in_i  (a_i^2+ a_i^{(2)}).$$ 
  By the linearity of  the map $a\mapsto a^{(2)}$   we have $\sum_in_i  a_i =0\implies  \sum_in_i a_i^{(2)}=0$.
  
We deduce that  $$2\sigma \sum_in_i  C(a_i)= \sum_in_ia^2_i=n_0a_0^2+ \sum_{i=1}^nn_ia^2_i.$$  Now from $ n_0  a_0 =-( \sum_{i=1}^nn_ia_i ) $ we deduce
$$n_0^2a_0^2=( \sum_{i=1}^nn_ia_i )^2 \implies n_0^2a_0^2+ n_0\sum_{i=1}^nn_ia^2_i=( \sum_{i=1}^nn_ia_i )^2+ n_0\sum_{i=1}^nn_ia^2_i .$$Since the elements $a_i$ with $i=1,\ldots,k$ are  linearly independent they can be treated as {\em independent variables}. If  this expression is 0,  we  have that only one of the coefficients $n_i$ can be different from 0, say $n_1\neq 0$  so, if $ \sum_in_i  C(a_i) = 0$  the relations are $$n_0a_0+n_1a_1=0=n_0a_0^2
+n_1a_1^2\implies n_0^2a_0^2
+n_0n_1a_1^2=(n_1^2+n_0n_1 )a_1^2=0\implies  a_0 = a_1 $$ a contradiction.   Therefore $ \sum_in_i  C(a_i) \neq 0$.

   \end{proof}

 \begin{constraint}\label{co6}  We impose that the  vectors $\mathtt v_i$ are generic for  avoidable resonances arising from degenerate possible combinatorial graphs with  at  most $n+1$  elements of each color.
\end{constraint} There are finitely many  degenerate possible combinatorial graphs with  at  most $n+1$  elements of each color. For each one of these it is enough to choose a single relation giving an avoidable resonance. Thus this constraint is given by a finite number of inequalities.
\begin{remark}\label{bala}
 It is essential that we introduce the notion of coloured rank, otherwise our statement is false as can be seen with the following graph:
\begin{equation}
\label{mig} \xymatrix{ & \ar@{-}[d]_{ -e_2+e_1}\ar@{=}[r] ^{ -e_2-e_1}(   -e_2+e_1)&  ( -2e_1 )&\\ e_1- e_3&  \ar@{-}[l]^{e_1-e_3 }\ar@{=}[r]0&   ( -e_3-e_1 )&  }
\end{equation}   
Relation is $   ( -e_3-e_1 )-(  e_1-e_3 ) -( -2e_1 )=0$, we have 
$$ C( -e_3-e_1 ) =-e_1e_3,\quad  C(   -e_3+e_1)=e_1^2-e_1e_3,\quad C( -2e_1 )=-e_1^2$$ $$-e_1e_3-  (e_1^2- e_1e_3)+e_1^2=0.$$
Actually this graph does not really pose any problem since its only geometric realization is in $S$ (hence it is {\bf not} a true combinatorial graph).

A more complex example is

 $$ \xymatrix{&&&e_2-e_3\ar@{-}[d]^{ e_2-e_3}&\\
 -3e_1+e_2\ar@{=}[r]^{-e_1-e_4}&2e_1-e_2-e_4\ar@{-}[r]^{ e_1-e_4}&e_1-e_2\ar@{-}[r]^{e_1-e_2}&0\ar@{= }[r]^{-e_2-e_3}&-e_2-e_3}.$$
What is common of these two examples is that in each there is a pair of vertices $a,b$, not necessarily joined by an edge, of distinct colors,  with $a+b=-2e_i$ for some index $i$. In both cases  by changing root if necessary we have a vertex equal to $-2e_i$  or in group notation $-2e_i\tau.$ 
 
\begin{definition}\label{ilpunto1}
We shall say that a connected graph  $G$ is {\em allowable}  if there is no vertex $b=-2e_i, -3e_i+e_j$, otherwise it is {\em not allowable}.
\end{definition}
We may assume $a\in\Z^m$ black and $c=b\tau,\ b\in\Z^m$ red. We then easily see that 
\begin{proposition}\label{ilpunto0}
If a graph is not allowable then it has no geometric realization outside the special component (i.e. it is not compatible).\end{proposition}
\begin{proof} 
 The quadratic equation \eqref{bacos}, for a vertex $x$, corresponding to a red vertex $b $  
   can be written as   \begin{equation}\label{bak}
|x-\frac {\pi(b)}2|^2=-\frac 14 |\pi(b)|^2+K(b)=-\frac 14 |\pi(b)|^2 -\frac 12 |\pi(b)|^2 -\frac 12  \pi(b^{(2)}) = -\frac 14( 3|\pi(b)|^2+2\pi(b^{(2)})). 
\end{equation} In   case of a vertex $-3e_i+e_j$,    
$$3b^2+2b^{(2)}=3(-3e_i+e_j)^2 +2{ (-3e_i^2+e_j^2) } ] $$
  $$ = 27 e_i^2-18e_ie_j+3e_j^2-6e_i^2+2e_j^2 = 21  e_i^2-18e_ie_j +5e_j^2$$
 The symmetric matrix
 $$ X=\begin{vmatrix}21&-9\\
-9&5
\end{vmatrix},\quad \det X=24$$ is positive definite  so \eqref{bak} has no real solutions.\medskip

For the vertex $b=-2e_i$. Since $C(-2e_i)=-e_i^2,\ K(-2e_i)=-|\mathtt v_i|^2$      we get $$0= |x|^2+(x,\pi(-2e_i))-K(-2e_i)=|x|^2-2(x, \mathtt v_i )+|\mathtt v_i|^2= |x-\mathtt v_i|^2.$$  Hence the only real solution of $|x-\mathtt v_i|^2=0$ is $x=\mathtt v_i$. Then we apply  Remark \ref{laspe}  where we have shown that the  special component is an isolated component of the graph.\end{proof}

The fact that we can  exclude the existence of more complicated graphs of this form which may have realization in $S^c$ is quite difficult and will take the last part of this paper.
\end{remark}

  \section{Geometric realization \label{Magt}}
  {\em We now justify why, in dimension $n$,  we can impose our  constraints only to graphs with at most $2n+2$  vertices.}
 \subsection{The polynomial realizations} \subsection{Determinantal relations} 

1)\quad Given a combinatorial graph $\GA$ with $n$ linearly independent black vertices  $a_1  ,\cdots, a_n,\  a_i=\sum_{j=1}^m a_{i,j} e_j$   consider the $n$  vector valued linear functions  $\pi(a_i)= \sum_{j=1}^m a_{i,j} \mathtt v_j  ,\ \mathtt v_j\in\R^n$.  The $n$ coordinates of the functions $\pi(a_i)$  can be taken as the columns  of an $n\times n$ matrix  $A(\mathtt v)$  with entries linear functions in the coordinates of the vectors $\mathtt v_i$ which we are considering as independent variables, that is coordinates for the $mn$ dimensional vector space of $m$ tuples of $n$ dimensional vectors $\mathtt v_i$.

Since  the $a_i$ are linearly independent so are the  columns of the matrix $A(\mathtt v)$ (as functions)  and the determinant $d=\det A(\mathtt v)$ is a non zero polynomial in these entries.  We can thus impose  $ \det A(\mathtt v)\neq 0$ as a constraint.

In fact $d$ is a linear combination of the determinants of  the matrices with the columns $n$ of the various $\mathtt v_i$.\smallskip

We then solve the $n$ linear equations, out of the list \eqref{bacos} corresponding to the vertices $a_i$  by Cramer's rule. We  thus obtain the vector 
$x$ as a vector of rational functions $ x_i= {u_i}/ d$.  

\begin{remark}\label{gere} We substitute this vector of functions in  the remaining equations    \eqref{bacos}. 

 If under this substitution all other   equations vanish then   we call $x$ the {\em generic realization} of the graph $\GA$.
 In this case once  we specialize the $\mathtt v_i$ to vectors in $\R^n$ outside the hypersurface given by $d=0$ we have that $\GA$  has a unique  geometric realization obtained by specializing the generic one.\smallskip

If   the graph $\GA$ does not have  a  generic realization this means that at least one of the equations in \eqref{bacos}  with $x$  substituted as before is a non zero rational function  $u/d^2$  in the coordinates of the $\mathtt v_i$  with denominator $d$ or $d^2$.  When  we specialize the $\mathtt v_i$ to vectors in $\R^n$ outside the hypersurfaces given by $d=0,\ u=0$ then equations \eqref{bacos} are incompatible  and $\GA$  has no  geometric realization. 
\end{remark} 
\begin{constraint}\label{c6}
We  impose as inequalities all the functions $d,u$  or just $d$ arising from this algorithm for all graphs with $\leq 2n+2$ vertices and $n$ linearly independent black vertices.
\end{constraint} \medskip

2)\quad If now  $\GA$ has $n+1$ linearly independent black vertices  $a_1  ,\cdots, a_{n+1},\  a_i=\sum_{j=1}^m a_{i,j} e_j$  we can choose $n$ out of them in $n+1$  ways and we have $n+1$ different determinants $d_i$ and $n+1$ different ways  of writing the generic solution, if it exists,  as $x_i=u_i/d_i$. \smallskip

This on the other hand must be  the same rational function, in other words the system of $n+1$  linear equations out of the list \eqref{bacos}  relative to these vertices in $n$ variables  must be compatible. This is so only if the determinant of  the $n+1\times n+1$ matrix made from the columns  of the system and the constant coefficients is identically 0.

If it is not 0  then it generates an avoidable resonance and $\GA$ has no generic realization.\medskip
\begin{constraint}\label{c7}
We impose as inequality the non vanishing of these $n+1\times n+1$ determinants.
\end{constraint}
3)\quad  Assume now that $\GA$ has $n+1$   linearly independent   vertices  $h$ black and $k>0$    red $$a_1  ,\cdots, a_k, b_1  ,\cdots, b_k,\  a_i=\sum_{j=1}^m a_{i,j} e_j,\ b_i=\sum_{j=1}^m b_{i,j} e_j.$$ Replace the equations  \eqref{bacos}  for $b_i,\ i=1,k-1$  by subtracting the equation for $b_k$.  

We get a system of $n$ linear equations for $x$  which as in the previous case has a unique   generic solution  $x=u/d$.
 
 If this is a generic realization for $\GA$  it must satisfy the equation $|x|^2+(x,\pi(b_k))=K(b_k)$. That is
 $$ |u|^2+d(u,\pi(b_k))=d^2K(b_k).$$
 
 In the next section we shall prove that,      under the  hypotheses 2) or 3), if  the equations are compatible  the generic solution is a polynomial  in the $\mathtt v_i$ and then its generic realization is necessarily in the special component.  This will prove \begin{theorem}\label{aMT}
If $\GA$ is a    combinatorial graph   of rank   $n+1$    which has a realization   for generic  $\mathtt v_i$'s, then   its generic realization is in the special component (the solution $x$ belongs to the set $S$).
\end{theorem}

  \medskip

\section{Determinantal varieties}\label{detvar} Consider the space  $V=\R^n$  and $n$ linear maps  $w_j:(\mathtt v_1,\cdots,\mathtt v_m)\mapsto  \sum_{i=1}^ma_{j,i}\mathtt v_i$     from $V^{\oplus m}$ to $V=\R^n$  given by the $n \times  m$ matrix $A:=(a_{j,i})$. In an equivalent formulation this is a linear map $ \rho :\R^m\otimes V =V^{\oplus m} \to  \R^n\otimes V =V^n$ with Matrix $A\otimes 1$.

 \begin{lemma}
An  $m$--tuple  of vector values functions  $m_i:= \sum_j a_{ij}\mathtt v_j$ is formally linearly independent -- that is the $n\times m$ matrix of the $a_{ij}$ has rank $n$-- if and only if the associated   map  $\rho : V^{\oplus m} \to  V^n$ is surjective. 
\end{lemma} \begin{proof}
$A$ is surjective if and only if $A\otimes 1$     is surjective.
\end{proof}We may identify  $ \R^n\otimes V =V^{\oplus n}$ with $n\times n$ matrices and  we have   the determinantal variety  $D_n$ of  $V^{\oplus n},$ defined by the vanishing of the determinant  $\det$ (an irreducible polynomial),  and formed by all the $n$--tuples of vectors $u_1,\ldots,u_n$ which are linearly dependent. 

 The variety $D_n$ defines a similar determinantal variety $D_\rho:= \rho ^{-1}(D_n)$ in $V^{\oplus m}$, defined by the vanishing of the polynomial   $\det\circ\rho$,   which   depends on  the map $\rho$. This is a proper hypersurface if and only if $\rho$ is surjective otherwise  $\det\circ\rho=0$.
 \begin{lemma}\label{irre}
If $\det\circ\rho\neq 0$ it is an  irreducible polynomial. 
\end{lemma}\begin{proof}
If $\rho$ is surjective, up to a linear coordinate change it can be identified with the projection on the first $n$ summands,  so it is clear that in this case  $D_\rho$ is an irreducible hypersurface  with equation the irreducible polynomial $\det\circ\rho$.\end{proof}  We need to see when different maps give rise to different determinantal  varieties in $V^{\oplus m}$.  
\begin{lemma}\label{tras} Given a surjective map $\rho :V^{\oplus m}\to V^{\oplus n},$
a vector $a\in V^{\oplus m}$ is such that $a+b\in D_\rho,\ \forall b\in D_\rho$ if and only if $\rho(a)=0$.
\end{lemma}
\begin{proof}
Clearly  if $\rho(a)=0$ then $a$ satisfies the condition. Conversely if $\rho(a)\neq 0$, we think of $\rho(a)$ as a non zero matrix  $B$.

If $\det(B)\neq 0$ then  $\rho(a)+0\notin D_\rho$. Otherwise $B$ has rank $0<h<n$ and there is an other matrix $C$  of rank $n-h$ so that $\det(B+C)\neq 0$. Then there is a $b$ so that  $C=\rho(b) \in D_n$  and  $B+C=\rho(a+b)\notin D_n$.\end{proof}
Let $\rho_1,\rho_2:V^{\oplus m}=V\otimes\R^{\oplus m} \to V^{\oplus n}=V\otimes\R^{\oplus m}$ be two surjective maps,  given by $\rho_1=1_V\otimes A,\ \rho_2= 1_V\otimes B$  for  two $n\times m$ matrices $ A=(a_{i,j}),\ B=(b_{i,j});\ a_{i,j},b_{i,j}\in\mathbb C$ .
\begin{proposition}\label{kke}
$\rho_1^{-1}(D_n)=\rho_2^{-1}(D_n)$ if and only if  the two matrices $A,B$ have the same kernel.
\end{proposition}
\begin{proof} The two matrices $A,B$ have the same kernel if and only if $\rho_1,\rho_2 $   have the same kernel.
By Lemma \ref{tras}, if  $\rho_1^{-1}(D_n)=\rho_2^{-1}(D_n)$ then  the two matrices $A,B$ have the same kernel. Conversely if  the two matrices $A,B$ have the same kernel we can write $B=CA$ with $C$ invertible.  Clearly $CD_n=D_n$ and the claim follows.\end{proof}
We shall also need the following well known fact:
\begin{lemma}\label{zade}
Consider the determinantal variety $D,$ given by  $d(X)=0,$ of $n\times n$  complex  matrices of determinant zero. The real points of $D$ are Zariski dense in $D$.\footnote{this means that a polynomial vanishing on the real points of $D$ vanishes also on the complex points.}
\end{lemma}
\begin{proof}
Consider in    $D$   the set of real matrices    of rank exactly $n-1$. This set is    obtained from a  fixed matrix (for instance the diagonal matrix $I_{n-1}$ with all 1 except one 0)  by multiplying $AI_{n-1}B$  with $A,B$ invertible matrices.  If a polynomial $f$  vanishes on the real points of  $D$ then $F(A,B):=f(AI_{n-1} B)$  vanishes  for all $A,B$  invertible matrices and real. This set is the set of  points in $\mathbb R^{2n^2}$  where a polynomial (the product of the two determinants) is non zero. But a polynomial which vanishes in all the  points of any space $\mathbb R^{s}$  where another  polynomial   is non zero is necessarily the zero polynomial. So $f$ vanishes also on complex points. This is the meaning of  Zariski dense.\end{proof}
So let $\GA$ be a graph of rank $\geq n+1$, consider as before the variety  $R_{\GA}$    of realizations of the graph, with its map $\theta:R_{\GA}\to \C^{mn}$. Assume that $\GA$ has  a generic realization, so that    $\theta(R_{\GA})$  is not contained in any real algebraic hypersurface.

  \begin{theorem}\label{codim}  There is an irreducible hypersurface   $W$ of $\C^{mn}$    such that the map $\theta$ has an inverse on $\C^{mn}\setminus W$. The inverse is a polynomial map given by the generic realization.
\end{theorem} 
\begin{proof}
{\it Black vertices} Assume first that we have $n+1$ linearly independent black vertices $a_i$, the functions  $\pi(a_i)$  of the $\mathtt v_i$  are $n+1$ linearly independent linear maps  from $V^{\otimes m}$ to $V$  or in an equivalent formulation this is a linear map $ \rho :\R^m\otimes V =V^{\oplus m} \to  \R^{n+1}\otimes V $ with Matrix $B\otimes 1$, and $B$ an $(n+1)\times m$ matrix of rank  $n+1$..
\smallskip

We have $n+1$ linear equations $ (x,\pi(a_i))= b_i$ which are generically compatible.

 We  solve them by Cramer's rule  choosing an index $j$ and   discarding the equation \eqref{bacos} associated to the vertex $a_{j}$. Since the equations are always compatible we must obtain, generically,   the same solution for all choices of $ j$. Consider  the matrix $M_{j}$ with rows the $\pi(a_i)$, $ i=1,\dots,n+1$ $i\neq j$. The solution  is  a rational function $u_j/d_j$ of the $\mathtt v_i$ having as denominator the determinant $d_j$ of    $M_{j}$.

From Lemma \ref{irre} each   of these determinants is an irreducible polynomial so it defines an irreducible hypersurface $H_j$.

 We claim that these hypersurfaces are   all unequal  so the $d_j$ are all different. In fact  
   the matrices are obtained by $B$ dropping one row define the various determinantal varieties, $H_j$.  These projections have different kernels so  the result follows by Proposition \ref{kke}.
  
  Therefore for two different indices $i\neq j$ we have  $u_i/d_i=u_j/d_j$ with $d_i,d_j$ two  different irreducible polynomials. Then $u_id_j=u_jd_i$ implies that              $d_i$ divides $u_i$ so that $u_i/d_i$ is a polynomial.
 \smallskip

{\it Red vertices}
\smallskip

When we also have red edges we select $n+1$ linear and quadratic equations  associated to the $n+1$ vertices which are formally independent. By subtracting a given quadratic equation  to the others we see that the equations \eqref{bacos}  (for these vertices) are clearly equivalent to a system on $n$ linear equations associated to formally linearly independent vectors  in $\R^m$, plus the given  quadratic equation chosen arbitrarily among the ones appearing in \eqref{bacos}.   

Thus a realization of $\GA$ is obtained by solving the system of  $n$   linear equations   $$\sum_{j=1}^ma_{ij}(x,\mathtt v_j)=(x,t_i)=b_i,\quad i=1,\cdots,n$$  with the $t_i=\sum_{j=1}^ma_{ij} \mathtt v_j$ linearly independent (as functions)    and $b_i$ equals some quadratic expression $\sum_{h,k}b^i_{h,k}(\mathtt v_h,\mathtt v_k) $.\smallskip

We solve these equations  by Cramer's rule    considering the $\mathtt v_i$ as parameters and   obtain   $ x_i=f_i/d $, where  $d :=\det(A(\mathtt v))$ is the determinant of the matrix $A(\mathtt v)$ with rows $t_i$. 

We have thus expressed  the coordinates $x_i$ as rational functions  of the coordinates of the vectors $\mathtt v_i$. The denominator is an irreducible polynomial  vanishing exactly on the determinantal variety of the $\mathtt v_i$ for which the matrix of rows $t_j,\ j=1,\ldots,n$ is degenerate.\smallskip

By hypothesis, this  solution satisfies  a further  quadratic equation in  \eqref{bacos}  identically.

 \begin{lemma}
\label{aMT1} Given $x=(x_1,\ldots,x_n)=(f_1/d,\ldots,f_n/d)$ with  the $f_i$ polynomials in the $\mathtt v_i$ with real coefficients.

 Assume there are two real polynomials  $a , b$  in the $\mathtt v_i$, such that $\sum_ix_i^2+(x,a)+b=0 $ holds identically (in the parameters $\mathtt v_i$);
then $x$ is a polynomial in the $\mathtt v_i$.  \end{lemma}
\begin{proof} Substitute  $x_i=f_i/d$ in the quadratic equation and get 
$$d^{-2}(\sum_if_i^2)+d^{-1}\sum_if_ia_i+b=0,\implies \sum_if_i^2 +d \sum_if_ia_i+d^{ 2} b=0.$$
Since $d=d(v)=\det(A(\mathtt v))$ is irreducible this implies that $d$ divides $\ \sum_if_i^2 $ (in the space of real polynomials).

Since the  $f_i$ are real, for those $v:=(\mathtt v_1,\ldots,\mathtt v_m)\in\R^{mn}$ for which   $d(A(\mathtt v))=0$,  we have $f_i(v)=0,\ \forall i$; so $f_i$ vanishes on all real solutions of $d=   d(A(\mathtt v))=0$.  

These solutions are Zariski dense, by Lemma \ref{zade},  in the determinantal variety  $d =0$. 

In other words  $f_i(v)$ vanishes on all the  $v$ solutions of $d(A(\mathtt v))=0$ and thus $d $  divides $f_i(v)$ for all $i$, hence  $x $ is a polynomial.\end{proof}
This finishes the proof of Theorem \ref{codim}.\end{proof}

Summarizing, we impose  
\begin{constraint}\label{c6b}   For any colored--non--degenerate possible combinatorial graph $\GA$ with at most $2n+2$ vertices (including the root) with red and/or black rank $ n+1$, we impose that the  vectors $\mathtt v_i$ are generic for all resonances described above. That is the determinants we need to invert  are  resonance inequalities.
 \end{constraint}
\begin{example}
We consider the combinatorial graph  in dimension $n=2$. 
\begin{equation}\label{EGA3}  \xymatrix{ & \ar@{-}[d]^{-e_2-e_1}(-e_2-e_1,-)&&\\ (e_1-e_3,+)\ar@{->}[r]^{e_3-e_1\quad}& ( 0,+)  \ar@{->}[r]^{e_3-e_2}&  (  e_3-e_2,+)  }  \end{equation}   
The equations are
\begin{equation}\label{EGB3}
\left\{\begin{array}
{l} (x,\mathtt v_1-\mathtt v_3)= |\mathtt v_1|^2-(\mathtt v_1,\mathtt v_3) \\ (x, \mathtt v_3-\mathtt v_2)= |\mathtt v_3|^2-(\mathtt v_2,\mathtt v_3) \\ |x|^2-(x,\mathtt v_2+\mathtt v_1)= - (\mathtt v_2,\mathtt v_1)
\end{array}\right. 
\end{equation} 

In order to solve the first two equations \eqref{EGB3} by Cramer's rule we impose that the determinant $$d=(v_{1,1}-v_{3,1})(v_{3,2}-v_{2,2})-( v_{1,2}-v_{3,2})(v_{3,1}-v_{2,1}) \neq 0.$$ We obtain the solution $x= (x_1,x_2)$: 
$$ x_1=(|v_1|^2-(v_1,v_3))(v_{3,2}-v_{2,2})-( v_{1,2}-v_{3,2})(|v_2|^2-(v_2,v_3))/d\,,$$ $$x_2=(v_{1,1}-v_{3,1})(|v_2|^2-(v_2,v_3))-(|v_1|^2-(v_1,v_3))(v_{3,1}-v_{2,1})/d .$$ We substitute for $x$ in the last equation, rationalize and obtain that a realization exists only if

$$ \left( (v_1,v_2)-(v_1,v_3)+|v_3|^2-(v_2,v_3) \right)\cdot \,\left( {v_{1,1}}^3\,v_{2,1} + v_{1,1}\,{v_{1,2}}^2\,v_{2,1} + 
    {v_{1,2}}^2\,{v_{2,1}}^2 +\right.$$   $$ \left.{v_{1,1}}^2\,v_{1,2}\,v_{2,2} + {v_{1,2}}^3\,v_{2,2} - 
    2\,v_{1,1}\,v_{1,2}\,v_{2,1}\,v_{2,2} +\right.$$ $$\left. {v_{1,1}}^2\,{v_{2,2}}^2 - {v_{1,1}}^3\,v_{3,1} - 
    v_{1,1}\,{v_{1,2}}^2\,v_{3,1} - 3\,{v_{1,1}}^2\,v_{2,1}\,v_{3,1} - 3\,{v_{1,2}}^2\,v_{2,1}\,v_{3,1} + \right.$$ $$\left.
    2\,v_{1,2}\,v_{2,1}\,v_{2,2}\,v_{3,1} - 2\,v_{1,1}\,{v_{2,2}}^2\,v_{3,1} + 3\,{v_{1,1}}^2\,{v_{3,1}}^2 + 
    2\,{v_{1,2}}^2\,{v_{3,1}}^2 + 3\,v_{1,1}\,v_{2,1}\,{v_{3,1}}^2 - \right.$$ $$\left. v_{1,2}\,v_{2,2}\,{v_{3,1}}^2 + 
    {v_{2,2}}^2\,{v_{3,1}}^2 - 3\,v_{1,1}\,{v_{3,1}}^3 -  v_{2,1}\,{v_{3,1}}^3 + {v_{3,1}}^4 - \right.$$ $$\left.
    {v_{1,1}}^2\,v_{1,2}\,v_{3,2} - {v_{1,2}}^3\,v_{3,2} - 2\,v_{1,2}\,{v_{2,1}}^2\,v_{3,2} - 
    3\,{v_{1,1}}^2\,v_{2,2}\,v_{3,2} - 3\,{v_{1,2}}^2\,v_{2,2}\,v_{3,2} + \right.$$ $$\left. 2\,v_{1,1}\,v_{2,1}\,v_{2,2}\,v_{3,2} + 
    2\,v_{1,1}\,v_{1,2}\,v_{3,1}\,v_{3,2} + 4\,v_{1,2}\,v_{2,1}\,v_{3,1}\,v_{3,2} +  
    4\,v_{1,1}\,v_{2,2}\,v_{3,1}\,v_{3,2} \right.$$ $$\left.- 2\,v_{2,1}\,v_{2,2}\,v_{3,1}\,v_{3,2} - 
    3\,v_{1,2}\,{v_{3,1}}^2\,v_{3,2} - v_{2,2}\,{v_{3,1}}^2\,v_{3,2} + \right.$$ $$\left.2\,{v_{1,1}}^2\,{v_{3,2}}^2 + 
    3\,{v_{1,2}}^2\,{v_{3,2}}^2 - v_{1,1}\,v_{2,1}\,{v_{3,2}}^2 + {v_{2,1}}^2\,{v_{3,2}}^2 + 
    3\,v_{1,2}\,v_{2,2}\,{v_{3,2}}^2 - \right.$$ $$\left. 3\,v_{1,1}\,v_{3,1}\,{v_{3,2}}^2 - v_{2,1}\,v_{3,1}\,{v_{3,2}}^2 + 
    2\,{v_{3,1}}^2\,{v_{3,2}}^2 - 3\,v_{1,2}\,{v_{3,2}}^3 - v_{2,2}\,{v_{3,2}}^3 + {v_{3,2}}^4 \right)=0.$$
    This is one of the resonances we want to avoid.
\end{example}

 We thus have the final definition of generic for tangential sites $S$.
\begin{definition}\label{fincon}
We say that the tangential sites are {\em generic} if they do not  vanish for any of the polynomials given by Constraints \ref{c0},  \ref{c1} through \ref{c6b} applied to combinatorial graphs with at most $2n+2$ vertices.
\end{definition}
  We have ensured that for generic choices of $S$ only those graphs   which are generically realizable are realized.
  \begin{example}
Consider the possible combinatorial graph:
  $$   \xymatrix{ & \ar@{=}[d]^{-e_3-e_4}(-e_3-e_4,-)\ar@{=}[dl]_{-e_1-e_4}\ar@{=}[dr]^{-e_2-e_4}&&\\ (e_3-e_1,+)\ar@{<-}[r]^{e_3-e_1\quad}& ( 0,+)  \ar@{->}[r]^{e_3-e_2}&  (  e_3-e_2,+)  }\,,$$  It is easily seen that in dimension $n=2$ this graph is generically realizable, and its equations  have the unique solution $x=\mathtt v_3$ so it is in the special component.

\end{example}

We now want to study those graphs of rank $n+1$  which are generically realizable in dimension $n$. As we have seen, on a Zariski open set of the space $\mathtt v_1,\ldots,\mathtt v_m$  we have a unique realization given by solving a system of $n$  linear equations and thus given by a vector $x$ whose coordinates are rational functions in the vectors $\mathtt v_i$. 

We have proved, Theorem \ref{codim},  that in fact the coordinates are polynomials and have called this function  the {\em generic realization}.
 
 \begin{lemma}\label{spegra}
If   a   graph of rank $\geq n+1$ has a  generic solution to the associated system, in dimension $n$, which is given by a polynomial then the graph is special and the polynomial is of the form $\mathtt v_i$ for some $i$.
\end{lemma}
\begin{proof}Denote by $a_i$  resp. $b_j$ the black and red vertices.
 
 The root $x$ is a solution of  the equations \eqref{bacos}
 $$(x,\pi(a_i))=K(a_i),\quad |x|^2+(x,\pi(b_j))=K(b_j).$$
 If the  solution $x$ is polynomial in the $\mathtt v_i$,   it is linear by a simple degree computation. 
 
 Let $g\in O(n)$  be an element of the orthogonal group of $\R^n$,  substitute in the   equations $\mathtt v_i\mapsto  g\cdot \mathtt v_i$.     By their definition the functions $K$ are   invariant under $g$  and the transformed equations have as solution $x(g)$ with  $(x(g),g\pi(a_i))= K(a_i)$.
 
 We have   $(x(g), \pi(a_i))= (g^{-1}x(g),\pi(a_i))$  so $ x(g)=g x$    is also  an equivariant linear map under the orthogonal group of $\R^n$. It follows by simple invariant theory that it has the form $x=\sum_sc_s\mathtt v_s$ for some numbers $c_s$.

  By Lemma \ref{CK} and the fact that the given system of equations is satisfied for all $n$ dimensional vectors $\mathtt v_i$ it is valid for the vectors $\mathtt v_i$  with only one coordinate  $v_i$ different from 0, or if we want for 1--dimensional vectors so that now the symbols $\mathtt v_i$ represent simple variables (and not vector variables). 
  
    So choose a vertex adjacent to the root, this is an edge either black $e_i-e_j$ or red $-e_h-e_k$.  The corresponding equation  for $x$ is
  $$x( v_i- v_j)\stackrel{\eqref{lan}}= v_j( v_i- v_j),\quad\text{or}\quad x(x v_h- v_k)\stackrel{\eqref{lar}}=-  v_h  v_k.$$
  In the first case $x= \mathtt v_j$ in the second  $x=\mathtt v_h,\ \mathtt v_k$.

   \end{proof}
   
\begin{proof}[Proof of Theorem \ref{aMT}]  By Theorem \ref{codim},  if we have a generic solution $x=F(v)$ this   is a   polynomial in $\mathtt v_1,\ldots,\mathtt v_m$.  By Lemma \ref{spegra} this is of the form $F(v)=\mathtt v_i$.\end{proof}  
We arrive at the conclusion of this first part.
\begin{theorem}\label{bou}\strut  Under the finitely many constraints  \ref{c0} through \ref{c6b} a combinatorial graph with $h$ black and $k$ red vertices has no geometric realization in the following  cases:
\begin{enumerate}\item The black or the red vertices  are linearly dependent.
\item It has  $n+1$  linearly independent vertices.
\item It has at least  $\geq 2n+1$ vertices.\end{enumerate}
\end{theorem}
\begin{proof} i)  is the content of   Theorem \ref{ridma}. \smallskip

ii)  follows from   Theorem   \ref{aMT}.\smallskip

iii) Given a combinatorial graph $\GA$    contained in a larger combinatorial graph $\GA'$ if  $\GA$  has no generic realization then so is for  $\GA'$. If  $\GA$ has $2n+1$ vertices different from the root, then it has at least $n+1$ elements of the same color or $n$ vertices of each color.  

 If $n+1$ elements of the same color are linearly independent then the statement follows from case ii).
  
  For a combinatorial graph with $n$ linearly independent black and $n\geq k>0$ linearly independent red vertices we can still apply  Theorem  \ref{aMT} since a red vertex  $b$ is linearly independent from the black ones as $\eta(b)=-2$.   \end{proof}        

\begin{remark}
In the next sections we will show that for generic $\mathtt v_i$  the graphs with a realization have at most $n+1$ vertices which are affinely independent. However this is  hard to prove, it will take the next 40 pages.

\end{remark} \part{Degenerate resonant graphs}

\section{Degenerate resonant graphs\label{drg}} {\em  The purpose of this section is to prove Theorem \ref{MM}.  }

 \subsection{Degenerate resonant graphs}

\begin{definition} We say that a combinatorial graph $A$ is {\em degenerate--resonant},   if it is degenerate and, for
all the possible linear relations $\sum_in_ia_i=0$  among its
vertices we have also $\sum_in_iC  (a_i)=0.$
\end{definition}

What we claim is that a  degenerate--resonant graph $A$ has no geometric realizations outside the special component.

\begin{remark}
One may easily verify that the previous condition,   although
expressed using a chosen root,   does not depend on the choice of
the root.
\end{remark}
 
\begin{theorem}\label{MM}
A degenerate--resonant graph $A$ is {\em not allowable} hence it has no geometric realizations outside the special component.
\end{theorem}
From this Theorem  follows the final description of thre connected components of $\Gamma_S$:
\begin{theorem}\label{affind}
For generic $\mathtt v_i$  the graphs with a realization have at most $n+1$ vertices which are affinely independent.
\end{theorem}

\subsubsection{Minimal degenerate resonant graphs}\quad Clearly, in order to prove  Theorem \ref{MM}  it is enough to prove it for  minimal degenerate resonant graphs  $\GA$, that is  graphs  which do not contain any proper degenerate resonant graph. \smallskip

We choose a maximal tree $T\subset \GA$ and then we  have noticed, in Remark 
\ref{maxt},   that a relation on the vertices implies a relation on the edges and conversely.  
\begin{lemma}\label{endp}
Every relation among the vertices of $T$  contains the end points of $T$ with non zero coefficient.

There is a unique (up to scale) relation among the vertices.

There is a unique (up to scale) relation among the edges.
\end{lemma} 
\begin{proof}
If an end  vertex of $T$  does not appear we can remove it from $T$ and obtain a proper degenerate resonant graph contrary to the assumption.

If we have two different relations and we choose an end  vertex of $T$ we can build a linear combination of these two relations in which this vertex does not appear contradicting the previous statement.

Finally since the edges are as many as the vertices  this follows since they span the same vector space\end{proof}
   Our first task is to understand the nature of these relations among the edges $\ell_i$. \medskip 

{\bf Some examples.}\quad  
\begin{proposition}\label{dueed}

A combinatorial graph in which the same edge $\ell$ appears twice has no generic geometric realization. Also in case $\ell$ black if $\ell$ and $-\ell$ both appear.
\end{proposition} \begin{proof}Suppose we have twice the same edge  $\ell $.   We take the root at one end of one of the two $\ell$ and denote by $a=\ell$ the other end.  If $\ell= e_1-e_2$, consider the other $\pm \ell$ and  say $b,c$ are the two vertices of the same color $\sigma$ so that $b-c=a$. By Lemma \ref{avr}  we have to treat only the case  in which the resonance relation is identically zero. Then  we have
$$\frac\sigma 2(b^2+b^{(2)}- c^2-c^{(2)})= C(b)- C(c)= C(\ell)=   e_1^2 - e_1e_2.  $$  
If $b=\sum_j u_je_j,\ c=\sum_j w_je_j$  we have  $u_i=w_i$ for $i\neq 1,2$ and $u_1=w_1+1,\  u_2=w_2-1$.
$$b^2=\sum_j  u_j^2 e_j^2 +2\sum_{i<j}u_iu_je_ie_j,\ c^2=\sum_j  w_j^2 e_j^2 +2\sum_{i<j}w_iw_je_ie_j.$$ Comparing the   terms                         in the $e_i^2$ on both sides  we have
$$  2e_1^2=\sigma \sum_j (u_j^2+u_j-w_j^2-w_j)e_j^2=\sigma (u_1^2+u_1-w_1^2-w_1)e_1^2 +\sigma (u_2^2+u_2-w_2^2-w_2)e_2^2$$
substituting   $u_1=w_1+1,\  u_2=w_2-1$ we have:
$$\implies 0=( w_2-1)^2+ w_2-1-w_2^2-w_2=-2w_2\implies w_2=0,\  u_2=-1 $$
Next compare the mixed terms      $e_ie_j,\ i\neq j$      
 \begin{equation}\label{ausi}
 - e_1e_2= \sum_{i<j}u_iu_je_ie_j-\sum_{i<j}w_iw_je_ie_j\implies u_1u_2-w_1w_2=-1,\implies u_1=1,\ w_1=0.
\end{equation}If there is a $j\neq 1,2$ with $u_j=w_j\neq 0$ then, since  the coefficients of $e_1e_j,\ e_2e_j$ in the left of Formula \eqref{ausi} are 0 we deduce $u_1=w_1,\ u_2=w_2$  a contradiction.
Therefore $b=e_1-e_2,\ c= 0 $ and the two edges are the same.\bigskip

   If $\ell= -e_1-e_2$ say $b,c$ are the two vertices of opposite colors $1 , -1$ so that $b+c=a$. Hence the resonance relation is 
\begin{equation}\label{RRis}
\frac 1 2(b^2+b^{(2)}- c^2-c^{(2)})= C(b)+ C(c)= C(\ell)=  - e_1e_2  .
\end{equation}If $b=\sum_j u_je_j,\ c=\sum_j w_je_j$  from $b+c=\ell$ we have  $u_i=-w_i$ for $i\neq 1,2$ and $u_1=-w_1-1,\  u_2=-w_2-1$.

Comparing  the   terms                         in the $e_i^2$ on both sides
$$0=u_i^2+u_i- w_i^2-w_i\implies u_i=w_i=0,\ \forall  i\neq 1,2$$$$0=(w_i-1)^  2+w_i-1- w_i^2-w_i=-2w_i, \ i=1,2 ,\  u_1=u_2=-1 .     $$ 
  
We thus have $c=0, \ b=a$ the same edge.\end{proof}
\subsubsection{Recall the basic formulas\label{BaF}} We  work with $G_{2}$  identified with elements in $\Z^m$  either with $\eta(a)=0$,   {\em black} or $\eta(a)=-2$   {\em red}.  

We have set  $C(a)=\frac{1}{2}  (  a^2+  a^{  (2)})$ for $a$ black and  $C(a)=-\frac{1}{2}  (  a^2+  a^{  (2)})$ for $a$ red.

In our computations we use always the rules: 
\begin{itemize}
\item for $u,  v$ black, we have $u+v$ black and

$$1)\quad  C  (u+v)=\frac{1}{2}  (  (u+v)^2+  (u+v)^{  (2)})=C  (u)+C  (v)+uv$$
\item for $u$ black $v$ red, we have $u+v$ red and

$$2)\quad C  (u+v)=-\frac{1}{2}  (  (u+v)^2+  (u+v)^{  (2)})= -C  (u)+C  (v)-uv$$

\item for $u,  v$ red, we have $u-v$ black and
$$3)\quad C  (u-v)=\frac{1}{2}  (  (u-v)^2+  (u-v)^{  (2)})=\frac{1}{2}  (  (u^2+v ^2-2uv+  (u-v)^{  (2)})$$$$=\frac{1}{2}  (  (u^2+v ^2-2uv+  (u-v)^{  (2)})=-C  (u)+C
(v)+v^2-uv$$
\item for $u$ black, we have $-u$ black and
$$4)\quad C  (-u)=C  (u)-u^{  (2)}. $$\end{itemize}

\subsection{Encoding graphs}  

 In order to understand relations among edges,   consider   the complete graph $T_m$ on the vertices  $1,  \ldots,  m$.    If we are given a     list  $P$ of   edges $\ell_i\in X$ we associate to it   the subgraph $\Lambda_P$ of  $T_m$,   called   {\em the encoding graph of $P$}   in which, the vertices are the indices appearing in $P$ and we join two vertices $i,  j$ with a black edge  if $P$ contains an edge marked $e_j-e_i$ or  $e_i-e_j$ and by a    red edge edge if $P$ contains an edge marked $-e_j-e_i$. 
 
 We mark $=$ the red edges. A priori it is possible that both markings appear but by Proposition  \ref{dueed} each appears at most once.  
In order to distinguish combinatorial from encoding graphs we refer to {\em indices} the vertices of an encoding graph.\smallskip

In particular given a  degenerate resonant graph $\Gamma$ we choose a maximal tree, as in Lemma \ref{endp}, which determines  a minimal relation among its edges and define  $\mathfrak E:=(\mathcal V,\mathcal E)$  the vertices and the edges of the encoding graph of  the edges appearing in this relation.   $\mathfrak E$ depends on the choice of the tree, but one can analyze what happens changing this choice, as in \eqref{esc}. \smallskip

{\bf Examples}\quad For the graph of Formula \eqref{mig}, which is already a tree,  the encoding graph of the graph and of the minimal relation coincide:
\begin{equation}\label{esmd}
 \xymatrix{ &3\ar@/^/[r]^{-e_1-e_3}&\ar@/^/[l]^{e_1-e_3}1\ar@/^/[r]^{-e_1-e_2}&\ar@/^/[l]^{e_1-e_2}2& }.
\end{equation} 

In example \eqref{prie} some maximal trees  and their encoding graphs:$$ \xymatrix{ &  \bullet\ar@{<- }[d] _{2,  1}& &\\  &  \bullet & &\\  x\ar@{ ->}[ru] _{3,  2} &&    \textcolor{Red}{\bullet}   \ar@{=}[lu] ^{1,  3} &   } \ 
\xymatrix{ &  \bullet & &\\  &  \bullet & &\\  x\ar@{ ->}[ru] _{3,  2}\ar@{<-}[ruu] ^{  1,3}\ar@{=}[rr]_{1,  2} &&    \textcolor{Red}{\bullet}    &   }\ 
\xymatrix{ &  \bullet & &\\  &  \bullet & &\\  x \ar@{<-}[ruu] ^{  1,3}\ar@{=}[rr]_{1,  2} &&    \textcolor{Red}{\bullet}    \ar@{=}[lu] ^{1,  3} &   } $$

$$\xymatrix{   &  2& &\\  1\ar@{ ->}[ru] \ar@{=}[rr]    &&    3  \ar@{=}[lu]   &   }\ 
\xymatrix{   &  2& &\\  1\ar@{=}[ru] \ar@{->}[rr]    &&    3  \ar@{<-}[lu]   &   }\ 
\xymatrix{   &  2& &\\  1\ar@{ =}[ru] \ar@{=}[rr]    &&    3\ar@/^/[ll]    &    },\ 
$$ In this case there is no relation.

\smallskip

We use the symbol $\mathcal V$ also for the indices and by $V_{\mathcal V}$ the lattice spanned by the $e_j,\ j\in \mathcal V$.\smallskip

 Recall that the {\em valency} of a
vertex in a graph is the number of edges which admit it as vertex.
\begin{lemma}\label{com}
The encoding graph  $\mathfrak E:=(\mathcal V, \mathcal E)$  of a minimal relation is connected and each of its vertices has valency $\geq 2$.
\end{lemma}\begin{proof}
For each connected component $C$ of $\mathfrak E$ consider the subspace $V_C$ spanned by the vectors $e_i,\ i\in C$, which contains the span of the edges in $C$. 
 
The subspaces $V_C$ form a direct sum, so the   relation decomposes into a sum of terms each supported in a componente $V$ and each a relation. Hence the encoding graph of a minimal relation is connected.  \medskip
 
   The graph $\mathfrak E$ cannot have any vertex of valency 1,
since this would appear in only one edge of $\mathcal E$  which is clearly linearly independent from
the others and does not appear in a relation.
\end{proof}
   A  basic relation among the edges $\ell_i$ is the {\em circular relation}. We can visualize the algorithm as a substitution of   two consecutive edges with a single one:
\begin{equation}\label{esc}
(e_i-e_j)+(e_j-e_k)+(e_k-e_i)=0,\quad \xymatrix{i\ar@{-}[rd]\ar@{-}[rr]&&k\\&j\ar@{-}[ru]}
\end{equation} $$ (e_i-e_j)-(-e_j-e_k)+(-e_k-e_i)=0,\quad \xymatrix{i\ar@{-}[rd]\ar@{=}[rr]&&k\\&j\ar@{=}[ru]}$$
$$ -2e_i=-(e_i-e_j)+(-e_i-e_j).$$
In general the encoding graph of such a  relation, {\em with signs $\delta=\pm$  see \eqref{priR}}, is a simple circuit as:  
   $$ \xymatrix{   &  1\ar@{ ->}[rd]^+\ar@{ =}[ld]_{-}&  \\   5\ar@{ ->}[d]_+&   &2\ar@{ ->}[d] ^+&  \\  4   \ar@{=}[rr]_+   & &3      },\  \xymatrix{   &  1\ar@{->}[rd]^+\ar@{<-}[ld]_{+}&  \\   5\ar@{->}[d]^{-}&   &2\ar@{->}[d]^+ &  \\  4     & &3  \ar@{<-}[ll]_{-}      },\ \xymatrix{   &  1\ar@{ ->}[rd]^+\ar@{ =}[ld]_{-}&  \\   5\ar@{=}[d]_+&   &2\ar@{ =}[d] ^+&  \\  4   \ar@{=}[rr]_{-}   & &3      },\  $$
 \begin{lemma}\label{icamm}
Consider $k$ edges  $\ell_i=\theta_ie_i-e_{i+1},\ \theta_i=\pm 1,\ i=1,\cdots,k$.

1)\quad The edges $\ell_i$ are linearly independent and there exist unique $\  \delta_i=\pm 1$:
\begin{equation}\label{priR}
\sum_{i=1}^k\delta_i \ell_i= \theta  e_1- e_{k+1},\quad \theta=
\prod_{i=1}^k\theta_i.
\end{equation}2)\quad Moreover $\delta_k=1$ and for all $1<u\leq k$ we have $\delta_u=\delta_{u-1}$ if   $\delta_u$ is black,  $\delta_u=-\delta_{u-1}$ if   $\delta_u$ is red, $\delta_1=
\theta\theta_1$.
\smallskip

3)\quad As element in $G_2$  we have that  $ \theta  e_1- e_{k+1}$ is the composition $\ell_k\circ\ell_{k-1}\circ\cdots\circ\ell_1$ of the $\ell_i$ as group  elements.\end{lemma}\begin{proof} 
1)\quad By induction there exist $\eta_i=\pm 1$ so that   $\sum_{i=1}^{k-1}\eta_i\ell_i= \prod_{i=1}^{k-1}\theta_i e_1- e_{k} $.  

Set $\delta_k=1,\ \delta_i=\theta_k\eta_i,\ i=1,\cdots, k-1$  and we have $$\ell_k+\sum_{i=1}^{k-1}\theta_k  \eta_i\ell_i=\theta_k ( \prod_{i=1}^{k-1}\theta_i e_1- e_{k})+  (\theta_k e_k-e_{k+1})= \prod_{i=1}^{k }\theta_i e_1- e_{k+1}  .$$ The Formula $\delta_1=
\theta\theta_1$ is proved by induction.

2)\quad Since each $1<u\leq k$ does not appear  in the right hand side of Formula  \eqref{priR}  we must have cancellation from the only two edges in which $e_u$ appears, that is, cf. Formula \eqref{cofu}
\begin{equation}\label{comc}
\ell_u\circ\ell_{u-1}=(\theta_{u}e_u -e_ {u+1})+ \theta_{u}(\theta_{u-1}e_{u-1}-e_u)= \theta_{u}\theta_{u-1}e_{u-1} -e_ {u+1}.
\end{equation}
3)\quad This follows from the previous Formula by induction.\end{proof}
 Now choose an index $p\in \mathcal V$ and consider a maximal simple path in $\mathcal V$ from $p$,  that is a sequence of distinct indices $p=p_1,\cdots, p_k$ with $p_i,p_{i+1}$ joined by an edge $\ell_j\in\mathcal E$.  Since $p_k$ has valency  $>1$  there is an edge $\ell_{k+1}\neq \ell_k$ joining $p_k$ with a vertex $p_{k+1}$. 
 
 Since the path is maximal  we must have  $p_{k+1}=p_i$ for some $i<k$. We have thus a {\em simple circuit} originating in $p_i$,  in the graph $\mathfrak E$.\smallskip
 
 In order to simplify the notations and changing name to the indices we may assume that the circuit is $1,2,\ldots, j,1$.
  So we have for each pair $i,i+1$  an edge $\ell_i=  \theta_ie_i-e_{i+1}, \ i=1.\cdots, j-1,\   \ell_j=  \theta_je_j-e_1$ in the minimal relation of which $\mathfrak E$ is the encoding.\smallskip
 
 From formula \eqref{priR} we deduce, since $e_{j+1}=e_1$:
 \begin{equation}\label{larr}
 \sum_{i=1}^j\delta_i\ell_i=(\delta_j-1)e_{1},\qquad \delta_j=   \prod_{i=1}^j \theta_i.
\end{equation}
 If $\delta_j=1$ this is a relation and by Definition \eqref{priR}  the number of red edges in the list is even, otherwise this number  is odd, $\sum_{i=1}^j\delta_i\ell_i=-2e_{1}$  and these edges are linearly independent and span a lattice of index 2 in $\Z^j$ (see Lemma \ref{zita}).
 \begin{definition}\label{evc}
We say that a simple circuit in  $\mathfrak E$ is even if it contains an even number of red edges, otherwise it is odd.
\begin{remark}\label{mino}
A minimal odd circuit may be formed by just two edges $  \xymatrix{1\ar@/_/[r]_{e_1-e_2}&\ar@/_/[l]_{-e_1-e_2}2  }$ cf.   \eqref{esmd}.
\end{remark}
\end{definition}

Thus we have proved:
 \begin{proposition}\label{pari}
  Take a list of  edges $L:=\{\ell_1,\cdots,\ell_j\}$,  and  $k$ of this list are red edges,   with encoding graph a  simple path  from $p_1$ to $p_{j+1}$ which adding an edge $\ell_{j+1}$  becomes a circuit from $p_1$  to $p_1$. \begin{enumerate}\item The edges $L$  are linearly independent.  
\item  A linear combination  $\Sigma$ with signs of these elements is $e_i- (-1)^ke_{j+1}$. 

\item  If the circuit is even  there is a unique relation, up to sign:

$$R=\sum_{i=1}^{j+1}\delta_i\ell_i=0,\quad \delta_i=\pm 1$$ for the edges  $L':= \{\ell_1,\cdots,\ell_j, \ell_{j+1}\}$  with coefficients $\pm 1$.

\item  If  the circuit is odd 
   the edges  $ \{\ell_1,\cdots,\ell_j, \ell_{j+1}\}$ are linearly independent and span the vector  space with basis the $e_i$  for $i$  the vertices of the circuit. \item  In this last case there is a linear combination of the edges $L'$  with coefficients $\pm 1$ equal to $2 e_i$ for each index $i$ in the vertices of the circuit.\end{enumerate}
  
 \end{proposition}
 \begin{proof} This is the content of Lemma \ref{icamm}.  

A degenerate example of these two cases is for  a circuit with two edges   $\ell_1,\ell_2$ between $1,2$ (cf. figure \eqref{esmd}), The case $\ell,-\ell$ black has been excluded by Proposition \ref{dueed}. \end{proof}
 
 \begin{corollary}\label{evx}
A circuit in the encoding graph corresponds to a relation  between the corresponding edges, and so to the entire encoding graph of the relation, if and only if it contains an even number of red edges and we call it an {\em even circuit}.
\end{corollary}   \subsubsection{Doubly odd circuits}
If the circuit $\mathcal C=\{1,\cdots, j,1\}$ we have chosen is odd we have seen that its edges are linearly independent so it cannot  coincide with the entire encoding graph $\mathfrak E$  of the relation.   We need to {\em double} the circuit. We can visualize the algorithm in 2 simple cases, depending on the color of the edge between  $i,p$:
\begin{equation}\label{esc1}
\xymatrix{(e_i-e_j)+(e_j-e_k)-(-e_k-e_i)=2e_i,\ &j\ar@{<-}[rd]^+\ar@{->}[rr]^+&&i \ar@/^/[r]^{ 2} \ar@{=}[r]_{2} &p\ar@{<-}[rd]^{-}\ar@{<-}[rr]^{-}&&\ell&\\
\!\!\!\!\!\!\!\!\!\!\!-(e_p-e_m)-(e_p-e_\ell)+(-e_\ell-e_m)=-2e_p,& &k\ar@{=}[ru]^{-}&&&m\ar@{=}[ru]^+&}
\end{equation} In one case $2e_i-2e_p+2(e_p-e_i)=0$  in the other $2e_i+2e_p+2(-e_p-e_i)=0$ the sum is 0. \smallskip

We want to show that, if in the encoding graph  $\mathfrak E$ oi of a minimal relation, we have encountered an odd circuit $\mathcal C$, which sums with signs to some $2e_i$ we will find a picture which generalizes the previous figure \eqref{esc1}. Assume the indices of $\mathcal C$ are  $1,\cdots, j $. 
\smallskip

Since  $\mathfrak E$ is connected  there is a vertex in $\mathcal C$, which without loss of generality we may take 1, and from this vertex starts a new simple path $m_1,\ldots, m_a$ with  vertices outside $1,\cdots, j $. Without loss of generality we may assume the new path to be    $j+1, \cdots , j+a+1$. Since no vertex has valency 1 at some point  there is  a further edge $m_{a+1}$ from $ j+a+1$ to one of the preceding vertices  $b<j+a+1$.  \smallskip
  
A priori we have two possibilities, the first is  $b\in \{2,\cdots, j \}$.  We claim that this case can be excluded since then we have in the encoding graph an even circuit which gives the relation and coincides with the encoding graph $\mathfrak E$. 

 In fact let us prove this with a picture:  The graph of the entire path looks as 
$$\xymatrix{            2 \ar@{-}[rrr] \ar@{-}[d] &    & & \ar@{-}[d]\\     1\ar@{-}[rrr] \ar@{-}[d] & &   & b\ar@{-}[d]\\    j+1  \ar@{-}[rrr]  & &   &  j+a   } $$  
Here we see 3 possible circuits and then at least one of them is even.

So the other alternative is that we  have a second circuit which is also odd and which is either disjoint from the first circuit  and connected by a path,   or  $b=1$. 

We call this a {\em doubly odd encoding graph}, the simplest examples being as in \eqref{esc1} and \eqref{esmd}:

\begin{example} An even and a doubly odd encoding graph:\bigskip

$$\xymatrix{       & 1   \ar@{=}[r] &10\ar@{-}[r]    &9\ar@{-}[r]  &8\ar@{-}[r] &7    && \\   &2 \ar@{-}[u] \ar@{=}[r]   &3\ar@{ =}[r]    &4\ar@{-}[r]  &5\ar@{-}[r] &6\ar@{=}[u]      &&  }$$\bigskip

\begin{equation}\label{doppio}
\xymatrix{    &12\ar@{-}[r]&13\ar@{=}[r]&14\ar@{-}[r]&15\ar@{=}[r]&16&    && \\ B\!\!\!\!\!\!\!\!\!\!\!& 11\ar@{-}[u] &   &19\ar@{-}[r]  \ar@{-}[u] _{\quad \quad \quad  C}&18\ar@{-}[r] &17\ar@{-}[u]     &&  \\   & 1 \ar@{-}[u]  \ar@{=}[r] &10\ar@{-}[r]    &9  \ar@{-}[r]  &8\ar@{=}[r] &7      && \\   &2 \ar@{-}[u]_{\quad \quad \quad  A} \ar@{=}[r]   &3\ar@{ =}[r]    &4\ar@{-}[r]  &5\ar@{-}[r] &6\ar@{=}[u]      &&  }
\end{equation}\end{example}
 
 We also have the special case $b=1$  where the two odd circuits have a vertex in common, as in the minimal case of \eqref{esmd}, depicted by the example
 
\begin{equation}\label{dops}
 \xymatrix{       &     \ar@{-}[dr]   \ar@{-}[dl] &   &&&&&&\\    \ar@{-}[dr]  & A &  \ar@{-}[dl] 1   \ar@{-}[ur]  & C  &  \ar@{-}[ul]  && \\       & \ar@{-}[ur]   \ar@{-}[ul]& &         \ar@{-}[ur]    \ar@{-}[ul]&}
\end{equation}

 \begin{proposition}\label{dodd}
A doubly odd circuit gives a minimal relation $R$. The coefficients in the two circuits are $\pm 1$ while in the path $P$  joining the two circuits the coefficients  are $\pm 2.$
\end{proposition}
\begin{proof}
Let $a,b$  be the end points of the path joining the two odd circuits. 

By Proposition  \ref{pari} $v)$ we have a linear combination  of the edges in each of these circuits equal to $2e_a,\ 2e_b$. 

By  Proposition  \ref{pari} $ii)$ we have a linear combination of the edges in $P$  with coefficients $\pm 1$ equal to $e_a\pm e_b$.   So $ 2e_a-2(e_a\pm e_b)\pm  2e_b=0$ gives the required relation which is clearly unique, since  removing the last edge the remaining are linearly independent, and  satisfies the constraint on the coefficients $\pm1,\ \pm 2$.

Of course in the special case $a=b$ we have no path.\end{proof}

Up to changing the indices we may assume that we walk the circuit first from 1  back to 1  in part $A$  then to $j$ on path $B$ and then back to $j$ on circuit $C$  so  that the indices are increasing from 1  to $k$. So the double odd circuit has the form:

\begin{equation}\label{doppiou}
\xymatrix{    &u\ar@{-}[r]& u+1\ar@{.}[r]&j\ar@{-}[r]&k\cdots \ar@{=}[r]& &    && \\ B\!\!\!\!\!\!\!\!\!\!\!&  u-1\ar@{-}[u] &   &j+1 \ar@{-}[r]  \ar@{-}[u] _{\quad \quad \quad  C}&\cdots \ar@{-}[r] & \ar@{-}[u]     &&  \\   & 1 \ar@{.}[u]  \ar@{=}[r] &h\ar@{-}[r]    &\cdots   \ar@{-}[r]  & \ar@{=}[r] &       && \\   &2 \ar@{-}[u]_{\quad \quad \quad  A} \ar@{=}[r]   &3\ar@{ =}[r]    & \ar@{-}[r]  &\cdots \ar@{-}[r] &6\ar@{=}[u]      &&  }
\end{equation} If $\theta $ is the color of the path $B$, we have   a unique choice of $\delta_i,\eta_j$  so that, Formula \eqref{larr}:
    \begin{equation}\label{boB}\theta e_{j}= \sum_{i=h+1}^{j-1} \eta_i\ell_i+e_1,\  
 -2e_1=\sum_{i=1}^h\delta_i\ell_i,\  -2\theta  e_j=\sum_{i=j}^{k}\delta_i\ell_i,\    \end{equation}
 \begin{equation}\label{boBf}\mathcal R:\qquad
 0=\sum_{i=1}^h\delta_i\ell_i  + 2\sum_{i=h+1}^{j-1} \eta_i\ell_i +\sum_{i=j}^{k}\delta_i\ell_i , \  \eta_i,\delta_i=\pm 1.
\end{equation}
\begin{remark}\label{dw} 1)\quad Notice that $\ell_i=\vartheta_ie_i-e_{i+1}$ if $j\neq  h+1,k$.  Then $\ell_{h+1}=\vartheta_{h+1} e_1-e_{h+2},\ell_k=\vartheta_ke_k-e_j$.

2)\quad We can also think of a doubly odd circuit as a form of degenerate  even circuit in which we {\em walk back}  on the path  joining the two odd circuits, then the values of the signs $\delta_i,\ \eta_j$  is determined by the same rules as in Lemma \ref{icamm}.

\end{remark}
\begin{proposition}\label{pmR}
$\mathcal R$  is the form of a minimal relation. By Lemma \ref{endp}  we know that this is unique up to scale, so if there is another relation among the edges $\mathcal R'$ and one of its coefficients is $\pm 1$  then $\mathcal R'=\pm\mathcal R$.
\end{proposition}
\label{zita}Let  $e_1,
\ldots,  e_{ k}$ be the basis vectors appearing in the minimal relation $\mathcal R$ in Formula \eqref{boBf}.\smallskip

 Set $\zeta:\Z^k\to\Z,  \ \zeta  (e_i)= \zeta_{i }$  recursively
 $$\zeta  (e_1) =1,\ \zeta  (e_{i+1})=\vartheta_i\zeta(e_i) \implies \zeta(\ell_i)=\vartheta_i\zeta_{i }-\zeta_{i+1}=0,\ i<k .$$ 

\begin{lemma}In case 1) the $\ell_i$  span the codimension 1 sublattice of the lattice $\Z^k$ with basis $e_1,
\ldots,  e_{ k}$     formed by the vectors $a$ such that \begin{equation}
\label{par}a=\sum_i\alpha_ie_i \ |\,   \zeta  (a)=\sum_i
 \zeta(e_{i })\alpha_i=0.
\end{equation}
 
In case 2) the $\ell_i $    span over $\Z$  the lattice  of index 2 in    $\Z^k$ given by  \begin{equation}
\label{par1}a=\sum_i\alpha_ie_i \ |\,   \eta   (a)=\sum_i
  \alpha_i\cong 0,\ \text{modulo}\ 2.
\end{equation}
\end{lemma}\begin{proof} In case 1) 
$\zeta(\ell_i) =0$, so the $\ell_i,\ i<k$ are in this proper subspace, but also $\ell_k$ is in this subspace since it is a linea r combination of the preceding ones, but when we add to the $\ell_i$ the vector $e_1$ they span  $\Z^k$  hence the claim.

In case 2) $\eta (\ell_i) \cong 0,\ \forall i$ modulo 2, so the $\ell_i$ are in this sub--lattice,  the fact that they span  is easily seen by induction adding $e_1$ as before.\end{proof}

\subsection{Minimal relations}We have taken a minimal degenerate resonant graph $\Gamma$,  and  a given maximal tree  $T$ in  $\Gamma$. The   relation for the vertices gives a relation for the edges and, in the previous paragraph, we have described the possible encoding graphs of this relation.\medskip

Call  $\mathcal E$ the set of edges appearing in the minimal relation.  Call $|\mathcal E|$  the subgraph of $T$  formed by
the edges $\mathcal E$.   $|\mathcal E|$  need not be a priori
connected  but only a {\em forest} inside $T$.  \smallskip

From what we have seen in the previous  paragraph  the encoding graph of   $\mathcal E$ is either 
  an even circuit  and the relation is a sum of edges
$\sum_j\delta_j\ell_j=0$,   with signs $\delta_j=\pm 1$  or a doubly odd circuit   and  we may have some coefficients $ \pm 2$ corresponding to the
edges appearing in the segment connecting the two odd circuits, relation $\mathcal R$  of Formula \eqref{boBf}.  

{\bf Warning}\label{warn}\quad  From now on we will write  instead of Formula \eqref{boBf}\ for $\mathcal R$ a compact Formula  $\sum_{i\ }\delta_i \ell_i$ but with the proviso that some $\delta_i=2\eta_i$ may be $\pm 2$.\medskip

 In
any case we list the edges appearing in the relation as $\ell_i$.  

\begin{definition}\label{abi}[of $a_i,b_i$]
Each $\ell_i$
black  i.e. $\theta_i=1$  is $\ell_i=a_i-b_i$  with $a_i,  b_i$,   its vertices of the
same color.  

For  $\ell_i$
 red i.e. $\theta_i=-1$ we have  $\ell_i=a_i+b_i$  with $a_i$ red and
$b_i$ black,  its vertices.
\end{definition}

The relation $\mathcal R$  \eqref{boBf} is thus in term of vertices with $\delta_i=\pm 1,\ \pm 2$
\begin{equation}
\label{leqed}\sum_{i\ }\delta_i  (a_i-\theta_ib_i)=\sum_{i\ \text{black}}\delta_i  (a_i-b_i)+\sum_{j\ \text{red}}\delta_j  (a_j+b_j)=0 .
\end{equation}

Note that
a vertex in $T$  need not appear in  $R$ 
however all end-points  in $T$   must appear by Lemma \ref{endp}.

 We say that an index is critical if the corresponding vertex in the encoding graph has valency $>1$. In Figure \eqref{doppiou}  1 and $k$ are critical. 
 
 In Proposition \ref{icin} we will describe precisely the entire encoding graph of $T$ and then in the even case we may also have two critical indices for this larger encoding graph.

\begin{remark}\label{crii}
The  non critical indices are divided in 2 or 3  sets (depending if we have only one critical index or two) which we denote $A,B,C$ as in the figures.
If $ u$  is not critical  we have $\delta_u=\vartheta_u\delta_{u-1}$ by Lemma \ref{icamm}.
 
\end{remark}

\medskip
\section{The resonance \label{TR}}
 \subsection{The resonance relation}

This section is devoted to the proof of Theorem \ref{MM}.

\subsubsection{Signs} With the notations of the previous paragraph    
we choose a root  $r$  in $T$  and then each vertex $x$ acquires a color  $\sigma_x=\pm1=\eta(x)+1$. Recall that the color of $x$ is   red, $\eta(x)=-2$ and $\sigma_x=-1$ if the path  from the root to $x$ has an odd number of red edges, the color is black  $\eta(x)=0$ and $\sigma_x=1$ if  the path  is even, cf. figure \eqref{prie}.

By convention by $\ell_i$ we mean $e_i-e_{i+1}$ if black, otherwise $\ell_i=-e_i-e_{i+1}$ with the proviso of Remark \ref{dw} 1) for the critical indices.  We use also the formula $\ell_i=\vartheta_ie_i-e_{i+1},\ \vartheta_i=\pm 1 $ when the color is not specified.
\begin{definition}\label{biai}\strut
 \begin{enumerate}\item Each red edge $\ell_i$  (that is $\vartheta_i=-1$) appears as edge with one
end denoted by $a_i$ red and the other denoted by $b_i$ black,   we have $\ell_i=a_i+b_i$.
\item 
 For a black edge  $\vartheta_i=1$  we define $a_i,b_i$ so that   $a_i=b_i+\ell_i$, and $a_i,b_i$ have the same color.

 We thus write $\ell_i=a_i-\vartheta_ib_i$. The relation becomes  in term of the vertices:
 \begin{equation}
\label{retr0} R:=\sum_{i }  \delta_i  (a_i-\vartheta_ib_i)= \sum_{i\,|\,\vartheta_i=-1} \delta_i   ( a_i+b_i )+\sum_{i\,|\,\vartheta_i= 1}  \delta_i  (a_i-b_i)=0.
\end{equation}
In particular for the resonant trees:
 \begin{equation}
\label{retr01}\mathcal  R:=  \sum_{i\,|\,\vartheta_i=-1} \delta_i   ( C(a_i)+C(b_i ))+\sum_{i\,|\,\vartheta_i= 1}  \delta_i  (C(a_i)-C(b_i))=0.
\end{equation}
\item An edge $\ell_i$ is connected to the root $r$ by a unique path $\pi_i$ ending with $\ell_i$, 

\item We denote $x_i$ the final vertex of $\pi_i$ and we set  $\sigma_i:=\sigma_{x_i}$. 

\item  If $\ell_i$ is black we set  $\lambda_i=1$ if  the edge is equioriented with the path, that is it points outwards, $\lambda_i=-1$ if it points inwards. Finally we set $\lambda_i=1$ if  the edge is red.  \end{enumerate}

\end{definition}

\begin{equation}
   \xymatrix{   r    \ar@{.}[r] \quad    \cdots  \cdots  \ar@{->}[r]^{\quad\ \ \ell_i}  & x_i    }\quad  \lambda_i=1,\qquad  \xymatrix{   r    \ar@{.}[r]  \quad    \cdots  \cdots  \ar@{<-}[r] ^{\quad\ \ -\ell_i}  & x_i    }\quad  \lambda_i=-1.
\end{equation}
\begin{remark}\label{finn}
A vertex $v$  can be equal to one or more elements $a_h,b_h$  according to its valency in the tree $T$.
\end{remark}
\begin{lemma}\label{calc}

1)\quad For $a_i$ and $    \ell_i=-e_i-e_{i+1}$ red, we have $b_i+a_i=  \ell_i$  and $b_i$ is black:
\begin{equation}\label{c+c}
C(a_i)+C(b_i)=-a_i^{(2)}  -\ell_ia_i+e_ie_{i+1}
\end{equation} 2)\quad  For $a_i=b_i+\ell_i$ and $\ell_i=e_i-e_{i+1}$ black we have with $\sigma_i$ the sign of $a_i,b_i$:
\begin{equation}\label{c-c}
C(a_i)-C(b_i)=\sigma_i[ -e_{i+1}^2+ e_ie_{i+1} +\ell_ia_i   ].
\end{equation}
\end{lemma}
\begin{proof}1)\quad  When $\ell_i=-e_i-e_{i+1}$ red, $\ell_i ^2 +\ell_i^{(2)}= 2e_ie_{i+1}$ we have:\smallskip

$$C(a_i)+C(b_i)=-\frac12(a_i^2+a_i^{(2)})+\frac12(b_i^2+b_i^{(2)}) =-\frac12(a_i^2+a_i^{(2)})+\frac12((\ell_i-a_i)^2 +\ell_i^{(2)}-a_i^{(2)})$$
$$=-\frac12(a_i^2+a_i^{(2)})+\frac12( \ell_i^2-2\ell_ia_i+a_i^2  +\ell_i^{(2)}-a_i^{(2)}) =-a_i^{(2)}  -\ell_ia_i+e_ie_{i+1}.$$
\medskip

2)\quad When  $    \ell_i= e_i-e_{i+1}$ black $\ell_i^2 -\ell_i^{(2)}=2e_{i+1}^2- 2e_ie_{i+1}  $ we have:\smallskip

 $$C(a_i)-C(b_i)=\sigma_i[\frac12(a_i^2+a_i^{(2)})-\frac12(b_i^2+b_i^{(2)})] =\sigma_i[\frac12(a_i^2+a_i^{(2)})-\frac12((a_i-\ell_i)^2 -\ell_i^{(2)}+a_i^{(2)})]$$
$$=\sigma_i[ -\frac12( \ell_i ^2 -2\ell_ia_i-\ell_i^{(2)} )]=\sigma_i[ -e_{i+1}^2+ e_ie_{i+1} +\ell_ia_i   ].$$ \end{proof}

In particular for the resonant trees Formula \eqref{retr01} becomes:
\begin{proposition}
 \begin{equation}\mathcal R:=
\label{retr} \sum_{i\,|\,\vartheta_i=-1} \delta_i   ( -a_i^{  (2)}- \ell_i a_i +e_ie_{i+1} )+\sum_{i\,|\,\vartheta_i= 1}  \delta_i \sigma_i( -e_{i+1}^2+ e_ie_{i+1}+\ell_i a_i )=0.
\end{equation}
$$\sum_{i\,|\,\vartheta_i=-1} \delta_i   ( b_i^{  (2)} +\ell_ib_i  -e_ie_{i+1} )+\sum_{i\,|\,\vartheta_i= 1}  \delta_i \sigma_i(e_i ^2-e_ie_{i+1}+\ell_i  b_i )=0 $$
\end{proposition}
\begin{proof}
We start from the relation $\sum_i\delta_i\ell_i=0$  written in  the previous formula \eqref{retr0}  $$0=\sum_{i }  \delta_i  (a_i-\vartheta_ib_i)=\sum_{i\,|\,\vartheta_i=-1} \delta_i (a_i+b_i)+\sum_{i\,|\,\vartheta_i= 1}  \delta_i (a_i-  b_i).  $$

We next  have by the resonance hypothesis $$\sum_{i\,|\,\vartheta_i=-1} \delta_i   ( C(a_i)+C(b_i) )+\sum_{i\,|\,\vartheta_i= 1}  \delta_i  (C(a_i)-C(b_i))=0.$$
We then apply Lemma \ref{calc}.
The second identity follows from the first by substituting $a_i=b_i\pm \ell_i$ in the two cases. \end{proof}

 \subsubsection{Some reductions}
Denote by $b_i=\sum_{h=1}^mb_{i,  h}e_h$ and expand the second Formula
\eqref{retr}. Observe that  the coefficients  of the mixed terms  $e_ie_j,\ i\neq j$   come all  from the sum

$$B:=\sum_{i\,|\,\vartheta_i=-1}   \delta_i   (   \ell_i b_i -e_ie_{i+1} )+\sum_{i\,|\,\vartheta_i= 1} \delta_i  \sigma_i( -e_ie_{i+1}+\ell_i b_i ).
$$ where $i\in [1,  \ldots,   k]$  the support of the relation \eqref{boBf}.

If $h\notin[1,  \ldots,   k],   $  the
coefficient of $e_h$ in $B$  (which must be equal to 0)  is thus
$$ \sum_{i\,|\,\vartheta_i=-1}   \delta_i       \ell_i b_{i,  h} +\sum_{i\,|\,\vartheta_i= 1} \delta_i  \sigma_i  \ell_i b_{i,  h} =0.
$$ 
By the uniqueness of the relation it follows that this relation is a multiple of \eqref{boBf} (with the conventions that some $\delta_i=2\eta_i$)   hence the numbers $b_{i,  h},\ \vartheta_i=-1$ and $ \sigma_i    b_{i,  h},\ \vartheta_i= 1$ are all equal.  

Since  now we can choose as root one of the elements $b_i$  we deduce that all these coefficients $b_{i,  h}$ equal to 0. Thus:
\begin{lemma}\label{supp}
With this choice of root, all $b_i,a_i$ have support in the vertices $[1,2,\cdots,k]$ of the encoding graph of the relation.

\end{lemma} 

 \smallskip

 Let $T'$ be the forest  support of the edges $\ell_i$, of the relation.   If this is a tree it must coincide with $T$ by minimality of $\Gamma$.

If $T'$ is not a tree the edges in $T\setminus T'$  are linearly independent  with respect to the span of the edges in $T'$ otherwise we would have a second relation contrary to Lemma \ref{endp}.

There is at least one segment $S$ (a simple path) in $T\setminus T'$  joining two end points in $T'$, the edges in $S$   are linearly independent from the edges  in the relation, a typical case will be that in Figure \eqref{minG0}.  
 \smallskip

 Since $S$ connects two points   $p,q\in T'$ the element $g\in G_2$ with $g\cdot p=q$ is of the form  $E, E\tau,\ E\in\Z^m_2$.  Since $p,q$ have both support in  $[1,2,\cdots,k]$ and $g=q\circ p^{-1}$    we have that    $g$ has the form $E=\sum_{i=1}^k\alpha_ie_i$ and $\eta(E)=0,-2$.

\begin{lemma} \strut \begin{enumerate}[1)]\item If we are in case 2) $T=T'$.
\item If we are in case 1)  we must have $\zeta  (E)\neq 0$.
\item 
The element $g$ is either an edge or it is of the form $-2e_i$ for some index $i$.  In this case     the graph is not allowable since  we found the desired pair of Proposition \ref{ilpunto0}.

\item $E$ is    either a red edge of the form $-e_i-e_j$ with
$i,  j$ of the same value of $\zeta$ or  a black edge of the form $-e_i+e_j$
with $i,  j$ of the opposite  value of $\zeta$.\end{enumerate}

\end{lemma}
\begin{proof}
 We have that $g=m_k\circ m_{k-1}\circ \cdots \circ m_1 $ so $E=\sum_i \delta_im_i=\sum_{i=1}^k\alpha_ie_i$.
  
1)\quad If we are in case 2) then, by Lemma \ref{zita} ii), $2E$ is a linear combination of the $\ell_i$ with integer coefficients.  Adding to this $-2E=\sum_i -2\delta_im_i$ we have  a new relation containing edges not supported in $T'$  contradicting the hypotheses. \smallskip

2)\quad If we are in case 1)  we must have $\zeta  (E)\neq 0$ otherwise, by Lemma \ref{zita} i),   $E$ is in the span of the edges $\ell_i$ and we have another relation among the edges of $T$ contradicting minimality.    \smallskip

3)\quad   Let   $U$  be the encoding graph of the  edges   $m_i\in S$ and $V$ its vertices. We have $|E|\subset [1,  2,
\ldots,  k]\cap V$, where  by $|L|$  we denote   the support   of a vector $L=\sum_a\beta_a e_a$, that is the set of indices $a$ appearing in $L$.
\smallskip

We claim that  $U$  is  connected, in fact  if $U=\bigcup U_i$ with $U_i$ connected  we    decompose $E=\sum_iE_i$ where $E_i$  is the part of the linear combination of the $m_i$  with support in $U_i$. We have observed that  linear combinations  of connected components are linearly independent.  Therefore each  $E_i$ given by each component must have support  in $|E|$.  

If $U$  is  not connected we deduce the existence at least two different linear combinations  $E_1,E_2$ of edges in $\Gamma$  with support in  $[1,  2,
\ldots,  k]$, which gives a new relation, a contradiction.  \smallskip

 Next if $V\cap [1,  2,
\ldots,  k]=\{i\}$ then we must have $E=-2e_i$ and we are in case 3).\smallskip

So there are at least  two different indices $i,j$  in  $[1,  2,
\ldots,  k]\cap V$  connected  by a   minimal simple path in $U$.  By Lemma \ref{icamm}   a linear combination $L$ of the edges $m_i\in S$ is an edge $E'$ supported in $[1,  2,
\ldots,  k]\cap V$. But then this edge must be  equal to $E$  since otherwise we have another relation for $\Gamma$ by Lemma \ref{zita} i) and 3) is proved.

As for 4  one must have $E$ linearly independent from the space spanned by the vectors of the relation so the statement follows again from Lemma  \ref{zita}.   \end{proof} 
\smallskip

Since $\Gamma$ is a full graph, the edge $E$ joining $p,q$ is in $\Gamma$.  If $S$ is not $E$, that is it is a path with at least two edges we    construct a new maximal tree $\tilde T $  in $\Gamma$ by replacing the  last edge of the path $S$ with the  edge $E$. 
\begin{lemma}\label{comt} Either $\Gamma$ is not allowable or 
$\tilde T =T'\cup E$.\smallskip

The encoding graph of $\tilde T$ is the encoding graph of the relation which is an even circuit plus the edge $E$ which separates this circuit in two odd circuits.
\end{lemma}
\begin{proof} If $T'\cup E$ is a tree then it must be equal  to $\tilde T$ by the assumption  of minimality. 

If $T'\cup E$ is not a tree  we can repeat the argument of the previous Lemma and find either a not allowable graph or a new $E'$  linear combination of the edges $\ell_i$. 

Since the span  of the edges $\ell_i$  is of codimension 1 in the span of the vectors $e_i,\ i=1,\cdots,k$ (Lemma  \ref{zita}) we have that $E,E'$ are linearly dependent modulo the span of the $\ell_i$. This generates a new relation and so a contradiction.\smallskip

 The circuits we generate in the encoding graph are odd since otherwise we would have a second even circuit and a new relation.\end{proof}

  \medskip

  {\bf Example}
$$ I)\quad\xymatrix{       & 1    \ar@{=}[r] &10\ar@{=}[r]    &9  \ar@{-}[r]  &8\ar@{=}[r] &7      && \\   &2 \ar@{=}[u] \ar@{=}[r]   &3\ar@{ =}[r] \ar@{ ->}[u]    &4\ar@{-}[r]  &5\ar@{=}[r] &6\ar@{=}[u]      &&  }$$
\begin{equation}
\label{dude}II)\quad \xymatrix{       & 1    \ar@{=}[r] &10\ar@{=}[r]    &9  \ar@{=}[r]  &8\ar@{-}[r] &7      && \\   &2 \ar@{=}[u] \ar@{=}[r]   &3\ar@{ -}[r] \ar@{ =}[ru]    &4\ar@{=}[r]  &5\ar@{=}[r] &6\ar@{=}[u]      &&  }\end{equation}

\begin{proposition}\label{icin}
Thus we have 5 possible pictures for the encoding graph of $T$. 
\begin{enumerate}[1]\item It is an even circuit.
\item It is a doubly odd circuit $ABC$ and $B\neq\emptyset$.
\item   It is a doubly odd circuit $AC$ and $B=\emptyset$.
\item It is an even circuit plus a black edge dividing it in two odd circuits.
\item It is an even circuit plus a red edge dividing it in two odd circuits.\end{enumerate}
\end{proposition} In the first 3 cases the encoding graph of $T$ coincides with the encoding graph of the minimal relation, In the last 2 cases  we also have the extra edge $E$.
\begin{corollary}\label{frec}
We  can now free  the statement of  Lemma \ref{supp} from the choice of the root to be one of the $b_i$.
 In fact from Proposition \ref{icin}  each vertex of $T$  is of the form $a_i$ or $b_i$ for some $i\in  [1,2,\cdots,k]$.
\end{corollary}
\begin{remark}\label{divido}
In case 2), 3) we divide the edges in three sets $A,B,C$ where $A$ are the edges of the first circuit, $C$ the ones of the second circuit  and $B$ (possibly empty) the edges of the segment joining the two circuits.

See figures \eqref{doppio} where $B$ is formed by 4 edges and  \eqref{dops} where $B$ is empty.\medskip

In case 4)and 5) with an extra edge we divide the edges in two sets $A$, $B$ separated by the extra edge $E$. Figure \eqref{dude}.\medskip

 The encoding graphs are all connected with all vertices of valency   2 only in case 1.
 
 A vertex of valency $>2$ will be called {\em critical}. Without loss of generality, if there are critical  vertices we may assume that 1 is critical.
 
In 2), 4), 5) we have  two vertices of valency 3 and one of valency 4 in case 3). 

As for a non critical index $u$ we shall say that $u\in A$  resp. $u\in B,C$ if the two edges $\ell_{u-1},\ell_u$ are in $A$ (resp. $B,C$).
\end{remark}
 
\section{The contribution of an index $u$}
  \subsection{The strategy}
We want to exploit Formula \eqref{retr} in order to understand  the graph. We proceed as follows.  
\begin{definition}
Given a quadratic expression $Q$ in the elements $e_i$ and any index $u$  we set $e_uC_u(Q)$  to be the  sum of all terms in $Q$ which contain $e_u$ but not $e_u^2$.
\end{definition}
Notice that $C_u$  is a linear map from quadratic expressions to linear expressions in the $e_i,\ i\neq u$. By Formula \eqref{retr} we    have  $C_u(\mathcal R)=0,\ \forall u$.  We observe that only the terms $\ell_ia_i$ or $-e_ie_{i+1}$  may contribute to  $C_u(\mathcal R)$ hence:
\begin{equation}\label{cur}
C_u(\mathcal R)=C_u\left(\sum_{i\,|\,\vartheta_i=-1} \delta_i   (-\ell_i a_i + e_ie_{i+1} )+\sum_{i\,|\,\vartheta_i= 1}  \delta_i \sigma_i(  e_ie_{i+1}+ \ell_i a_i )\right)=0.
\end{equation}

 We choose an index $u$  of valency 2, which appears thus only  in $\ell_{u-1}=\vartheta_{u-1}e_{u-1}-e_u$ and  in $\ell_{u }=\vartheta_{u }e_{u }-e_{u+1}$. This is any index in case 1)  of Proposition \ref{icin}  with no extra edge while it excludes the  {\em critical indices } in the other cases (see Remarks \ref{crii} and \ref{divido}). 
 
 In particular by our conventions we take $u\neq 1$.\smallskip
 
 \begin{remark}\label{prov}
It is possible that $u-1$ or $u+1$ or both  are critical, then the formula  for  $\ell_{u-1},\ \ell_{u+1}$ has to be interpreted as in Remark \ref{dw} 1).
\end{remark}
 \begin{definition}\label{Su}
If $u$ is a non critical index   denote by  $S_u$ the segment generated by the two edges $\ell_{u-1}, \ell_u $ in the tree $T$.
\end{definition}
 We now choose the root $r$ so that the segment $S_u$, generated by the two edges $\ell_{u-1}, \ell_u $,  appears as follows:
  \begin{equation}
\label{esub} \quad  \xymatrix{   r    \ar@{-}[r]^{\ell_u} &s\ar@{--}[rr]^{\bar a_{u-1}} &&  y \ar@{-}[r]^{\ell_{u-1}}  & x_{u-1}   } .
\end{equation} 
Depending on the color and for black edges the orientation, we have 9 different possibilities:
\begin{align*}\label{le9}
 \xymatrix{   r    \ar@{->}[r]^{\ell_u} & \ldots   \ar@{->}[r]^{\ell_{u-1}}  & x_{u-1}   }; \quad&
 \xymatrix{   r    \ar@{->}[r]^{\ell_u} & \ldots   \ar@{<-}[r]^{\ell_{u-1}}  & x_{u-1}   }; &
 \xymatrix{   r    \ar@{<-}[r]^{\ell_u} & \ldots  \ar@{->}[r]^{\ell_{u-1}}  & x_{u-1}   } ;&\\
  \xymatrix{   r    \ar@{<-}[r]^{\ell_u} & \ldots   \ar@{<-}[r]^{\ell_{u-1}}  & x_{u-1}   } ;\quad& 
  \xymatrix{   r    \ar@{=}[r]^{\ell_u} & \ldots   \ar@{->}[r]^{\ell_{u-1}}  & x_{u-1}   } ;&
 \xymatrix{   r    \ar@{=}[r]^{\ell_u} & \ldots   \ar@{<-}[r]^{\ell_{u-1}}  & x_{u-1}   } ;&
\\
 \xymatrix{   r    \ar@{->}[r]^{\ell_u} & \ldots   \ar@{=}[r]^{\ell_{u-1}}  & x_{u-1}   }; \quad& 
 \xymatrix{   r    \ar@{<-}[r]^{\ell_u} & \ldots   \ar@{=}[r]^{\ell_{u-1}}  & x_{u-1}   } ;& 
 \xymatrix{   r    \ar@{=}[r]^{\ell_u} & \ldots   \ar@{=}[r]^{\ell_{u-1}}  & x_{u-1}   } \\
\end{align*} When we   add the color $\sigma_{u-1} $ of the vertex $x_{u-1} $ we have 18 cases to treat with $x_{u-1}= a _{u-1},\ b_{u-1}$.\smallskip

  \begin{definition}\label{icolo} 1)\quad The choice of a root $r$ in $T$  induces a partial order in the edges and vertices where $a\preceq  b$ means that $a$ is in the segment joining $r$ to $b$ and possibly it is $b$. $a\prec   b$  means $a\preceq  b,\ a\neq b$.

2)\quad By $\sigma_\ell$ we denote  the color of the endpoint $v_\ell$ of the segment starting from the root and ending with $\ell$,  and for a vertex $v$  by $\sigma_v$ we denote  its  color, ($\sigma_\ell=\sigma_{v_\ell}$).
\end{definition}
\begin{theorem}\label{ilve} Given an edge $\ell_0$ we have
\begin{equation}\label{ukvef}
v:=v_{\ell_0}=\sigma_{\ell_0}  \sum_{\ell\preceq v}\sigma_\ell\lambda_\ell\ell =\sigma_{v}  \sum_{\ell\preceq v}\sigma_\ell\lambda_\ell\ell.
\end{equation}
\end{theorem}
\begin{proof}
By induction.  Il only one edge $\ell =\ell_0
$ precedes $v$   then $    v=\lambda_\ell\ell=\sigma_\ell^2\lambda_\ell\ell$. Otherwise  $\ell_0$   ends in $v$ and originates in $w\prec v$. 

We have $\sigma_v=  \sigma_w$ if $\ell_0$ is black and, by induction $$v=\lambda_{\ell_0}\ell_0+w=\lambda_{\ell_0}\ell_0+\sigma_v  \sum_{\ell\preceq w}\sigma_\ell\lambda_\ell\ell=\sigma_v \sum_{\ell\preceq v}\sigma_\ell\lambda_\ell\ell$$
If $\ell_0$ is red we have $\sigma_v=  -\sigma_w,\ \lambda_{\ell_0}=1$   and, by induction $$v=\lambda_{\ell_0}\ell_0-w=\lambda_{\ell_0}\ell_0+\sigma_v \sum_{\ell\preceq w}\sigma_\ell\lambda_\ell\ell=\sigma_v\sum_{\ell\preceq v}\sigma_\ell\lambda_\ell\ell$$\end{proof}
\begin{corollary}\label{ilve1} If  $w\prec v$  we have
\begin{equation}\label{ukveg}
v =\sigma_{v} ( \sum_{w\preceq \ell\preceq v}\sigma_\ell\lambda_\ell\ell+ \sigma_{w}w) .
\end{equation}
\end{corollary}\begin{proof}
Split  Formula  \eqref{ukvef}
$$v=\sigma_{v}  \sum_{\ell\preceq v}\sigma_\ell\lambda_\ell\ell=\sigma_{v}  \sum_{w\preceq \ell\preceq v}\sigma_\ell\lambda_\ell\ell+ \sigma_{v}  \sum_{\ell\preceq w}\sigma_\ell\lambda_\ell\ell=\sigma_{v}  \sum_{w\preceq \ell\preceq v}\sigma_\ell\lambda_\ell\ell+ \sigma_{v} \sigma_{w}  w $$
\end{proof}
We  write $ \mathcal R = \mathcal R'+\mathcal R''$ (Formula \eqref{retr}) and separately compute the contributions of 
$$ \mathcal R':=\sum_{i\,|\,\vartheta_i=-1} \delta_i    e_ie_{i+1}  +\sum_{i\,|\,\vartheta_i= 1}  \delta_i \sigma_i e_ie_{i+1}   ,\quad \mathcal R'':=-\sum_{i\,|\,\vartheta_i=-1} \delta_i     \ell_i a_i + \sum_{i\,|\,\vartheta_i= 1}  \delta_i \sigma_i  \ell_i a_i\,,$$Set $\gamma_i=-\delta_i$  if  $\theta_i=-1$  and  $\gamma_i=\sigma_i\delta_i$  if  $\theta_i=1$ so  $\mathcal R''=\sum_{i } \gamma_i   \ell_i a_i$.

$$C_u(\mathcal R)=  C_u(\mathcal R') +C_u(\mathcal R''),\  \mathcal R'' =\sum_{i } \gamma_i    \ell_i a_i .$$\medskip

We need the following formulas for the elements $a_j$, with color $\sigma_j$, easily proved from Theorem \ref{ilve}.
The notations are those of Definition \ref{biai}:\begin{equation} 
\label{La}a_j=\begin{cases}\begin{matrix}
1)\quad -\sum_{\ell\preceq \ell_j}\sigma_\ell\lambda_\ell\ell  ,&\sigma_j=-1, &\ell_j\ \text{red}\quad \\2)\quad
 - \sum_{\ell\prec  \ell_j}\sigma_\ell\lambda_\ell\ell  ,&\sigma_j= 1,\ \ &\ell_j\ \text{red}\quad \\3)\quad
\sigma_j\sum_{\ell\preceq \ell_j}\sigma_\ell\lambda_\ell\ell , &\lambda_j= 1,\ \ &\ell_j\ \text{black}\\4)\quad
\sigma_j\sum_{\ell\prec  \ell_j}\sigma_\ell\lambda_\ell\ell , &\lambda_j= -1, &\ell_j\ \text{black} 
\end{matrix}
\end{cases}
\end{equation}
\begin{proof} From Formula \eqref{ukvef}  let $v,w$ be the two end points of $\ell_j$; we have 4 cases due to the definition of $a_j,\ \sigma_j,v$.  
If $\ell_j$  is red (that is $\theta_j=-1$) or if  it is black (that is $\theta_j=1$)  and we have $\lambda_j=1$, then $a_j=v$  these are cases 1), 3). Otherwise $a_j=w$.  In case 2), $\ell_j$ red  and $\sigma_j=1$  we have   $w=  - \sum_{\ell\prec  w}\sigma_\ell\lambda_\ell\ell=  - \sum_{\ell\prec  \ell_j}\sigma_\ell\lambda_\ell\ell$ since $\sigma_w=-1$. In case 4) we have $\sigma_v=\sigma_w$  and the Formula holds.\end{proof}
\subsection{Computations of  $C_u$}

 \medskip

If $i\neq u-1,u$ set $\mu_u(i)$  to be the coefficient of $e_u$ in $a_i$. 
\begin{proposition}\label{vamu}
The value of  $\mu_u(i)$  depends upon the relative position  of  the edge $\ell_i$ with respect to the segment $S_u$.  We have 4 different possibilities, cf. Definition \ref{colo} 1).
\begin{enumerate}\item $\ell_u\not\prec \ell_i$.
\item $\ell_u  \prec \ell_i\prec\ell_{u-1}$.
\item $\ell_{u-1} \prec \ell_i $.
\item $\ell_u \prec \ell_i\not \prec\ell_{u-1}$ and  $\ell_{u-1}\not \prec \ell_i $.\end{enumerate}
In case i)   $\mu_u(i)=0.$ In case ii) and iv)  $\mu_u(i)=\pm 1.$  In case iii)   $\mu_u(i)=0 $ if $\bar L_u=0$  otherwise 
$\mu_u(i)=\pm 2$

\end{proposition}
\begin{proof}
The statements follow from the next Corollary \ref{leduei1} and Corollary \ref{ilve1}. 
\end{proof}

Then
\begin{lemma}\label{primifa} 
If $i\neq u-1,u$ we have $C_u(\ell_i a_i )=\mu_u(i)\ell_i.$\smallskip

The contribution $C_u(\mathcal R')$ depends on the two colors $\theta_{u-1},\ \theta_{u }$ of $\ell_{u-1},\ell_u$ (and $\delta_u=\theta_u \delta_{u-1} $ see Remark   \ref{crii}, Formula \eqref{boBf}) according to the following table:\smallskip

\begin{equation}\label{Rp}
\begin{matrix}
colors\ of\ u-1,u&&& contribution\ of\ \mathcal R'\\
rr  &\delta_{u-1}=-\delta_{u}&& \delta_{u-1}e_{u-1}+\delta_{u }e_{u +1}= -\delta_{u }[e_{u-1}- e_{u+1 }]\\
rb    &\delta_{u-1}=\delta_{u}\,\quad && \delta_{u-1}e_{u-1}+ \delta_{u }e_{u +1}\ \ = \delta_{u }[e_{u-1}+ e_{u +1}]\\
br  &\delta_{u-1}=-\delta_{u}&&  \delta_{u-1}\sigma_{u-1}e_{u-1}+ \delta_{u }e_{u +1}= -\delta_{u}[\sigma_{u-1}e_{u-1}-  e_{u +1}]\\
bb   &\delta_{u-1}=\delta_{u}\, \quad &&  \delta_{u-1}\sigma_{u-1}e_{u-1}+  \delta_{u }e_{u +1}\ \ = \delta_{u }[\sigma_{u-1}e_{u-1}+ e_{u+1 }]
\end{matrix}
\end{equation}
\end{lemma}
\begin{proof}
The first statement is clear since the edge $\ell_i$ does not contain the term $e_u$.  

For the second  we see  that the contribution  to $C_u(\mathcal R')$ comes from the two terms $e_{u-1}e_u,\ e_ue_{u+1}$. 

 The term  $e_{u-1}e_u$ if  $\theta_{u-1}=-1$, i.e. $\ell_{u-1}$ is red,   appears from $C_u( \delta_{u-1} e_{u-1}e_u)= \delta_{u-1} e_{u-1}$. 
 
 If $\theta_{u-1}= 1$, i.e. $\ell_{u-1}$ is black,     appears from $C_u( \sigma_{u-1}\delta_{u-1} e_{u-1}e_u)= \sigma_{u-1}\delta_{u-1} e_{u-1}$.\smallskip

The term  $e_ue_{u+1}$, if  $\theta_{u }=-1$, i.e. $\ell_{u }$ is red,   gives rise to $C_u( \delta_{u } e_ue_{u+1})= \delta_{u } e_{u+1}$.

 If   $\theta_{u }= 1$, i.e. $\ell_{u }$ is black,   gives rise to $C_u( \sigma_{u }\delta_{u } e_ue_{u+1})= \sigma_{u }\delta_{u } e_{u+1}$.

We then use the fact that $\delta_u=\delta_{u-1}$ if $\delta_u$ is black, while $\delta_u=-\delta_{u-1}$ if $\delta_u$ is red.\end{proof}
We thus write
\begin{equation}\label{laff}
0=C_u(\mathcal R)=- \sum_{i\,|\,\vartheta_i=-1,\ i\neq u-1,u} \delta_i \mu_u(i)\ell_i +\sum_{i\,|\,\vartheta_i= 1,\ i\neq u-1,u}   \delta_i \sigma_i  \mu_u(i)\ell_i+L_u
\end{equation} where $L_u$  is the contribution from $C_u(\mathcal R')  $, which we have computed in the Table \eqref{Rp},  plus the contribution    from the terms associated to $a_{u-1}\ell_{u-1},\ a_u\ell_u  $.  \begin{definition}\label{abau}
In $a_{u-1}$ given by   Formula \eqref{La}, consider the part $\bar a_{u-1}$ of the sum   formed by the edges $\ell_i,\ \ell_u\prec\ell_i\prec\ell_{u-1}$. 

 Write $a_{u-1}=\bar a_{u-1}+\tilde a_{u-1}$  we have   $C_u(\bar a_{u-1}\ell_{u-1})=- \bar a_{u-1}.$
\end{definition} Recall we have set $\gamma_i=-\delta_i$  if  $\theta_i=-1$  and  $\gamma_i=\sigma_i\delta_i$  if  $\theta_i=1$ so  $\mathcal R''=\sum_{i } \gamma_i   \ell_i a_i$.
\begin{equation}\label{LuR}
L_u=C_u(\mathcal R') + \gamma_{u-1} C_u(\bar a_{u-1}\ell_{u-1})+ \gamma_{u-1} C_u(\tilde a_{u-1}\ell_{u-1})+\gamma_u C_u(a_u\ell_u)
\end{equation}$$=C_u(\mathcal R') - \gamma_{u-1}  \bar a_{u-1} + \gamma_{u-1} C_u(\tilde a_{u-1}\ell_{u-1})+\gamma_u C_u(a_u\ell_u)=- \gamma_{u-1}  \bar a_{u-1} +\bar L_u.
 $$\smallskip

   The value of $\bar L_u$ depends upon 3 facts, 1) the two colors of  $\ell_{u-1},\ell_u $. 2) The orientation $\lambda$  of the edges  $\ell_{u-1},\ell_u $ which are black. 3)  The color $\sigma_{u-1}$ of  $x_{u-1}$. 
   
    We thus obtain 18 different cases described in \S \ref{18}, see the pictures after \eqref{esub}. \smallskip
   
   The final computation is summarized in Proposition \ref{leduei}. The proof is very lengthy due to the case analysis but otherwise straightforward.

\subsubsection{The contribution of  $\gamma_uC_u( a_u\ell_u ) $ to Formula \eqref{LuR}}
   If $\ell_u=-e_u-e_{u+1}$ is red we have $a_u=\ell_u,\ \gamma_u=-\delta_u$ and $C_u( \ell_ua_u)=2 e_{u+1}$. 
   
    If  $\ell_u=e_u-e_{u+1}$ is black we have $\sigma_u=1,\ \gamma_u=\delta_u$, if $\lambda_u=1$ we have $a_u=\ell_u$ and $C_u( \ell_ua_u)=-2  e_{u+1}$. If $\lambda_u=-1$ we have $a_u=0$ and $C_u(\ell_ua_u)=0.$  
   
   Summarizing:\smallskip

   \begin{equation}
\label{ellu} \begin{matrix}
\quad\quad  C_u(\gamma_u\ell_ua_u)=-2\delta_ue_{u+1},\quad &\ell_u\quad \text{is red}\quad&\\
C_u(\gamma_u\ell_ua_u)=-2\delta_u e_{u+1},\quad &\ell_u\quad \text{is black}& \lambda_u=1\quad\\
C_u(\gamma_u\ell_ua_u)=0,\quad\quad \quad \quad &\ell_u\quad \text{is black}& \lambda_u=-1.
\end{matrix}\end{equation}
\subsubsection{The contribution of  $ \gamma_{u-1} C_u(\tilde a_{u-1}\ell_{u-1})$ }
    
   The vertex $a_{u-1}$ is one of the two end points $ y, \  x_{u-1}$  of   the edge $\ell_{u-1}$.

 We have $a_{u-1}=\bar a_{u-1}+\tilde a_{u-1}$, see Figure \eqref{esub}, where  by   Formula \eqref{La}

\begin{equation}
\label{La1}\tilde a_{u-1}=\begin{cases}\begin{matrix}
-\sigma_{u}\lambda_u\ell_u  +\ell_{u-1}  ,\ &\text{ if}\ \sigma_{u-1}=-1,\quad  &\ell_{u-1}\ \quad\text{red}\quad\\
-\sigma_{u}\lambda_u\ell_u  ,\quad\ &\text{if}\ \sigma_{u-1}= 1,\ \quad&\ell_{u-1}\  \quad\text{red}\quad\\
\sigma_{u-1}\sigma_{u}\lambda_u\ell_u  +\ell_{u-1}  ,\quad\ &\text{if}\  \lambda_{u-1}= 1,\ \quad&\ell_{u-1}\  \quad\text{black}\\
\sigma_{u-1}\sigma_{u}\lambda_u\ell_u  ,\quad\ &\text{if}\  \lambda_{u-1}= -1,\ &\ell_{u-1}\   \quad\text{black}\\
\end{matrix}
\end{cases}
\end{equation}

%

The contribution $L_u$   is split in $ \gamma_{u-1} C_u(\bar a_{u-1}\ell_{u-1})$   and and a final term $\bar L_u=C_u(\mathcal R') + \gamma_{u-1} C_u(\tilde a_{u-1}\ell_{u-1})+\gamma_u C_u(a_u\ell_u)$  coming  from  $\tilde a_{u-1}\ell_{u-1},\ a_u\ell_u  $.\smallskip

We are left  to compute $C_u(\tilde a_{u-1}\ell_{u-1})$ and  we need:
    \begin{equation}\label{iCu}
C_u(\ell_{u-1}\ell_u)=\vartheta_{u-1}\vartheta_u e_{u-1}+e_{u+1},\ \quad C_u(\ell_{u-1}^2)=-\vartheta_{u-1}2e_{u-1}.
\end{equation} 
    $$C_u(\ell_{u-1}\tilde a_{u-1}) \stackrel{\eqref{La1}}=\begin{cases}\begin{matrix}
-\sigma_{u}\lambda_uC_u(\ell_{u-1}\ell_u) +C_u(\ell_{u-1}^2)  ,\ &\ \sigma_{u-1}=-1,\ &\ell_{u-1}\ \text{red}\quad \\
-\sigma_{u}\lambda_uC_u(\ell_{u-1}\ell_u)  ,\quad & \!\! \sigma_{u-1}= 1,\   &\ell_{u-1}\ \text{red}\quad \\
\sigma_{u-1}\sigma_{u}\lambda_uC_u(\ell_{u-1}\ell_u)  +C_u(\ell_{u-1}^2)  , \ &\  \lambda_{u-1}= 1,\ \quad &\ell_{u-1}\ \text{black}\\
\sigma_{u-1}\sigma_{u}\lambda_uC_u(\ell_{u-1}\ell_u)  , \ &\  \lambda_{u-1}= -1,\ &\ell_{u-1}\ \text{black}
\end{matrix}
\end{cases} $$ gives   ($\gamma_i=-\delta_i$  if  $\theta_i=-1$  and  $\gamma_i=\sigma_i\delta_i$  if  $\theta_i=1$) $\gamma_{u-1} C_u(\ell_{u-1}\tilde a_{u-1})\stackrel{\eqref{iCu}}=$
$$=\begin{cases}\begin{matrix}
\delta_{u-1}\sigma_{u}\lambda_u(-\vartheta_u e_{u-1}+e_{u+1}) -2\delta_{u-1}e_{u-1}  ,\ & \sigma_{u-1}=-1,\   &\ell_{u-1}\  \text{red}\quad \\
\delta_{u-1}\sigma_{u}\lambda_u(-\vartheta_u e_{u-1}+e_{u+1})  ,\quad\ & \sigma_{u-1}= 1,\ \ \ &\ell_{u-1}\  \text{red}\quad \\
\delta_{u-1} \sigma_{u}\lambda_u( \vartheta_u e_{u-1}+e_{u+1})  -2\delta_{u-1}\sigma_{u-1}e_{u-1}  ,\  & \lambda_{u-1}= 1,\ \quad &\ell_{u-1}\  \text{black}\\
\delta_{u-1}\sigma_{u}\lambda_u( \vartheta_u e_{u-1}+e_{u+1})  , \ & \lambda_{u-1}= -1,\ &\ell_{u-1}\  \text{black} 
\end{matrix}
\end{cases} $$ 

   If $\ell_{u-1}$ is red  from the previous Formula we deduce $\gamma_{u-1} C_u(\ell_{u-1}\tilde a_{u-1})=$ \begin{equation}
\label{ru-1}=\begin{cases}\begin{matrix}
-\delta_{u-1}( e_{u+1}+3e_{u-1}),\ &\  \sigma_{u-1}=-1,   \ &\ell_{u} &\text{red}\quad \\
 -\delta_{u-1}( e_{u+1}+e_{u-1}) ,\quad\ &\  \sigma_{u-1}= 1,\  \quad &\ell_{u} &\text{red}\quad \\
 \delta_{u-1} \lambda_u[e_{u+1}-e_{u-1}] -2\delta_{u-1} e_{u-1} ,\ &\  \sigma_{u-1}=-1,  \ &\ell_{u} &\text{black}\\
\delta_{u-1} \lambda_u [e_{u+1}-e_{u-1}] ,\quad\ &\  \sigma_{u-1}= 1,\  \quad  &\ell_{u} &\text{black}
\end{matrix}
 \end{cases}
\end{equation}
   If $\ell_{u-1}$ is black      the contribution  to     $L_u$   of
   $C_u(\ell_{u-1}\tilde a_{u-1})= $    
   \begin{equation}
\label{bu-1}=\begin{cases}\begin{matrix}
-\delta_{u-1}\sigma_{u-1}  ( -e_{u-1}+e_{u+1})  -2\delta_{u-1}\sigma_{u-1}e_{u-1}  ,\  & \lambda_{u-1}= 1,\ \quad &\ell_{u }\  \text{red}\\
-\delta_{u-1}  ( - e_{u-1}+e_{u+1})  , \ & \lambda_{u-1}= -1,\ &\ell_{u }\  \text{red} 
\\
\delta_{u-1} \lambda_u(   e_{u-1}+e_{u+1})  -2\delta_{u-1}\sigma_{u-1}e_{u-1}  ,\  & \lambda_{u-1}= 1,\ \quad &\ell_{u }\  \text{black}\\
\delta_{u-1} \lambda_u(  e_{u-1}+e_{u+1})  , \ & \lambda_{u-1}= -1,\ &\ell_{u }\  \text{black} 
\end{matrix}
\end{cases} \end{equation}

%
%
%
%
 {\bf Warning}\quad We have been a bit sloppy since we need to recall  Remark \ref{prov}, we should have written  $f_i$  instead of   $e_i$ so that if an index $i$ is not critical  $f_i=e_i$  and if it is critical  $f_i$ is given by the Formulas of  Remark \ref{prov}.  This hopefully should not  generate confusion.
\subsection{The 18  cases\label{18} for the value of $\bar L_u $} So now we expand  $ \bar L_u $ as sum of the 3 terms, by Formula \eqref{LuR}, from    Formulas \eqref{Rp},\eqref{ellu}, and \eqref{ru-1} or \eqref{bu-1}.\medskip

\noindent 1)\  $\ell_{u-1},\ell_u$ both red $\sigma_{u-1}=1,\ \delta_u=-\delta_{u-1}.$
$$-\delta_{u }[e_{u-1}- e_{u+1 }]-2\delta_ue_{u+1}+\delta_u(e_{u+1}+e_{u-1})\ =0.$$
2)\  $\ell_{u-1},\ell_u$ both red $\sigma_{u-1}=-1,\ \delta_u=-\delta_{u-1}.$
$$-\delta_{u }[e_{u-1}- e_{u+1 }]-2\delta_ue_{u+1}+\delta_u[e_{u+1}+3e_{u-1}]= 2\delta_u e_{u-1}.$$
3)\  $\ell_{u-1}$   red, $\ell_u$ black $\sigma_{u-1}=1,\lambda_u=1,\ \delta_u= \delta_{u-1}.$
$$ \delta_{u }[e_{u-1}+e_{u +1}]-2\delta_u e_{u+1}+\delta_u   [e_{u+1}-e_{u-1}]=0$$
4)\  $\ell_{u-1}$   red, $\ell_u$ black $\sigma_{u-1}=-1,\lambda_u=1,\ \delta_u= \delta_{u-1}.$
$$ \delta_{u }[e_{u-1}+e_{u +1}]-2\delta_u e_{u+1}+\delta_u [e_{u+1}-e_{u-1}] +2\delta_ue_{u-1}=2\delta_u e_{u-1}$$
5)\  $\ell_{u-1}$   red, $\ell_u$ black $\sigma_{u-1}=1,\lambda_u=-1,\ \delta_u= \delta_{u-1}.$
$$ \delta_{u }[e_{u-1}+e_{u +1}]-\delta_u  [e_{u+1}-e_{u-1}]=2\delta_u e_{u-1}$$
6)\  $\ell_{u-1}$   red, $\ell_u$ black $\sigma_{u-1}=-1,\lambda_u=-1,\ \delta_u= \delta_{u-1}.$
$$ \delta_{u }[e_{u-1}+e_{u +1}]-\delta_u  [e_{u+1}-e_{u-1}] -2\delta_ue_{u-1}=0$$
7)\  $\ell_{u-1}$   black, $\ell_u$ red $\sigma_{u-1}=1,\lambda_{u-1}=1,\ \delta_u=-\delta_{u-1}.$
$$-\delta_{u}[ e_{u-1}-  e_{u +1}] -2\delta_ue_{u+1}+\delta_{u}  [e_{u+1}-e_{u-1} ] +2\delta_{u} e_{u-1}=0$$
8)\  $\ell_{u-1}$   black, $\ell_u$ red $\sigma_{u-1}=-1,\lambda_{u-1}=1,\ \delta_u=-\delta_{u-1}.$
$$-\delta_{u}[-e_{u-1}-  e_{u +1}]-2\delta_ue_{u+1}-\delta_{u}  [e_{u+1}-e_{u-1} ] -2\delta_{u} e_{u-1}=-2\delta_ue_{u+1}.$$
9)\  $\ell_{u-1}$   black, $\ell_u$ red $\sigma_{u-1}=1,\lambda_{u-1}=-1,\ \delta_u=-\delta_{u-1}.$
$$-\delta_{u}[ e_{u-1}-  e_{u +1}]-2\delta_ue_{u+1}-\delta_{u}  [e_{u+1}-e_{u-1} ]= -2\delta_ue_{u+1} $$
10)\  $\ell_{u-1}$   black, $\ell_u$ red $\sigma_{u-1}=-1,\lambda_{u-1}=-1,\ \delta_u=-\delta_{u-1}.$
$$-\delta_{u}[-e_{u-1}-  e_{u +1}]-2\delta_ue_{u+1}+\delta_{u}  [e_{u+1}-e_{u-1} ] =0. $$
11)\  $\ell_{u-1}$,  $\ell_u$  both black,  $\sigma_{u-1}=1,\lambda_{u-1}=1,\ \lambda_u=1,\ \delta_u= \delta_{u-1}.$
$$\delta_{u }[ e_{u-1}+ e_{u+1 }]-2\delta_ue_{u+1}+\delta_u [e_{u-1}+e_{u+1} ]  -2\delta_ue_{u-1} =0$$
12)\ $\ell_{u-1}$,  $\ell_u$  both black $\sigma_{u-1}=-1,\lambda_{u-1}=1,\ \lambda_u=1,\ \delta_u= \delta_{u-1}.$
$$\delta_{u }[-e_{u-1}+ e_{u+1 }]-2\delta_ue_{u+1}+\delta_u[e_{u-1}+e_{u+1} ]  +2\delta_ue_{u-1} =2\delta_ue_{u-1}$$
13)\ $\ell_{u-1}$,  $\ell_u$  both black $\sigma_{u-1}=1,\lambda_{u-1}=-1,\ \lambda_u=1,\ \delta_u= \delta_{u-1}.$
$$ \delta_{u }[ e_{u-1}+ e_{u+1 }] -2\delta_ue_{u+1}+\delta_u [e_{u-1}+e_{u+1} ]=2\delta_u e_{u-1}$$
14)\ $\ell_{u-1}$,  $\ell_u$  both black $\sigma_{u-1}=-1,\lambda_{u-1}=-1,\ \lambda_u=1,\ \delta_u= \delta_{u-1}.$
$$\delta_{u }[-e_{u-1}+ e_{u+1 }] -2\delta_ue_{u+1}+\delta_u[e_{u-1}+e_{u+1} ]=0$$
15)\  $\ell_{u-1}$,  $\ell_u$  both black,  $\sigma_{u-1}=1,\lambda_{u-1}=1,\ \lambda_u=-1,\ \delta_u= \delta_{u-1}.$
$$\delta_{u }[ e_{u-1}+ e_{u+1 }]-\delta_u[e_{u-1}+e_{u+1} ]  -2\delta_ue_{u-1}=-2\delta_ue_{u-1} $$
16)\ $\ell_{u-1}$,  $\ell_u$  both black $\sigma_{u-1}=-1,\lambda_{u-1}=1,\ \lambda_u=-1,\ \delta_u= \delta_{u-1}.$
$$\delta_{u }[-e_{u-1}+ e_{u+1 }]-\delta_u[e_{u-1}+e_{u+1} ] +2\delta_ue_{u-1} =0$$
17)\ $\ell_{u-1}$,  $\ell_u$  both black $\sigma_{u-1}=1,\lambda_{u-1}=-1,\ \lambda_u=-1,\ \delta_u= \delta_{u-1}.$
$$\delta_{u }[ e_{u-1}+ e_{u+1 }]-\delta_u[e_{u-1}+e_{u+1} ]=0 $$
18)\ $\ell_{u-1}$,  $\ell_u$  both black $\sigma_{u-1}=-1,\lambda_{u-1}=-1,\ \lambda_u=-1,\ \delta_u= \delta_{u-1}.$
$$\delta_{u }[-e_{u-1}+ e_{u+1 }]-\delta_u[e_{u-1}+e_{u+1} ]=-2\delta_ue_{u-1}$$
By inspection we see that we have proved the following remarkable:
\begin{proposition}\label{leduei}
The contribution of $\bar L_u$ equals to 0 if and only if $\sigma_{u-1}= \lambda_{u-1}\lambda_u $. 

In this case the coefficient of $e_u$ in the end point $x_{u-1}$ of  the segment $S_u$ (defined in \eqref{esub}) is $0$.\medskip

If $\sigma_{u-1}= -\lambda_{u-1}\lambda_u $ the contribution of $\bar L_u$ equals to $\pm 2e_{u\pm 1}$. In this case the coefficient of $e_u$ in the end point $x_{u-1}$ of  the segment $S_u$ is $\pm 2$.
\end{proposition}
 \begin{proof}  The first is by inspection, as for the second we check a few cases.
 
This coefficient comes from the two contributions of  $\ell_{u-1},\ell_u$. 

They appear by $\sigma_{u-1}[\sigma_u\lambda_u\ell_u+\sigma_{u-1}\lambda_{u-1}\ell_{u-1}]$.  \smallskip

Now $\sigma_u\lambda_u\ell_u=-\ell_u=e_u+e_{u+1}$  if $\ell_u$ is red and similarly $\sigma_{u-1}\lambda_{u-1}\ell_{u-1}=e_u+e_{u-1}$  if $\ell_{u-1}$ is red and $\sigma_{u-1}=-1$.  This is case 2).   \smallskip

 If $\ell_{u-1}$ is black  then the coefficient of $e_u$ in $\sigma_{u-1}\lambda_{u-1}\ell_{u-1}$ is 1 if and only if $\sigma_{u-1}\lambda_{u-1}=-1$  and in this case this is equivalent to $\sigma_{u-1}=-\lambda_{u-1} \lambda_{u }.$  These are cases 8,9.\smallskip

Similar argument when $\ell_u$ is black.\end{proof} \begin{corollary}\label{leduei1}
If $\ell_{u-1}\prec\ell_j$ we have $\mu_u(j)=0$ if the contribution of $\bar L_u$ is 0, otherwise $\mu_u(j)=\pm 2$.
\end{corollary}
\begin{proof}
By   Formula \eqref{ukveg}
  the  vertex $a_j= \sigma_{a_j} ( \sum_{x_{u-1}\preceq \ell\preceq a_j}\sigma_\ell\lambda_\ell\ell+ \sigma_{x_{u-1}}x_{u-1})$.  
   
The edges $\ell$  with $x_{u-1}\preceq \ell\preceq a_j$
do not contain $e_u$ so  $\mu_u(j)$  equals $\pm $  the coefficient of $e_u$ in the end point $x_{u-1}$ of  the segment $S_u$.\end{proof}
Formula \eqref{laff}  can be written as  \begin{equation}\label{laff1}
\gamma_{u-1}\bar a_{u-1} =- \sum_{i\,|\,\vartheta_i=-1,\ i\neq u-1,u} \delta_i \mu_u(i)\ell_i +\sum_{i\,|\,\vartheta_i= 1,\ i\neq u-1,u}   \delta_i \sigma_i  \mu_u(i)\ell_i+\bar L_u
\end{equation} 
\begin{proposition}\label{4po}
We have 4 possibilities for $\bar a_{u-1}$ given by Definition \ref{abau}.

  If $\ell_{u-1}$ is red
\begin{equation}
\label{relr3} \begin{matrix}
1)\qquad \delta_{u-1}\bar a_{u-1}=& \sum_{i\,|\,\vartheta_i=-1,\atop i\neq u-1,u} \delta_i \mu_u(i)\ell_i  -\sum_{i\,|\,\vartheta_i= 1,\atop i\neq u-1,u}   \delta_i \sigma_i  \mu_u(i)\ell_i \\\\
2)\qquad \delta_{u-1}\bar a_{u-1}=& \sum_{i\,|\,\vartheta_i=-1,\atop i\neq u-1,u}\delta_i \mu_u(i)\ell_i  -\sum_{i\,|\,\vartheta_i= 1,\atop i\neq u-1,u}   \delta_i \sigma_i  \mu_u(i)\ell_i\pm 2\delta_ue_{u\pm 1}
\end{matrix}
\end{equation}

If $\ell_{u-1}$ is black
\begin{equation}
\label{relb3} \begin{matrix}
1)\qquad -\sigma_{u-1}\delta_{u-1}\bar a_{u-1}=& \sum_{i\,|\,\vartheta_i=-1,\atop i\neq u-1,u} \delta_i \mu_u(i)\ell_i  -\sum_{i\,|\,\vartheta_i= 1,\atop i\neq u-1,u}   \delta_i \sigma_i  \mu_u(i)\ell_i .\\\\2 )\qquad 
-\sigma_{u-1}\delta_{u-1}\bar a_{u-1}= &\sum_{i\,|\,\vartheta_i=-1,\atop i\neq u-1,u} \delta_i \mu_u(i)\ell_i  -\sum_{i\,|\,\vartheta_i= 1,\atop i\neq u-1,u}   \delta_i \sigma_i  \mu_u(i)\ell_i \pm 2\delta_ue_{u\pm 1}.
\end{matrix}
\end{equation}
\end{proposition}
\begin{proof}
Since $\gamma_{u-1}=-\delta_{u-1}$  if  $\theta_{u-1}=-1$  and  $\gamma_{u-1}=\sigma_{u-1}\delta_{u-1}$  if  $\theta_{u-1}=1$  this follows from  Formula \eqref{laff1}.
\end{proof}\section{The possible graphs}
{\em We now discuss the implications of the previous sections to the form of  the possible minimal degenerate graphs.}
\subsection{Contribution  of $\bar L_u$ equals to 0\label{coL0} } We say that $u$ is of type  I. 

By definition
\begin{equation}\label{baa}
\bar a_{u-1}=\sum_{\ell_{u }\prec\ell\prec \ell_{u-1}}\alpha_\ell \ell,\ \alpha_\ell =\pm 1
\end{equation} it is also given by the Formulas of Proposition \ref{4po}.

Recall that   $\mu_u(i),\ i\neq u-1,u$  denotes  the coefficient of $e_u$ in $a_i$.

\begin{remark}\label{uod}
Formulas  \eqref{relr3} 1)  or \eqref{relb3} 1) must coincide with $\gamma_{u-1}\sum_{\ell_{u }\prec\ell\prec \ell_{u-1}}\alpha_\ell \ell $  Formula \eqref{baa}.
\end{remark}  \begin{proposition}\label{intve} When $\bar L_u=0$  all internal vertices of $S_u$ have valency 2.
\end{proposition} \begin{proof}
 Notice that any edge $\ell_j$ comparable with $\ell_u$ and not with $\ell_{u-1}$ appears as
 \begin{equation}
\label{esu1}   \xymatrix{  &&&\ar@{-}[ld]^{\ell_j}\\ && &\\ r    \ar@{-}[r]^{\ell_u\ \quad} &c \ldots  \ldots\bullet  \ldots d \ar@{.}[ru] \ar@{-}[r]^{\ \qquad \ell_{u-1}}   & x_{u-1}   } .
\end{equation} has  $\mu_u(j)=\pm 1$, by Corollary \ref{ilve1},  so  appears in the relation \eqref{relr3} 1)  and \eqref{relb3} 1),   this is a contradiction with the definition of  $\bar a_{u-1}$ by \eqref{baa}. \smallskip
 
 Thus  if $\bar L_u=0$ no  edge  is comparable  with $\ell_u$ and not with $\ell_{u-1}$.  \end{proof} 
 
  \begin{corollary}\label{seq}
If we have a sequence  of consecutive indices  $u,u+1,u+2,\cdots, u+k$  all of type I then $\cup_{i=0}^k  S_{u+i}$  is a segment with all its internal vertices of valency 2.
\end{corollary}
\begin{proof}
By induction  $\cup_{i=0}^{k-1}  S_{u+i}$ and $S_{u+k}$  are segments with all   internal vertices of valency 2.

Now  the intersection  $\cup_{i=0}^{k-1}  S_{u+i}\cap S_{u+k}$  contains the edge  $\ell_{u+k-1}$.  Then every vertex internal  to  $\cup_{i=0}^{k-1}  S_{u+i}\cup S_{u+k}$ is internal in at least one of the two segments.\end{proof}
  \medskip
 
 {\bf Case 2)}\quad  i.e. the encoding diagram is doubly odd. Recall that in the basic relation $\mathcal R$  the coefficients  $\delta_i$ are $\pm 1$  for the edges in $A\cup C$ and $\pm 2$ for the edges in $B$.  
 \begin{proposition}\label{intveo}
In case of a doubly odd circuit $ABC$ and $\bar L_u=0$, if $u\in A\cup C$ the segment $S_u$ is all formed by elements in $A\cup C$.

If $u\in B$ the segment $S_u$ is all formed by elements in $B$
\end{proposition}
 \begin{proof} In Formula  \eqref{baa} the coefficients are all $\pm 1$  so that in the corresponding Formulas  \eqref{relr3} 1)  and \eqref{relb3} 1),  the coefficients must be either all $\pm 1$ or all $\pm 2$ by Remark \ref{uod}. This depends uniquely  on the value $\delta_{u-1}$, if $u\in A\cup C$ then $\delta_{u-1}=\pm 1$ otherwise $\delta_{u-1}=\pm 2$.
 
 \end{proof}

 {\bf Case 1)}\quad   with an extra edge $E$ and $\bar L_u=0$. \begin{proposition}\label{edE}
The edge $E$  is not in the segment  $S_u$.
\end{proposition}\begin{proof}
It is not possible that $E$  is in between  $\ell_{u-1},\ell_u$ otherwise,  by Remark \ref{uod},  $E$ would appear in the Formulas  \eqref{relr3} 1)  and \eqref{relb3} 1). But by the definition of $C_u$  in these formulas appear only the edges $\ell$ in the relation.\end{proof}  
\subsection{Some geometry of trees}  Let us collect some generalities which will be used in the course of the proof.   In all this section $T$ will be a tree,   for the moment with no further structure and later related to the Cayley graph. Sometimes it is convenient to distinguish between $T$ as a set of edges and $|T|$  as its geometric realization.

\begin{definition}\label{Tge}
Given a set $A$ of edges in $T$ let us denote  by $\langle
A\rangle$ the minimal tree contained in $T$ and containing $A$,
we call it the {\em tree generated by $A$}.
\end{definition}  The simplest trees are the {\em segments} $S$ in which no vertex has valency $>2$. In fact in a segment we have exactly two end points of valency 1 and the {\em interior points} of valency 2. The geometric realization $|T|$  of a tree $T$  is homeomorphic to a usual segment in $\R$ if and only if $T$ is a segment.
\begin{remark}\label{tws} A connected subset of a segment is a segment.

The intersection $|S_1|\cap |S_2|$ of two segments $S_1,S_2$ in   $|T|$ is either empty or  a vertex or a segment.
\end{remark}
\begin{proof} The first is clear. 
Take any two vertices $a,b$ in  $S_1\cap S_2$. The segment connecting $a,b$ in $S_1$  must coincide with  that  connecting $a,b$ in $S_1$ therefore .$S_1\cap S_2$ is connected .\end{proof}
\begin{lemma}
\label{SEP} 1)\quad If $A$ consists of 2 edges then  $\langle
A\rangle$  is a segment,   more generally if  $A$ is the union of 2
segments $S_1,  S_2$ with the interior vertices in $A$ of valency 2 then
again $\langle A\rangle$  is a segment,  if moreover   $|S_1|\cap |S_2|\neq\emptyset$,   then  $S_1\cup S_2 =\langle  S_1,
S_2\rangle$ and all its interior vertices have valency 2.
\smallskip

 \quad If we only assume that $ S_2$ has interior vertices of valency 2   but we also assume that $|S_1|\cap |S_2|\neq\emptyset$   then

2)    $\langle S_1,  S_2\rangle=S_1\cup S_2$ and it  is a segment.   
\end{lemma}\begin{proof}
1)\quad  If $|S_1|\cap |S_2|$,    is empty,   there is a
unique segment in $\langle  S_1,
S_2\rangle$ joining  two end points   and the statement is clear.
\smallskip
 If $|S_1|\cap |S_2|$ is a vertex then it is either an end point of both and then  $|S_1|\cup |S_2|$ is a segment or it must be an interior point of at least one of the two with valency $>2$. The picture explains  what is happening.
 
 \begin{equation}\label{intt}
I)\quad\xymatrix{       &     \ar@{-}[rd] &   &    & \ar@{-}[ld] &       && \\   & \ar@{-}[r]   & \ar@{ -}[r]      & \ar@{-}[rr]  & &        &&  }
\end{equation}   
 
    If $|S_1|\cap |S_2|$ is a  segment with end points $a,b$, then if $a$ is an interior point of $S_1$ it cannot be an  interior point of $S_2$ since it has valency 2.  Similar reasoning for $b$.

2)\quad If $A=S_1\cap
S_2$  is  a segment.   Unless $S_2\subset S_1$  one of the end points $a$ of $A$
is an internal vertex of $S_1$,   since this has valency 2 this is
possible only if  $a$  is an end point of $S_1$,   if also the
other end point of $A$ is an internal vertex of $S_1$ the same
argument shows that $S_1\subset S_2$.  The final case is that the
other end of $A$  is  also an end point of $S_2$ and then the
statement is clear. \end{proof}
\begin{proposition}\label{unsegp}
 Take   segments $S_1,S_2,\cdots, S_k$  in $T$ which all  contain an edge $E$ and $S_i\cap S_j$ is a segment. Then  $\cup_{i=1}^kS_i$ is a segment.    

\end{proposition}
\begin{proof}
By induction $S:=\cup_{i=1}^{k-1}S_i$ is a segment with one end point  an end point say in $S_1$ and the other an end point of $S_2$.    The intersection $S\cap S_k$ is a segment containing $S_k\cap S_1$ and  $S_k\cap S_2$.   If $S_k$ is contained in one of these two intersections we are done. Otherwise we have 4 possibilities,    $S_k\cap S_1$ is a segment initial in $S_1$,  then clearly $S_k\cup S$ is a segment. $S_k\cap S_2$ is a segment final  in $S_2$,  then clearly $S_k\cup S$ is a segment.  The remaining case $S_k\subset S$.\end{proof}\subsection{All non critical indices are of type I}
\begin{theorem}\label{lu0}
A)\quad In case of an even circuit where all non critical indices are of type I we have that $T$  is a segment.\smallskip

B)\quad In case of a doubly odd circuit where all non critical indices are of type I we have that the unions $$S_{A }:=\cup_{a\in A}S_a,\  S_{B }:=\cup_{b\in A}S_b,\  S_{C }:=\cup_{c\in A}S_c,\  $$  are segments with internal vertices of valency 2.\smallskip

  $S_{B}$    is formed by all the edges in $B$. $S_A$ and $S_C$ are either formed of edges all in $A$ and all in $C$  or $S_{A\cup C}$ is a segment.
\end{theorem} 
\begin{proof} A) follows from Corollary  \ref{seq} of Proposition \ref{intve}.

B)\quad We apply again Corollary  \ref{seq} of Proposition \ref{intve}.
If two   segments both with  internal vertices of valency 2 have an edge in common  then their union is a segment   with  internal vertices of valency 2.  This applies recursively to the segments $S_u, S_{u+1}$  where $u$  runs in either $A,B,C$.   It also applies to  $S_A, S_C$ in case they have an edge in common. 

We then apply Proposition \ref{intveo} which tells us that $S_B$ is formed entirely by edges in $B$.\end{proof}
In this case we  have the following possibilities for the tree $T$.
$$a'): \xymatrix{    a\ar@{-}[rr]^{S_{A\cup C}  }& &b  \ar@{-}[rr]^{S_B} && c    },\quad b'):\  \xymatrix{    a\ar@{-}[rr]^{S_{A }  }& &b  \ar@{-}[rr]^{S_B} && c \ar@{-}[rr]^{S_C} &&d   }. 
$$
\begin{equation}
\label{unseg1}  \xymatrix{&& &&d&&&&&\\  c') &&a\ar@{-}[rr]^{S_{A }  }& &\ar@{-}[u]^{S_C}b  \ar@{-}[rr]^{S_B} && c     }. 
\end{equation} 
\begin{theorem}\label{lu1}
In case of a single circuit with an extra edge $E$  in which all non critical indices are of type I we have that the unions $$S_{A }:=\cup_{a\in A}S_a,\  S_{B }:=\cup_{b\in A}S_b   $$  are segments with internal vertices of valency 2.
$S_A$ is  formed of edges all in $A$ and  $S_{B}$    is formed by all the edges in $B$ and they are separted by the edge $E$.\end{theorem} 
\begin{proof} We apply Corollary  \ref{seq} of Proposition \ref{intve} as before and  Proposition \ref{edE} implies that $E$  is not  in $S_A\cup S_B$.

Since every end point of $T$  must appear in the relation  the only possibility is given by the picture
$$\xymatrix{     a \ar@{- }[rr]^{S_A} &&  b\ar@{-}[r]^{E}  &c  \ar@{- }[rr]^{S_B}  &&d  }$$\end{proof}

  \subsection{The contribution  of $\bar L_u$ equals to $\pm 2\delta_ue_{u\pm 1}$.  }\quad We say that $u$ is of type II  \smallskip 
  
  We want to prove
  \begin{theorem}\label{itre}
In case of a doubly odd circuit the tree $T$  is formed by 3 segments, $S_A,S_B,S_C$ each formed only by the edges in $A$ or $B$ or $C$.  Moreover the internal vertices of $S_B$ have all valency 2. 
\end{theorem}
 Thus the possible form of $T$ is that given by the next pictures on page \pageref{lefig}.\medskip

   We  thus have, from \eqref{relr3} or  \eqref{relb3}, a relation expressing $\pm 2\delta_ue_{u\pm 1}$ as linear combination of the edges $\ell_j\neq  \ell_{u-1},\ \ell_u$.  
\begin{equation}
\label{relr2}\pm 2\delta_ue_{u\pm 1}=
 \sum_{i\,|\,\vartheta_i=-1,\ i\neq u-1,u} \delta_i \mu_u(i)\ell_i  -\sum_{i\,|\,\vartheta_i= 1,\ i\neq u-1,u}   \delta_i \sigma_i  \mu_u(i)\ell_i+\gamma_{u-1}\bar a_{u-1} 
\end{equation}

 Now these edges are linearly independent so such an expression, if it exists, it is unique.  Let us assume for instance that the relation expresses $2e_{u-1}$, the other case is identical.\smallskip
 
We choose the root $r$ as in Figure \eqref{esub}. In order to understand which elements appear in $C_u$, 
 first remark that From Proposition \ref{vamu}  we have:
  \begin{lemma}\label{icontr}\strut
\begin{enumerate}\item If $\ell_u\not\prec\ell_j$ then $\mu_u(j)=0$  and $\ell_j$ does not appear in $C_u$.\item If $\ell_u\prec\ell_j$ and $\ell_j\not\prec  \ell_{u-1}$, we are in the case of figure  \eqref{esu1}    and they contribute by  $\pm \delta_j$.

\item  If  $\ell_u\prec\ell_j \prec\ell_{u-1}$ we have  $\mu_u(j)=\pm 1$  and then a contribution  $\pm \delta_{u-1}$   from  $\delta_{u-1}\bar a_{u-1} $ so a total contribution $\pm \delta_ j \pm \delta_{u-1}$.
\item Finally if  $ \ell_{u-1}\prec\ell_j$  they contribute by $ \pm 2\delta_j$ since $\mu_u(j)=\pm 2$ by   Corollary \ref{leduei1}.
\end{enumerate}
 
 \end{lemma}
 \begin{proof}
 The only edges $\ell_j$ that may contribute to the expression of $C_u$ are those for which  $\ell_u\prec\ell_j$ in fact otherwise  $e_u$ has coefficient 0 in $a_j$ since the path from the root to $a_j$  does not contain  $\ell_u, \ell_{u-1}$.
  \begin{equation}
\label{esu2}   \xymatrix{  &\ell_j &&\\ r  \ar@{-}[ru]   \ar@{-}[r]^{\ell_u} &s\ar@{-}[rr]^{\bar a_{u-1}}&&y \ar@{-}[r]^{\ell_{u-1}}  & x_{u-1}   } .
\end{equation} 
 The other cases      are similar.

 \end{proof}

\medskip
 
{\bf Case 1A} (single even circuit) no extra edge: 

\begin{proposition}\label{noo}
In this case such a relation cannot occur. 
\end{proposition}\begin{proof}
For instance if $2e_{u-1}$ is a linear combination $\sum_jc_j\ell_j$ of the edges $\ell_j\neq  \ell_{u-1},\ \ell_u$ since $e_{u-1}$ only appears in $\ell_{u-2}$ with sign $-1$ we must have that $c_{u-2}=-2$ and then $2e_{u-2}$ is a linear combination $\sum_jc_j\ell_j$ of the edges $\ell_j\neq  \ell_{u-2}, \ell_{u-1},\ \ell_u$, continuing by induction we reach a contradiction.\end{proof}

\subsubsection{{\bf Case 1B}  (single even  circuit)  with an extra edge:} \quad We may assume that the extra edge is $E=\vartheta e_1-e_h$, this edge divides the circuit into two parts $A,B$. The edges in $A:=\{\ell_1,\ldots,\ell_{h-1}\}$ and $E$ form an odd circuit as well as the edges in $B$ and $E$.   

Since $u$ has valency 2 we have $1,h\neq u$, we may  assume for instance that $h<u$ and $ u $ is an index in $B$ (we walk the circuit clockwise) the other case is identical. \begin{equation}\label{dude1}  \xymatrix{        \ar@{=}[r] &\cdots&h-1 \ar@{=}[r]    &h  \ar@{=}[r] &\cdots &u-1\ar@{-}[r] &u      && \\      \ar@{=}[u]_{\qquad\  A} \ar@{=}[r]   &\cdots 2\ar@{ -}[r] &1\ar@{ -}[r] \ar@{ =}[ru] ^E _{\qquad\qquad B}   &k&\cdots  \qquad \ar@{=}[r]  & \ar@{=}[r] &u+1 \ar@{=}[u]      &&  }\end{equation}
\begin{proposition}\label{ed1b}\strut If $u\in B$:\begin{enumerate}[1)]\item We have $E\prec \ell_{u-1}$.

\item The edges $\ell_a,\ a\in A$ satisfy  $\ell_u\prec \ell_a$ but not   $\ell_{u-1}\prec\ell_a$ or  $\ell_a\prec\ell_{u-1}$.

\item If an edge  $\ell_k,\ k\in B$  satisfies $\ell_u\prec\ell_k$  then  either $ \ell_k\prec\ell_{u-1}$ or $\ell_{u-1}\prec\ell_k     $. 

\item All the other edges are not comparable with $\ell_u$.\end{enumerate}

\end{proposition}\begin{proof} In this case all $\delta_j=\pm 1$.

We know that, by Proposition \ref{pari} v), we can write $2e_{h}=\gamma E+\sum_{i=1}^{h-1} \gamma_i\ell_i$ uniquely as the sum  of the edges of the odd circuit $A,E$ with signs $\gamma=\pm 1$.  

 If $\bar L_u=\pm 2e_{u-1}$ we write $2e_{u-1}= \pm \sum_{k=h}^{u-2}2\gamma_k\ell_k\pm 2e_h$ by Formula \eqref{priR}. We obtain a  relation   \begin{equation}\label{relrp}
\mathcal R^\dagger:\qquad 2e_{u-1}=\pm\gamma E\pm\sum_{i=1}^{h-1} \gamma_i\ell_i\pm \sum_{k=h}^{u-2}2\gamma_k\ell_k
\end{equation}    If $\bar L_u=\pm 2e_{u+1}= \pm \sum_{a=u+1}^{k}2\gamma_a\ell_a\pm 2e_1$ we  have a similar discussion for  $2e_1$ instead of $2e_h$.\smallskip

 The edges appearing in this  relation are all the edges of $A,E$ with coefficient $\pm 1$  and all the edges $\ell_k,\ h\leq k\leq u-2$ with coefficients $\pm 2$.  These edges are linearly independent  so this relation must be proportional (by $\pm 1$) to  \eqref{relr2}.  Notice that this is quite analogous to what we did for the relation of the odd circuits.

  \smallskip 
 
1)\quad    Since $E$,    is not an edge $\ell_i$,  it must   appear  in \eqref{relr2} as a term in  $\bar a_{i-1}$. This means that $E\prec \ell_{u-1}$.  \smallskip

2)\quad  We know that  all the edges in $A$ appear in \eqref{relrp} with coefficient $\pm 1$. If   $\ell_{u-1}\prec\ell_a$ then $\ell_a$  does not appear in $\bar a_{u-1}$  and  by  Corollary \ref{leduei1} it would have as coefficient $\pm 2$.

If  $\ell_a\prec\ell_{u-1}$ by Lemma \ref{icontr} iii) we would have a coefficient $0,\pm 2$ so 2) follows.  \smallskip

 3)\quad  The edges  $\ell_k,\ k\in B,\ h\leq k\leq u-2$ appear in $\mathcal R^\dagger$   with coefficient $\pm 2$.  
   
   In   \eqref{relr2}  if an  edge  $\ell_k,\ k\in B$   appears with coefficient $\pm 2$  then either $ \ell_k\prec\ell_{u-1}$ or $\ell_{u-1}\prec\ell_k     $  by  Lemma  \ref{icontr} ii). \smallskip

 4)\quad  All the others are not comparable with $\ell_u$.   \end{proof}
   \begin{proposition}\label{uvB}\strut  \begin{enumerate}[1)]\item If $u\in A,\ v\in B$  both of type II then $S_u\cap S_v=E$.
\item If $u,v\in B$  both of type II the union  of $S_u$ and $S_v$  is a segment.

\item The union of  $S_u,\ u\in B$ and $u$ of type II is a segment.
 \end{enumerate}

\end{proposition}
\begin{proof}
1)\quad In both cases the intersection  $S_u\cap S_v$  is a segment  $S$  (containing $E$), see \eqref{intt}.  In the first case by  Proposition \ref{ed1b} 2)  the   edges different from $E$ in $S_u$ are in $A$ while the other edges in $S_v$ are in $B$ so $S=E$.

2)\quad Take  $u,v\in B$  denote by  $\ell_h\prec \ell_k $  the end edges  of the segment  $S=S_u\cap S_v$      (possibly one of these edges is $E$). 

 If for $\ell_j\in  S_v$  we have  $\ell_k\prec \ell_j$  then  $\ell_u\prec\ell_k$, So by Proposition \ref{ed1b} 3)   either $\ell_j\preceq \ell_{u-1}$  or $\ell_{u-1}\prec \ell_j.$  
The first $\ell_j\preceq \ell_{u-1}$ contradicts the choice of  $\ell_k$  so we have the second and hence $\ell_{u-1}=\ell_k$. 

Recall that the two segments $S_u,S_v$ do not depend on the choice of the root, Definition \ref{Su}, so if we take as root the opposite end $x_{u-1}$ of $S_u$  we have a new order   $\prec'$ on the vertices of $T$.  In this new order if an edge  $\ell_j\subset S_v$  does not satisfy  $\ell_h\preceq\ell_j$  then   $\ell_h\prec'\ell_j$ and   then $\ell_h=\ell_u$.

So unless one is contained in the other the two segments intersect in a segment which is either initial in $S_u$ and final in $S_v$ or the converse. In all cases the union is a segment.

3)\quad This follows from Proposition \ref{unsegp}.\end{proof}
\subsubsection{Geometry of $T$ case 1B)}
   Denote by  $T_A$ and $T_B$  the  two minimal trees, inside $T$, generated by the edges $\ell_c$ with $c\in A,c\in B$ respectively. We have:
   \begin{corollary}\label{ePo} \strut  A)\quad  If the indices of $A$ (resp. of  $B$) are all of type I  then \begin{enumerate}[1)]\item  $T_A=\cup_{u\in A}S_u$ (resp.   $T_B=\cup_{v\in B}S_v$)  is a segment not containing $E$. Each internal vertex in $T_A$ is internal in at least one $S_u$ so it has valency 2.

   \item If the indices of $A$ and $B$ are all of type I  then  $T_A$ and $T_B$ form two disjoint segments separated by $E$. \end{enumerate}    
   \medskip

\quad B)\ If there is an index in $B$  (resp. in $A$) of type II,
\strut \begin{enumerate}[1)]\item the two minimal trees  $T_A$ and $T_B$  generated by $A,B$ respectively are segments and can intersect only in a vertex  or in the edge $E$. 
\item  If they intersect in a vertex then all $v\in A$ (resp. all   $v\in B$) have type I and the vertex is an end point of $E$.
 \end{enumerate}
\end{corollary}
\begin{proof} A)\quad 1)\ In this case we know, by \S \ref{coL0}, that all the segments $S_u$ for $u$ non critical are segments which do not contain $E$ and with the interior vertices of valency 2. 
The statement follows from 
Corollary \ref{seq}.
  \smallskip

2)\quad If these two segments have an edge in common  then, by the same Lemma \ref{SEP}, their union is a segment not containing $E$ and thus  this segment gives a minimal degenerate graph and the one we started from is not minimal.  The same happens if they meet in an end point of both. The only remaining case is that  $T_A$ and $T_B$ form two disjoint segments separated by $E$.
 \begin{equation}\label{casb}
\xymatrix{     a \ar@{- }[rr]^{T_A} &&  b\ar@{-}[r]^{E}  &c  \ar@{- }[rr]^{T_B}  && d }
\end{equation}
\medskip

B)\quad 1)\ Let us prove that $T_B$ and $T_A$ are segments $S_A,S_B$. We start  for $T_B$. By   Proposition \ref{uvB} 2) the union of  $S_u,\ u\in B$ and $u$ of type II is a segment $S$. If there are indices $u\in B$  of type I, we start with one preceding or following an index of type II  so $S_u\cap S\neq \emptyset$. Since the internal vertices of $S_u$ have valency 2 (Corollary \ref{ePo} 1))   it follows that $S\cup  S_u$ is a segment, it is all formed by edges in $B$ since otherwise it would  form a circuit with some edge of $A$ by 2) of Proposition \ref{ed1b}. Now we continue by induction.
\smallskip

As for $T_A$ if there is also a vertex of type II on $A$ then the previous discussion applies also to $A$  and we have $E$ internal to  $S_A,S_B$ so the picture is \begin{equation}\label{casb2}
\xymatrix{   &&&b&&&\\  & a\ar@{- }[rr]^{S_B}&&v \ar@{- }[u]^{S_A}    \ar@{-}[r]^{E}  &c  \ar@{- }[rr]^{S_B}  &&  d\\\  &  &&&e \ar@{- }[u]^{S_A} }
\end{equation}

Now assume that    all vertices of $A$ are of type I
so, by Part A) 1) ,  $T_A=S_A$  is a segment does not contain $E$  and $S_A\cap S_B$  can only intersect in an end vertex of $S_A$

  By Proposition \ref{ed1b} 2) $v$ is   an internal point of each $S_u$ with $u$ of type  II.
 Now suppose that this vertex $v\in S_u$ and it is not an end point of  $E$. 

Call $U$ the segment from $v$ to $E$, the picture is: 
$$\xymatrix{   &&&&&&\\  & \ar@{- }[r]^{S_B}&v \ar@{- }[u]^{S_A}  \ar@{- }[rr]^U \ar@{- }[r]_{\ell_j}&&  \ar@{-}[r]^{E}  &  \ar@{- }[rr]^{S_B}  &&  }$$

For all the edges    $\ell_j\in U$ the index $j$ must be of type I.   If $j$ is of type II then $v$  must be internal also to $S_j$ which contains $E$  and has one end edge $\ell_j$ to the left of $E$  so the second to the right of $E$. This contradicts the picture.

Moreover  $S_j\subset U$ since $E\notin S_j$  and $v$ has valency 3 so   cannot be internal to $S_j$. 

This means that $\ell_{j-1}\in U$  so it is of type I and 
continuing we have that all $\ell_f,\ f\leq j$ and $f\in B$ are of type I. 

 But $\ell_j$ is also an edge of $S_{j+1}$.   If $j+1$ is  of type I  then   $S_{j+1}\subset U$,  otherwise $v\in S_{j+1}$  is an internal vertex of valency 3 contradicting   1).  So $j+1$ is  of type II and then $E$ is in between  $\ell_j,\ \ell_{j+1}$.  We have again a contradiction  since $v\notin S_{j+1}$. 
 
 We reach the contradiction that all vertices in $B$ arein $U$  and are  of type I. So we have, if in $ A$ all indices are of type I:
 \begin{equation}\label{casb1}
\xymatrix{   &&&b&&&\\  & a\ar@{- }[rr]^{T_B}&&v \ar@{- }[u]^{T_A}    \ar@{-}[r]^{E}  &c  \ar@{- }[rr]^{T_B}  &&  d}
\end{equation}
  \end{proof}

\subsection{Contribution  of $\bar L _u$ equals to $\pm 2\delta_ue_{u\pm 1}$,  Case 2)}    Assume   $\bar L _u$ equals to $\pm 2\delta_ue_{u- 1}$. The other case is the same exchanging the order  in which we walk on the path.

   A doubly odd circuit is divided in 3 (or 2) parts: the two odd circuits $A,C$ and the segment $B$ (possibly empty)    joining them Figure \eqref{doppio}. We divide this into two subcases $u\in A\cup C$ and $u\in B$:
   
   \begin{proposition}\label{asuB}
Assume   $u\in B$.
\begin{enumerate} 
\item All internal vertices of the segment $S_u$ have valency 2. 
\item The edges in $A$ resp. in $C$ are on opposite sides of $S_u$.\end{enumerate}
\end{proposition}
\begin{proof} The picture is:\begin{equation}\label{doppiou1}
\xymatrix{    &u\ar@{-}[r]& u+1\ar@{.}[r]& k\ar@{-}[r]&t \cdots \ar@{=}[r]& &    && \\ B\!\!\!\!\!\!\!\!\!\!\!&  u-1\ar@{-}[u] &   &k+1 \ar@{-}[r]  \ar@{-}[u] _{\quad \quad \quad  C}&\cdots \ar@{-}[r] & \ar@{-}[u]     &&  \\   & 1 \ar@{.}[u]  \ar@{=}[r] &h\ar@{-}[r]    &\cdots   \ar@{-}[r]  & \ar@{=}[r] &       && \\   &2 \ar@{-}[u]_{\quad \quad \quad  A} \ar@{=}[r]   &3\ar@{ =}[r]    & \ar@{-}[r]  &\cdots \ar@{-}[r] &6\ar@{=}[u]      &&  }
\end{equation}

  If $u\in B$   we have  from Formula \eqref{boB} and  Formula \eqref{boBf}    \begin{equation}\label{boB2}
1)\quad  \boxed{ 2e_{u-1}=2\sum_{i=h+1}^{u- 1} \eta_i\ell_i\pm\sum_{i=1}^h\delta_i\ell_i,}
\end{equation}
 
with $\eta_i,\ \delta_i=\pm 1$. Since $u\in B$ we have $\delta_u=2\eta_u=\pm  2=\pm \delta_{u-1}$, Formula \eqref{boB}.

Due to the computations in \S \ref{18} we have that $\bar L_u=\pm 2\delta_ue_{u-1}= \pm 4e_{u-1}$  in cases 2, 4, 5, 12, 13, 15, 18 and $\bar L_u=\pm 4e_{u+1}$  in cases  8,\ 9.  

Therefore   $2e_{u-1}= 2\sum_{i=h+1}^{u- 1} \eta_i\ell_i\pm\sum_{i\in A} \delta_i\ell_i,
$ by formulas \eqref{boB2}, multiplied by $\pm 2$,  must coincide with one of those  for $\bar L_u$    given by \eqref{relr2}.   

In these Formulas  the edges   $\ell_{u} \prec \ell_i\prec \ell_{u-1} $,     by Formula \eqref{ukveg}, appear in $\delta_{u-1}\bar a_{u-1}$  with coefficients $\pm\delta_{u-1}=\pm 2$, they also appear under the $\sum$ sign in   \eqref{relr2} with coefficient  $\pm\delta_i$.  If these indices do not appear in  \eqref{boB2} they  must cancel with    edges with $\mu_u(i)\neq 0$.   In   \eqref{boB2} the indices $j\in C$  do not appear so we claim that $\ell_u\not\prec \ell_j$. 

 In fact if  $\ell_{u} \prec \ell_i\prec \ell_{u-1} $ then $\mu_u(i)=\pm 1$ so in order to cancel the contribution from  $\delta_{u-1}\bar a_{u-1}$  we should have $\delta_i=\pm 2$  which is not the case. If $\ell_{u-1} \prec \ell_i$  then $\mu_i=\pm 2$  and then this is not cancelled. So only $\ell_{u} \not\prec \ell_i$ is possible.\smallskip

 If $i\in A$  then in  \eqref{boB2}  $\ell_i$ appears with coefficient $\pm 1$, so in  \eqref{relr2} it must appear with coefficient $\pm 2$. Use  Proposition \ref{vamu}.
 
 If $\ell_{u} \prec \ell_i  $ then $\mu_u(i)=\pm 1$  and $\delta_i=\pm 1$.If $\ell_i$ is not comparable with  $\ell_{u-1}$ this is the only contribution to the Formula   \eqref{relr2}.   If $\ell_{u} \prec \ell_i\prec \ell_{u-1} $  in   $\delta_{u-1}\bar a_{u-1}$ the edge $\ell_i$ appears  coefficient $\pm 2$,    so a total of an odd coefficient again a contradiction.  The only possibility left is  $\ell_{u-1} \prec \ell_i$.  So  ii) is proved.\smallskip

 We claim that there is no edge $\ell_a,\ a\in B$  with $\ell_u\prec\ell_a$ and $\ell_a$ is not comparable with $\ell_{u-1}$.  Indeed this edge would have  $\mu_u(a)=\pm 1$ and would not appear in $\bar a_{u-1}$,\eqref{esu1} (recall  $\bar a_{u-1}$ is a sum   formed by the edges $\ell_i,\ \ell_u\prec\ell_i\prec\ell_{u-1}$).
 
This is incompatible with the fact that the coefficient of $\ell_a,\ a\in B,\ a\neq u,u-1$ in Formula \eqref{boB2} must be $\pm 2\eta_a=\pm 2$ so that in Formulas  \eqref{relr3}  or \eqref{relb3}  must be $\pm 4$.  But in   \eqref{relr3}  or \eqref{relb3}   the coefficient of $\ell_a,\ a\in B$ is $\pm 2$.\smallskip

Thus  we deduce that all internal vertices of the segment $S_u$ have indices in $B$ and have valency 2 (but in general not all indices in $B$ appear in $S_u$). \end{proof}Assume  $u\in A$   (the case $u\in C$ is similar). 
The picture is:
\begin{equation}\label{doppiouu}
\xymatrix{      \ar@{-}[r]&z\cdots \ar@{=}[r]& &    && \\ 
 k  \ar@{-}[r]  \ar@{-}[u] _{\quad \quad \quad  C}&\cdots \ar@{-}[r] & \ar@{-}[u]     &&  \\ 
    1 \ar@{.}[u]^B  \ar@{=}[r] &h\ar@{-}[r]    &\cdots  u \ar@{-}[r]  &u-1 \ar@{=}[r] &       && \\    2 \ar@{-}[u]_{\quad \quad \quad  A} \ar@{=}[r]   &3\ar@{ =}[r]    & \ar@{-}[r]  &\cdots \ar@{-}[r] &6\ar@{=}[u]      &&  }
\end{equation}
 \begin{proposition}\label{case2L}\strut

\begin{enumerate}[1)]\item If $j\in C$ then   $\ell_{u-1}\not\prec\ell_j $. 
\item  Inside the segment $S_u$  there are only   edges of  $A $.
\item All $\ell_j,\  j\in B\cup C$ are in branches which originate from internal vertices of $S_u$.
\item If  $j\in A  $ and $ j\leq u-2$ we   have  either   $\ell_{u-1}\prec\ell_j$ or $\ell_j\prec\ell_{u-1}$.  For the remaining $j\geq u-1\in A$ we have $\ell_u\not\prec\ell_j$. 
 
 \end{enumerate}

\end{proposition}\begin{proof}

  We have    a linear combination  of the edges  in $B,C$ with coefficients $\delta_i$  which is equal to $2e_1$.  $\delta_i=\pm 1$ if $i\in C$ and $\pm 2$ if $i\in B$ (cf.   Formulas \eqref{boB},   \eqref{boBf}).
 
  Then     $2\sum_{i=1}^{u-2}\delta_i\ell_i=2e_1-2\delta_{u-2} e_{u-1}$  Formula \eqref{priR},   
     $$\mathcal R^\dagger:\sum_{j\in B\cup C}\delta_j\ell_j-2\sum_{i=1}^{u-2}\delta_i\ell_i= 2\delta_{u-2}e_{u-1} =\pm 2 e_{u-1}.$$ The  expression  of $  2\delta_{u-2}e_{u- 1} $ as linear combination of the linearly independent edges $\ell_j\neq  \ell_{u-1},\ \ell_u$ is   unique.
The expression $\mathcal R^\dagger$   must be proportional, by $\pm 1$, to   \eqref{relr3}  or \eqref{relb3} by Proposition \ref{leduei}.\medskip            . 

1)\quad Comparing these relations we first observe that, if $j\in  C$     the edge $\ell_j$ must have coefficient   $\pm\delta_j=\pm1$.  By corollary \ref{leduei1}  if $\ell_{u-1}\prec \ell_j$ we have that $\mu_u(j)=\pm 2$  hence   we deduce that   $\ell_{u-1}\not\prec\ell_j $. \medskip            .

2)\quad If $\ell_u\prec\ell_j\prec \ell_{u-1}$ the coefficient of $\ell_j$  in the two possible relations \eqref{relr3}  or \eqref{relb3} comes from two terms,   a term $\pm \delta_j$ coming from the first two summands (since in this case $\mu_u(j)=\pm 1$), and a term $\pm \delta_{u-1}$ from $\bar a_{u-1}$, hence no index in $B$ or $C$ can appear in  $\bar a_{u-1}$ by parity.   Inside the segment $S_u$  there are only   edges of  $A $. 

3)\quad  Since the edges  in $B$ or $C$ appear in the relation $\mathcal R^\dagger$ with coefficient $\pm 1$   we deduce that $\mu_u(j)=\pm 1$ so all $\ell_j, j\in B\cup C$ are in branches which originate from internal vertices of $S_u$.

4)\quad In  $\mathcal R^\dagger$ the indices in $A$ which appear are   $i\in A, i\leq u-2$ and the corresponding edges have coefficient $\pm 2$  therefore this last statement follows from   Lemma \ref{icontr} since in this case all $\delta_i=\pm 1$.

 A similar consideration holds if $u\in C$.\end{proof} So the last case is for a doubly odd circuit with at least a vertex of type II.
\begin{corollary}\label{lasca}\strut\begin{enumerate}[1)]\item The edges in $B$ always form a segment $S_B$, its internal vertices have valency 2.

\item  If there is  an index of type II in $B$ all edges in $A$ and all edges in $C$ are separated and lie in the two trees $T_A,T_C$ originating from the two end points of  $S_B$. 
\item     $T_A=S_A,\ T_C=S_C$ are both segments with no edge in common.
\item If there is no index of type II in $B$  but  an index of type II in $A$ (or $C$) all edges in $A$ and all edges in $C$ are separated and lie in   two segments which can be disjoint or meet in one vertex.\end{enumerate}
 \end{corollary} 
\begin{proof}  1)\quad The proof is similar to that of Corollary \ref{ePo} where we showed  that, if $j\in B$  is of type I 
inside the segment $S_u$ there are only edges $\ell_j$ with $j\in B$ and its internal vertices have valency 2, we have proved this now also for type II.   The claim follows from  Lemma  \ref{SEP} or arguing as in Corollary  \ref{seq} of Proposition \ref{intve}.

  2)\quad  This follows from Proposition \ref{asuB}  ii)  since the internal vertices of $S_B$ have valency 2 and the edges in $A$ and $C$ are separated by $S_u$.  

3)\quad If all the vertices of $A$ are of type I then $T_A$ is a segment by   Corollary  \ref{seq} of Proposition \ref{intve},  same for $T_C$.

 So assume $A$ has a vertex $u$ of type II.  By Proposition \ref{case2L} 2) 
 inside the segment $S_u$  there are only   edges of  $A $ and by the same proposition item 4)  inside $T_A$  the internal vertices of $S_u$ have valency 2, so the argument is the same as that of  Corollary \ref{seq}.
 
 If $B$  has an index of type II  by case 2)  $T_A$ and $T_B$ are disjoint. If   $B$  has no index of type II since we are assuming the existence of indices of type II  we need to  have such an index in $A$ or in $C$ or in both.
 
 Assume there is such an index $u$ of type II  in $A$.  By Proposition \ref{case2L}
  all $\ell_j,\  j\in B\cup C$ are in branches which originate from internal vertices of $S_u$.  So the segments $S_A$ and $S_C$  meet in a vertex which is internal to $S_A$  and can be also internal  to $S_C$ while  $S_B$  meets $S_A$  in a a vertex which is internal to $S_A$  but it is also an end vertex for $S_B$.  Finally if there is  an index    of type II also  in $C$ then $S_B$  meets $S_A$ and $S_C$ in their intersection.
. 
  \end{proof} In the end we can have the following possible pictures:
 \subsubsection{Indices of type II}  If there is at least one index of type II  the case analysis that we have performed shows that
  between two edges in $A$ there are only edges in $A$ and the edges in $A$ form a segment, the same happens for $B,C$. Denoting $S_A,S_B,S_C$ these segments  their union is a tree, the internal vertices of $S_B$ have valency 2,  so their relative position a priori can be  only one of the following, up to exchanging $A$ with $C$. \smallskip
  
   If we are in case 2) $S_A$ and $S_C$ are opposite to $S_B$ so we are in case a) or the special a'), a")
  $$ a)  \xymatrix{     \ar@{-}[rrr]^{S_A}  && \ar@{-}[dd]^{S_B}  &\\  &&&&   \\  \ar@{-}[rrrr]^{S_C\qquad} &&&&\\  &&&&   }$$
  $$  a') \xymatrix{  &  \ar@{-}[ddd]^{S_A}  &&   &\\  &&&&   \\ & \ar@{-}[rr]^{S_B} &&\ar@{-}[rr]^{S_C}&&\\  &&&&   } \qquad a") \xymatrix{    \\\ar@{-}[rr]^{S_C}  & &  \ar@{-}[rr]^{S_B} &&  \ar@{-}[rr]^{S_A}  &&   }$$ \label{lefig}
  $S_A,\ S_C$ are on the same side of $S_B$  we are in case b) or the special b'), b"),  b''') , b'''') 
$$b)\quad   \xymatrix{  &  \ar@{-}[ddd]^{S_A}  && \ar@{-}[dd]^{S_B}  &\\  &&&&   \\  \ar@{-}[rrrr]^{S_C} &&&&\\  &&&&   } \qquad     \qquad b') \xymatrix{  &  \ar@{-}[ddd]^{S_A}  &&   &\\  &&&&   \\ & \ar@{-}[rr]^{S_C} &&\ar@{-}[rr]^{S_B}& &\\  &&&&   } $$ 
$$ b''') \xymatrix{  &  \ar@{-}[dd]^{S_A}  && \ar@{-}[dd]^{S_B}  &\\  &&&&   \\  \ar@{-}[rrrr]^{S_C} &&&&\\  &&&&   } \qquad  b'''') \xymatrix{  &  \ar@{-}[ddd]^{S_C}  && \ar@{-}[ddll]^{S_B}  &\\  &&&&   \\  \ar@{-}[rrrr]^{S_A} &&&&\\  &&&&   }  $$
Of course $ b'''') $ can also be more special if $S_A, S_C$  have only vertices of type I, and we may go back to the cases in Formula \eqref{unseg1}.

  We  may also have that $B$ is empty  so $S_B$ does not appear.

 \section{Final step}
\subsubsection{ All indices  are of type I,\ $L=0$}  We have already seen (Case 1) that the case of   the single circuit and all indices  are of type I is not possible. Let us thus treat the special case when we are in the doubly odd circuit  and still  all indices of $A\cup C$ are of type I   or  when just the indices of $A$ are of type I but we know that they form a segment.

If neither $S_A,S_B,S_C$ contains a critical vertex in the interior we   have seen that the graph spanned by  $A\cup C$  is a segment as well as $S_B$ and we have.
\begin{equation}
\label{unseg}a') \xymatrix{    r\ar@{-}[rr]^{S_{A\cup C}  }& &v  \ar@{-}[rr]^{S_B} && w    }. 
\end{equation}

 In this segment we take as root one on its end points say $r$, the segment is a sequence of edges $m_i$  and vertices  $c_i$  as
 $$ \xymatrix{    0\ar@{-}[r]^{m_1 }&c_1\ar@{-}[r]^{m_2 } &c_2 \ar@{.}[rr]  && c_{k-1}\ar@{-}[r]^{m_k }&c_k    }. $$According to Definition \ref{biai}  denote by $\bar\sigma_i,\bar\lambda_i$ the corresponding values  of color and orientation (with respect to this root) of $\ell_i$. 
 
 Of course the $m_i$ are a permutation of the $\ell_j$.  Recall that the notation $\sigma_i,\lambda_i$ is relative to the segment   $S_u$ as in the previous discussion (see formula \eqref{esub}).   
 
 Take a segment $S_u\subset T$ of some length $z$, it has some initial vertex  $c_p$  and $m_{p+1}=\ell_u, \ \ell_{u-1}$, in this second case it is oriented opposite to its orientation in picture \eqref{Su}.  Its other end point is in the first case  $c_{p+z}$  in the second $c_{p-z}$.
  $$ \xymatrix{    c_p\ar@{-}[rrr]^{x_{u-1} }&  &   & c_{p+z}     }\qquad \xymatrix{    c_p\ar@{-}[rrr]^{x_{u-1} }&  &   & c_{p-z}     }. $$
\begin{lemma}\label{basi}
\begin{equation}\label{ivc}
\ell_u\prec\ell_{u-1}\implies \bar\sigma_{u-1}=\sigma_{u-1} \bar\sigma_u\theta_u,\quad \ell_{u-1}\prec\ell_u\implies \bar\sigma_{u-1}=\sigma_{u-1} \bar\sigma_u\theta_u\theta_{u-1}.
\end{equation}
\end{lemma}\begin{proof}
In the first case $\ell_u\prec\ell_{u-1}$ we have  $c_{p+z}= x_{u-1} +\sigma_{  u-1 }c_p$ is the right end point of $\ell_{u-1}$ .
By Definition the color $\bar\sigma_{u-1}$ is the color of its end point, in the first case $c_{p+z}$ which has color  $\bar\sigma_{u-1}=\sigma_{u-1}\phi$ with $\phi$  the color of  $c_p$. Now the end point of $\ell_u$  is $c_{p+1}=\bar\lambda_u\ell_u+\theta_uc_p$ with color $\theta_u\phi=\bar\sigma_u$.  
Substituting we have Formula \eqref{ivc}.\smallskip

In the second case we have  $c_{p  }= x_{u-1} +\sigma_{  u-1 }c_{p-z }$ and the end point of  $\ell_{u-1}$  is $c_{p-z+1}.$
We have  $c_{p-z+1}= \bar\lambda_{u-1}\ell_{u-1}+\theta_{u-1}c_{p-z}. $ Let $\psi$ be the color of  $c_{p-z} $  we have $\psi=\sigma_{u-1}\phi$.

The color $\bar\sigma_{u-1}$ is the color of $c_{p-z+1}$, which is  $\bar\sigma_{u-1}=\theta_{u-1}\psi=\theta_{u-1}\sigma_{u-1}\phi =\theta_{u-1}\sigma_{u-1}\theta_u \bar\sigma_u$.\end{proof}\medskip

 In the next Lemma we analyze  the 9 cases in which $\bar L_u=0$, see \S \ref{18}.

 \begin{lemma}\label{icons}
 We claim that every edge $\ell_j,\ j\in A$  (resp. $j\in B$ or  $j\in C$)   has the property that $\delta_j= \delta  \bar\sigma_j$ if red and $\delta_j=  \delta \bar\lambda_j\bar\sigma_j$ if black, setting $\delta=\delta_1\bar\sigma_1$ (resp.  $\delta=\delta_h\bar\sigma_h$ where $h$ is the minimal element in $B$ or in $C$).
\end{lemma}
\begin{proof} By induction $\delta_{u-1}= \delta  \bar\sigma_{u-1}$ if  $\ell_{u-1}$ is red and $ \delta_{u-1}=  \delta \bar\lambda_{u-1}\bar\sigma_{u-1}$ if black.

Look at  $S_u$ and use the notations $\sigma_i,\lambda_i$  for the root chosen in \eqref{Su}, which of course depends on $u$.
Recall that the elements  $\delta_i=\pm 1$  are defined by Formula  \eqref{priR}.

{\bf Case 1)}\quad  If  $\ell_{u-1}, \ell_u$ are both red $\sigma_{u-1}=1$. 

By Lemma \ref{icamm} an definition  \eqref{priR} $\delta_u =-\delta_{u-1}$. 
From Formula \eqref{ivc}  $$\delta_u\stackrel{\eqref{priR}}=-\delta_{u-1}=  -\delta  \bar\sigma_{u-1}= \delta  \bar\sigma_{u }\sigma_{u-1}   =   \delta  \bar\sigma_{u }.$$

{\bf Case 3), 6)}  $\ell_{u-1}$ is red and $ \ell_u$ is black. We have $\sigma_{u-1}=\lambda_u,\ \delta_u=\delta_{u-1}=  \delta  \bar\sigma_{u-1}$.  
\smallskip

If $\ell_{u-1}\prec\ell_u$  we have $\sigma_{u-1}=-\bar \sigma_{u-1}\bar \sigma_{u}$ and $\bar \lambda_u=-\lambda_u$, thus $\bar \lambda_u=\bar \sigma_{u-1}\bar \sigma_{u}$.

If $\ell_u\prec\ell_{u-1}$  we have  $\sigma_{u-1}= \bar \sigma_{u-1}\bar \sigma_{u}$ and $\bar \lambda_u=\lambda_u$ thus $\bar \lambda_u=\bar \sigma_{u-1}\bar \sigma_{u}$.

In both cases thus $ \bar\sigma_{u-1}=\bar \lambda_u \bar \sigma_{u}$ and so 
 $ \delta_u=   \delta  \bar\sigma_{u-1} = 
 \delta  \bar\sigma_{u }\bar\lambda_u$.\medskip

{\bf Case 7), 10)}    $\ell_{u-1}$ is black and $ \ell_u$ is  red so $\delta_u =-\delta_{u-1}.$  We have  $\sigma_{u-1}=\lambda_{u-1}$.   \smallskip

If $\ell_{u-1}\prec\ell_u$ we have $\lambda_{u-1}\bar\lambda_{u-1}=-1.$ From formula \eqref{ivc}  $ \bar\sigma_{u-1}=\sigma_{u-1} \bar\sigma_u\theta_u\theta_{u-1} $ implies $ \bar\sigma_{u-1}= - \bar\sigma_u\sigma_{u-1}= - \bar\sigma_u\lambda_{u-1}= \bar\sigma_u\bar \lambda_{u-1}$
$$\delta_u= -\delta_{u-1}=   -\delta \bar\lambda_{u-1} \bar\sigma_{u-1}  =\delta  \bar\sigma_{u }.$$
If $\ell_u\prec\ell_{u-1}$ we have $\lambda_{u-1}\bar\lambda_{u-1}= 1,\quad \bar\sigma_{u-1}\stackrel{\eqref{ivc}}= - \bar\sigma_u\sigma_{u-1}$$$\delta_u= -\delta_{u-1}=   -\delta  \bar\sigma_{u-1}\bar\lambda_{u-1} =   \delta  \bar\sigma_{u }\sigma_{u-1} \bar\lambda_{u-1}  =   \delta  \bar\sigma_{u }\lambda_{u-1}\bar\lambda_{u-1}=\delta  \bar\sigma_{u }.$$
{\bf Case 11), 14), 16), 17) )} 
 $\ell_{u-1}, \ell_u$ are both black. 
 
 We have  $\sigma_{u-1}=\lambda_u\lambda_{u-1}$ by Proposition \ref{leduei}.
 
 If $\ell_{u-1}\prec\ell_u$ (in the order of the total segment) we have $\lambda_{u }\bar\lambda_{u }=\lambda_{u-1}\bar\lambda_{u-1}=-1, $ $\bar\sigma_{u-1}=  \bar\sigma_u\sigma_{u-1}$
 $$\delta_u=  \delta_{u-1}=   \delta  \bar\sigma_{u-1}\bar\lambda_{u-1} =  \delta  \bar\sigma_{u }\sigma_{u-1} \bar\lambda_{u-1}  =    \delta  \bar\sigma_{u }\lambda_{u-1}\lambda_{u }\bar\lambda_{u-1}=\delta \bar\lambda_{u } \bar\sigma_{u }.$$
 
 f $\ell_u\prec\ell_{u-1}$ (in the order of the total segment) we have $\lambda_{u }=\bar\lambda_{u },\ \lambda_{u-1}=\bar\lambda_{u-1}, $ $\bar\sigma_{u-1}=  \bar\sigma_u\sigma_{u-1}$
 $$\delta_u= -\delta_{u-1}=   -\delta  \bar\sigma_{u-1}\bar\lambda_{u-1} =   \delta  \bar\sigma_{u }\sigma_{u-1} \bar\lambda_{u-1}  =   \delta  \bar\sigma_{u }\lambda_{u-1}\lambda_{u }\bar\lambda_{u-1} =\delta \bar\lambda_{u } \bar\sigma_{u }.$$
Clearly  $\lambda_{u-1}\lambda_{u }\bar\lambda_{u-1}=\bar\lambda_{u }$.
 \end{proof}
 
We keep  the left vertex $r$  of $S_{A\cup C}$ as in \eqref{unseg}  as root, that is we consider it as the 0 vertex and want to compute first the value of the other end vertex $v$  of $S_{A\cup C}$ and then the end vertex $w$ of the total segment appearing in \eqref{unseg}. \smallskip

Recall that  we have an even number of red edges in $A\cup  C$ so that the end vertex $v$ is  black, let us denote  by $\ell_j$ the edge ending in $v$ so $\bar \sigma_j=1$.  

 By Proposition \ref{coo} the   group element  $g\in G_2$ so that $g\cdot 0=v$  is the composition of  the edges $ \ell_i$.  
We can compute it by using the 3 options of formula \eqref{La} for which $\bar \sigma_j=1$. 
\begin{proposition}\label{cov}
 \begin{equation}\label{IV}
v= \sum_{\ell\preceq \ell_j}\bar \sigma_\ell\bar \lambda_\ell\ell = \sum_{i\in A\cup C}\bar \sigma_i\bar \lambda_i\ell_i = \sum_{i\in A }\bar \sigma_i\bar \lambda_i\ell_i +  \sum_{i\in   C}\bar \sigma_i\bar \lambda_i\ell_i .
\end{equation}

\end{proposition}  \begin{proof}
We start from the 3 cases of Formula \eqref{La}  where  $\bar \sigma_j=1$.\begin{equation} \label{lala}
a_j=\begin{cases}\begin{matrix}
 - \sum_{\ell\prec  \ell_j}\bar \sigma_\ell\bar \lambda_\ell\ell  ,&\bar \sigma_j= 1,\ \ &\ell_j\ \text{red}\quad \\
 \sum_{\ell\preceq \ell_j}\bar \sigma_\ell\bar \lambda_\ell\ell , &\bar \lambda_j= 1,\ \ &\ell_j\ \text{black}\\
 \sum_{\ell\prec  \ell_j}\bar \sigma_\ell\bar \lambda_\ell\ell , &\bar \lambda_j= -1, &\ell_j\ \text{black} 
\end{matrix}
\end{cases}
\end{equation}
 
 If $\ell_j$ is red  or if it is black and $\bar\lambda_j=-1$  we have, by the Definition \ref{biai} of $a_j,b_j$,   that the last vertex   $v=b_j$ and not $a_j$, in the remaining case  $v=a_j$   we have  Formula \eqref{IV}.\smallskip
 
 Otherwise   $$ v= \bar \lambda_j\ell_j+\theta_ja_j =\begin{cases}\ell_j+ \sum_{\ell\prec  \ell_j}\bar \sigma_\ell\bar \lambda_\ell\ell  ,\ \bar \sigma_j= 1,\ \  \ell_j\ \text{red}\quad\\
-\ell_j+  \sum_{\ell\prec  \ell_j}\bar \sigma_\ell\bar \lambda_\ell\ell ,\  \bar \lambda_j= -1,  \ell_j\ \text{black} 
\end{cases} $$ 
In both cases we have Formula \eqref{IV}  for $v$.
 \end{proof}   
  By Lemma \ref{icons} we have  $\bar \lambda_j \bar\sigma_j=\delta\delta_j$ hence   $\sum_{j\in A}\bar \lambda_j \bar\sigma_j\ell_j=\delta\sum_{j\in A}  \delta_j\ell_j=\pm 2e_1$ and similarly $\pm\sum_{j\in C}\bar \lambda_j \bar\sigma_j\ell_j=\pm 2e_k$ (cf. \eqref{doppiou}).  \smallskip
  
  We thus  have that $v=\pm 2(e_1-e_k)$  or  $v=\pm 2(e_1+e_k)$ but this last  is impossible for a   vertex which has mass 0.  If $B=\emptyset$ then $k=1$ and $v=0$ so $T$ is not a tree. The same argument applies if  also $C=\emptyset $ so we are in the case of an even circuit.\bigskip

For  the segment $S_B$ with root $v$ and end $w$  the vertex $w$ can have any color,  we denote  by $\ell_j$ the edge ending in $w$.  Now keep in  mind  that we have defined $\delta_i=2\eta_i$ so $\delta=\pm 2$  and we have to divide by 2 to get the correct Formula.

 If the color of $w$ is black the previous argument applies and then gives as value of $S_B$  
 \begin{equation}\label{IVb} 
w=   \sum_{i\in B}\bar \sigma_i\bar \lambda_i\ell_i = \pm(e_1-e_k)
\end{equation}

   If the color of $w$ is -1, we claim that $w= -e_1-e_k$. For this  we need to analyze more cases.   If $\ell_j$ is red we apply the first of Formulas \eqref{La} and 
    \begin{equation}\label{IVc} 
w= -\sum_{\ell\preceq \ell_j}\sigma_\ell\lambda_\ell\ell=-  \sum_{i\in B}\bar \sigma_i\bar \lambda_i\ell_i = -e_1-e_k \end{equation}   
If $\ell_j$ is black we argue as in the previous Proposition and always have  $w=   \sum_{i\in B}\bar \sigma_i\bar \lambda_i\ell_i = -e_1-e_k$.
\begin{corollary}\label{PrC0}
The case of an even circuit or  \eqref{unseg} a')  does not occur or it produces a not--allowable graph   \ref{ilpunto1}.
\end{corollary} 
\subsubsection{Conclusion}\quad In the first case  we take as root the point $v$. Now the left and right hand vertices are $r=\pm 2(e_1-e_k),\ w=\pm  (e_1-e_k)$. The relation is, $  r=\pm 2w$  so if the graph is degenerate one should have $4 C(w)=C(r)=\pm 2 C(w)$ implies $C(e_1-e_k)=0$ implies $k=1$  and $v=w=0$  so $T$ is not a tree.

\medskip  

In the second case (root $v$)  $w= -e_1-e_k,\ r=\pm   2(e_1-e_k)$. Change the root to $r$  now $w=  -e_1-e_k \pm   2(e_1-e_k)$ equals  $-3e_1+e_k$ or  $-3e_k+e_1$  which   which also gives a non allowable graph from Definition \ref{ilpunto1} and Proposition \ref{ilpunto0}.

 If the edges in $A$ (an odd circuit) form a segment and are of type I  the same argument shows that fixing the root at one end the other end vertex is $-2e_i$ for some $i$. We deduce
\begin{corollary}\label{PrC}
The case of all indices of type I  in $A$ or in $C$ does not occur or it produces a not--allowable graph   \ref{ilpunto1}.
\end{corollary}

2) If  $A$   contains no index of type II)  we apply to it Lemma \ref{icons}  and deduce that the segment equals $\delta \sum_{i\in A}\delta_i\ell_i=-2\delta e_1$. Since the mass of a segment can only be $0,-2$ we deduce that if one extreme is set to be 0 the other is $-2e_1$.

3) is similar to 2).

Notice that at this point we have proved Theorem \ref{MM} for the doubly odd circuit  in all cases except a), b), and  b''''). 

4) Let us treat the case in which $u\in A$ gives a contribution to $\bar L_u$ equal $\pm 2e_{u-1}$ (the other is similar), from our analysis in our setting  all edges $\ell_j,\ j\leq u-2$ must be comparable with $\ell_u $.
\smallskip

In all cases we have that   $S_A$  and $S_C$ have a unique critical vertex which divides the segment.

 So $S_A$ is divided into two segments, one  $X$ ending with a red vertex $x$ the other $Y$ with a black vertex $y$ since in   $S_A$ there is an odd number of red edges  which are distributed into the two segments.  \smallskip

We choose as root the critical vertex. With this choice we denote by $\bar\sigma,\bar \lambda$ the corresponding values  on the edges  (in order to distinguish from the ones $ \sigma,  \lambda$ we have used where the root is at the beginning of $S_u$). 
 \smallskip

 \begin{lemma}\label{chiave} i)\quad   The edges in $Y, X$  have the property that, $ \delta_j\bar\sigma_j\bar\lambda_j=\delta$ is constant.
 \smallskip

Then using Formula \eqref{ukvef} of Theorem \ref{ilve}

  ii)\quad  $$y=\sum_{j\in Y} \bar\sigma_j \bar\lambda_j\ell_j= \delta\sum_{j\in Y}\delta_j\ell_j;\quad x=-\sum_{j\in X} \bar\sigma_j \bar\lambda_j\ell_j=- \delta\sum_{j\in X}\delta_j\ell_j$$
$$\delta=-1,\quad x-y=-2e_1 $$
 
\end{lemma}
\begin{proof}  i)    We want to prove that on   $X$ and $Y$ the value  $ \delta_j\bar\sigma_j\bar\lambda_j$ is constant. For this by induction it is enough to see that  the value does not change  for $\ell_u,\ell_{u-1}$. 

When they are not separated  by the critical vertex $v$ (of valency 4) we can use Lemma  \ref{icons}. 

When separated
  we first compare the values that we call $\bar\sigma_j$ when we place the root at the critical vertex with the values $\sigma_j$  when we place the root at the beginning of $\ell_u$. \begin{equation}
\label{esubx} \quad  \xymatrix{   r    \ar@{-}[r]^{\ell_u} &s\ar@{--}[r] &v \ar@{--}[r] &  y \ar@{-}[r]^{\ell_{u-1}}  & x_{u-1}   } .
\end{equation} 
We claim that $ \bar\sigma_u\bar\sigma_{u-1}=\sigma_{u-1} $.  

Let $g_1,g_2\in G_2$ be such that $r=g_1v,\ x_{u-1}=g_2v$ so  $x_{u-1}=g_2\circ g_1^{-1}r.$   $ \bar\sigma_u,\ \bar\sigma_{u-1} $  are respectively the color  of $g_1,g_2$ and so $\sigma_{u-1} $ the color of $g_2\circ g_1^{-1}$  is their product.\medskip

In order to prove  that $\delta_j \bar \sigma_j\bar\lambda_j$ is constant  we need to show that when $\ell_u,\ell_{u-1}$ are separated  the product of the two terms is 1.  That is we need
$$ 1=   \delta_{u-1}\bar\sigma_{u-1}\bar\lambda_{u-1}    \delta_u \bar\sigma_u \bar\lambda_u=  \delta_{u-1} \sigma_{u-1}\bar\lambda_{u-1}    \delta_u   \bar\lambda_u . $$
  We have $\bar\lambda_{u-1}= \lambda_{u-1}$  while $\bar\lambda_u=- \vartheta_u\lambda_u$. In other words  we need
  $$  -\delta_{u-1}\vartheta_u \sigma_{u-1} \lambda_{u-1}    \delta_u    \lambda_u=1. $$
Since by definition $\delta_{u-1}\vartheta_u =    \delta_u $
we have to verify that
   $$  -\delta_{u-1}\vartheta_u \sigma_{u-1} \lambda_{u-1}    \delta_u    \lambda_u=-  \sigma_{u-1} \lambda_{u-1}      \lambda_u=1.  $$  This is in our case the content of the second part of Corollary \ref{leduei}.
\smallskip

    ii) By Formula \eqref{ukvef} and part i)
   $$y=\sum_{j\in Y} \bar\sigma_j \bar\lambda_j\ell_j= \delta\sum_{j\in Y}\delta_j\ell_j;\quad x=-\sum_{j\in X} \bar\sigma_j \bar\lambda_j\ell_j=- \delta\sum_{j\in X}\delta_j\ell_j$$ hence
   $x-y=- \delta\sum_{j\in A}\delta_j\ell_j=\delta2e_1$.  But $\eta(x)=-2,\eta(y)=0$  implies $\delta=-1$.\end{proof}

  \begin{proposition}
i)\quad If the graph is resonant $x+y=-2e_i$ for some $i\neq 1$.  ii)\quad The graph is not allowable.
\end{proposition}
\begin{proof} ii)\quad If we take as root the vertex $x$  the other vertex of $S_A$ is $x+y$.   So if $x+y=-2e_j$ the graph is not allowable  by Definition \ref{ilpunto1}. \smallskip

i)\quad We choose as root the critical vertex of $S_A$.
We have  $x-y=-2e_1 =\sum_{j\notin A}\delta_j \ell_j$ is the minimal relation. Therefore the resonance relation has the form:
$$C(x)-C(y)=\sum \alpha_iC(w_i) $$  where the vertices $w_i$ are linear combination of the edges not in $A$. Therefore  these vertices have support which intersects the support of the vertices in $S_A$ only in $e_1$, hence we must have $C(x)-C(y)=\alpha e_1^2$ for some $\alpha$. 

Applying the mass $\eta$ we see that $\eta(C(y))=0,\ \eta(C(x))=-1$ hence $\alpha=-1.$\smallskip

So $C(x)-C(y)=- e_1^2$.  We now apply the rule \eqref{vee} ($u=-2e_1,\ g=y$) of the operator $C$ to $x $ red, $y$ black,   $x =g\cdot u=y-2e_1$
\begin{equation}\label{veeF}
C(x)=C(y-2e_1)=-   C(y)+ C(-2e_1) +2e_1y,\quad C(-2e_1)=-e_1^2
\end{equation}
$$\implies   -2 C(y)  +2e_1y=-y^2-y^{(2)}   +2e_1y=0.$$

$$y=\sum_i\alpha_ie_i\implies  -y^2-y^{(2)}=-\sum_i\alpha_i(\alpha_i+1)e_i^2-2\sum_{i<j}\alpha_i\alpha_je_ie_j=-2e_1y$$
$$\implies \alpha_i\alpha_j=0,\ 1<i<j,\  \alpha_1^2+\alpha_1=2,\  \implies \alpha_1=1,-2. $$
Since $\eta(y)=0$  we must have another index $i$ with $\alpha_i\neq 0$ and then all other $\alpha_j,\ j\neq 1,i$ are 0. So we have either $y=e_1-e_i$ or $y=-2e_1+2e_i$.
$$\alpha_1=1,\  y^2+y^{(2)}=e_1^2+e_i^2-2e_1e_i+e_1^2-e_i^2=2e_1(e_1-e_i)=2e_1y,\ $$
$$\alpha_1=-2,\  y^2+y^{(2)}=4e_1^2+4e_i^2-8e_1e_i-2e_1^2+2e_i^2\neq 2e_1y,\ $$
so $y=e_1-e_i,\ x= -e_1-e_i,\ x+y=-2e_i$. \end{proof}
\begin{remark}\label{mintree}
In the previous discussion $x,y$ are connected to the root by an edge so we can replace these two in the graph and now in the new tree we have a segment with the two consecutive  edges   $y=(e_1-e_i),\ x= -e_1-e_i$, So  the previous tree was not minimal. Arguing in the same way for $B,C$ we see in this case that a minimal graph in this case  has a simple structure  of {\em encoding graph} and {\em degenerate tree}:
\begin{equation}\label{minG} 
\xymatrix{ &&&encoding\ graph\\&4\ar@/^/[r]^{-e_4-e_3}&\ar@/^/[l]^{e_4-e_3}3\ar@{-}[r]^{\pm e_2-e_3}&1\ar@/^/[r]^{-e_1-e_2}&\ar@/^/[l]^{e_1-e_2}2& ,} \qquad  \xymatrix{     \ar@{-}[r]^{e_1-e_2}  &\ar@{-}[r]^{-e_1-e_2}  \ar@{-}[d]^{\pm e_2-e_3}  &\\  \ar@{-}[r]^{e_3-e_4}  &\ar@{-}[r]^{-e_3-e_4}&\\&tree   }
\end{equation} 
\end{remark}
We have thus verified  that the graph is not--allowable by Definition \ref{ilpunto1} for the two extremes of the segment $S_A$, a similar analysis would apply to $S_C$.

\subsection{The extra edge}

We treat now case 1) with an extra edge $E=\vartheta e_1-e_h,\ \vartheta=\pm1$.  We have the function $\zeta$  such that $\zeta(e_1)=1,\ \zeta(\ell_i)=0,\ \forall i$ and $\zeta(E)=2\vartheta $.  In this case the  even circuit is divided into two  odd paths. We divide the indices different from the two critical indices $1,h$  in  two blocks $A=(2,\ldots,h-1),\ B=(h+1,\ldots,k-1)$ and argue as in the previous section.

From Corollary \ref{ePo} it follows that, either the extra edge  is outside  the segment spanned by the $\ell_i$, this may happen if we are in a situation as (up to symmetry between $A,B$)

$$ a) \xymatrix{   &&\\ \ar@{-}[r]^{E}  &   \ar@{-}[rr]^{S_B} &&  \ar@{-}[rr]^{S_A}  &&   } \quad b) \xymatrix{  && \ar@{-}[d]^{E}  & \\  \ar@{-}[rr]^{S_B} &&  \ar@{-}[rr]^{S_A}  &&   } $$

In these cases the edge $E$ can be removed and the graph is not minimal. Otherwise it could separate the two segments spanned by the two blocks $A,B$ or it could appear in one or both of these segments according to the following pictures:

$$ c) \xymatrix{  &&\ar@{-}[dd]^{S_A}  && &  &\\ &&&&&\\ \ar@{-}[rr]^{S_B} &&\ar@{-}[r]^{E}&\ar@{-}[dd]^{S_A}\ar@{-}[rr]^{S_B} &&\\  &&&&   \\  &&&& \\&&&&  } $$
$$ d) \xymatrix{  &&\ar@{-}[dd]^{S_A}  && &  &\\ &&&&&\\ \ar@{-}[rr]^{S_B} &&\ar@{-}[r]^{E}  & \ar@{-}[rr]^{S_B} &&\\  &&&&   \\  &&&& \\&&&&  } \quad  e) \xymatrix{  &&\ar@{-}[dd]^{S_B}  && &  &\\ &&&&&\\ \ar@{-}[rr]^{S_A} &&\ar@{-}[r]^{E}  & \ar@{-}[rr]^{S_A} &&\\  &&&&   \\  &&&& \\&&&&  } $$

Cases d), e)  are   special cases of  c), and in fact follow from previous results, so we treat   case c).

  \subsubsection{ $E=\theta_E e_1-e_h$ } Let $\theta_E=\pm 1$ be its color. We look at the picture c).
  $$ c) \xymatrix{  &&a=h_1\cdot c\ar@{-}[dd]^{S_A^0}  && &  &\\ &&&&&\\ y\ar@{-}[rr]^{S_B^0} &&c\ar@{-}[r]^{\theta_E e_1-e_h}  &d\ar@{-}[dd]^{S_A^1}\ar@{-}[rr]^{S_B^1} &&x\\  && &&   \\  &&&b=h_2\cdot c& \\&&&&  } $$

The encoding graph is given in figure \eqref{dude1}.  As example
\begin{equation}\label{minG0}
\xymatrix{    \ar@{-}[dr]_{ e_1-e_2} \ar@{-}[r]^{ e_2-e_3}  &\ar@{=}[d]_{-e_3-e_1}  &    &\\  &\ar@{-}[r]_{e_4-e_1}  &\ar@{-}[ul]_{ e_3-e_4}& \\ &encoding\ graph  &&&} \xymatrix{  &a&&&&&&\\ tree\quad  y \ar@{-}[r]_{\quad e_1-e_2}  &\ar@{-}[u]^{ e_1-e_4} c \ar@{=}[r]^{- e_1-e_3}  &d\ar@{-}[d]^{e_3-e_4}\ar@{-}[r]_{e_3-e_2}&x\\     & &  b }
\end{equation}  \begin{lemma}\label{ABs}
We can fix the signs $\delta_i=\pm 1$ for which $ \sum_{i\in A\cup B } \delta_i\ell_i=0$   so that 
\begin{equation}\label{pmei}
-e_1-\theta_E e_h=\sum_{i\in A} \delta_i\ell_i ,\quad   \theta_E e_h+e_1= \sum_{i\in B } \delta_i\ell_i .
\end{equation}
\begin{proof} If $\theta_E=1,\ E=e_1-e_h$ the two paths from $1,h$  and $h$ back to 1 are both red  so
$$\sum_{i\in A} \delta_i\ell_i=-e_1-e_h,\     \sum_{i\in B } \delta_i\ell_i =e_1+e_h.$$
If $E=-e_1-e_h$  we have  the two paths from $1,h$  and $h$ back to 1 are both black and
$$\sum_{i\in A} \delta_i\ell_i= e_h-e_1,\     \sum_{i\in B } \delta_i\ell_i =e_1-e_h.$$
 \end{proof} 

\end{lemma}    If $E$ is black the  two vertices $y,x$ one is black   the other is red, by Lemma \ref{comt} the two circuits are both odd.   If $E$ is red the  two vertices $y,x$  have the same color.
   The same for $a,b$.
We need to argue as in Lemma \ref{chiave}

 \begin{lemma}\label{chiave1} i)\quad   Taking $c$ as root the indices in $A$  have the property that:\smallskip

  $ \delta_j\bar\sigma_j\bar\lambda_j=\delta$ is constant if $E$ is black. Same    for the indices in $B$.\smallskip

 If $E$ is red $ \delta_j\bar\sigma_j\bar\lambda_j=\delta$ is constant on the two segments $S_A^0,\ S_A^1$ and changes sign passing from one to the other.
 \smallskip

 \end{lemma}
\begin{proof}  i)    We want to prove that   the value  $ \delta_j\bar\sigma_j\bar\lambda_j$ is constant or changes sign. For this by induction it is enough to see what  the value does    for $\ell_u,\ell_{u-1}$. 

When they are not separated  by the edge $E$    we can use Lemma  \ref{icons}. 

Assume $u\in S_A^0,\ u-1\in S_A^1$ then 
  we first compare the values that we call $\bar\sigma_j$ when we place the root at $c$  with the values $\sigma_j$  when we place the root at the beginning of $\ell_u$. \begin{equation}
\label{esubx1} \quad  \xymatrix{   r    \ar@{-}[r]^{\ell_u} &s\ar@{--}[r] &c\ar@{-}[r]^E&d \ar@{--}[r] &  y \ar@{-}[r]^{\ell_{u-1}}  & x_{u-1}   } .
\end{equation} 
We claim that $ \bar\sigma_u\bar\sigma_{u-1}=\sigma_E\sigma_{u-1} $.  

Let $g_1,g_2\in G_2$ be such that $r=g_1c,\ x_{u-1}=g_2d$ so  $x_{u-1}=g_2\circ E^{-1} \circ g_1^{-1}r.$   

$ \bar\sigma_u,\ \bar\sigma_{u-1} $  are respectively the color  of $g_1,g_2\circ E$ and so $\sigma_{u-1} $ the color of $g_2\circ E^{-1}\circ g_1^{-1}$  is their product.\medskip

In order to prove  that $\delta_j \bar \sigma_j\bar\lambda_j$ changes by $\sigma_E$  we need to show that when $\ell_u,\ell_{u-1}$ are separated  the product of the two terms is  $\sigma_E$.  That is we need
$$\sigma_E=   \delta_{u-1}\bar\sigma_{u-1}\bar\lambda_{u-1}    \delta_u \bar\sigma_u \bar\lambda_u=  \delta_{u-1} \sigma_E\sigma_{u-1}\bar\lambda_{u-1}    \delta_u   \bar\lambda_u . $$ If $\ell_u,\ell_{u-1}$ are separated this means that $u$ is an index of type II, cf. Proposition  \ref{edE}.

  We have $\bar\lambda_{u-1}= \lambda_{u-1}$  while $\bar\lambda_u=- \vartheta_u\lambda_u$. In other words  we need
  $$  -\delta_{u-1}\vartheta_u \sigma_{u-1} \lambda_{u-1}    \delta_u    \lambda_u=1. $$
Since by definition $\delta_{u-1}\vartheta_u =    \delta_u $
we have to verify that
   $$  -\delta_{u-1}\vartheta_u \sigma_{u-1} \lambda_{u-1}    \delta_u    \lambda_u=-  \sigma_{u-1} \lambda_{u-1}      \lambda_u=1.  $$  This is in our case the content of the second part of Corollary \ref{leduei}.
\smallskip

We thus have taking $c$ as root by Theorem \ref{ilve} ($v:=v_\ell=\sigma_\ell  \sum_{\ell\preceq v}\sigma_\ell\lambda_\ell\ell $).      
\begin{align*}\label{iduela}
a=\bar \sigma_a\sum_{j\in S_A^0} \bar \sigma_j\bar\lambda_j\ell_j=\bar \sigma_a \delta\sum_{j\in S_A^0} \delta_j\ell_j,\  & h_1=(\bar \sigma_a \delta\sum_{j\in S_A^0} \delta_j\ell_j,\bar\sigma_a)\\
b=\bar \sigma_b(\theta_E   E +\sum_{j\in S_A^1} \bar \sigma_j\bar\lambda_j\ell_j)=\bar \sigma_b\theta_E ( E+ \delta\sum_{j\in S_A^1} \delta_j\ell_j ),\ &  h_2 =(b, \bar \sigma_b\theta_E  )\end{align*}
$$  b=-\bar \sigma_a\theta_E ( E+ \delta\sum_{j\in S_A^1} \delta_j\ell_j )\implies $$
\begin{equation}\label{Bara}
\bar a  -  \bar   b:=\bar \sigma_a  a-  \bar \sigma_a\theta_E b =E+\sum_{j\in A} \delta_j\ell_j=E-e_1-\theta_E e_h= -2e_h
\end{equation} A similar argument holds for $ y,x$ and   from \eqref{pmei}
 $$ \bar   y-  \bar   x = \bar \sigma_y  y-  \bar \sigma_x\theta_E x =E+\sum_{j\in B} \delta_j\ell_j=E+\theta_Ee_h+e_1= (\theta_E+1)e_1+(\theta_E-1)e_h
 $$  

$$\theta_A=-1\implies  \bar a-\bar b=\bar y-\bar x,\quad  \theta_A=1\implies  \bar a-\bar b-\bar y+\bar x=-2E$$ the resonance is thus
$$ C(\bar a)-C(\bar b)-C(\bar y)+C(\bar x)=\begin{cases}
4C(E)=4(e_1^2-e_1e_h) ,\  \theta_E= 1 \\0,\  \theta_E=-1
\end{cases}$$ 
This implies  that both $C(\bar a)-C(\bar b)$ and $C(\bar x)-C(\bar y)$ are quadratic expressions in $e_1,e_h$.

We may assume $\bar a, \bar y$ red  and $\bar b, \bar x$  black so 
$$ 2C(\bar a)-2C(\bar b)=-\bar a^2-\bar a^{(2)}-\bar b^2-\bar b^{(2)}$$
  write  $\bar a= u+v,\  \bar b= s+t$  where $s,u$ have support in $1,h$ and $v,t$ outside.

  $$-u^2-v^2-2uv-u^{(2)}-v^{(2)}-s^2-t^2-2st-s^{(2)}-t^{(2)}$$
   
implies    $$\implies   v^2+2uv +v^{(2)} +t^2+2st +t^{(2)}=0\implies  uv=-st,\  v^2  +v^{(2)} = t^2  +t^{(2)}=0 .$$ Then $v^2  +v^{(2)} =0$  implies  $v=-e_i$ for some $i$ or $v=0$.    Implies $u=s, v=-t$ or $u=-s, v=t$.
     From the Formula for $a$ we have that the coefficients  in $u$ for $e_1,e_h$ are $\pm 1$  so $a$ is the sum of  $e_1,e_h,v$ with coefficients $\pm 1,0$ furthermore $\eta(a)=-2$ implies that $a=-e_1-e_j$  or $a=-e_h-e_j$ where $j=i$ or  if $v  =0$ we have $a=-e_1-e_h$. 
   
   Then  from \eqref{Bara} since $\bar \sigma_a=-1$ we have 
 $    -a+\theta_E b =  -2e_h,\    \theta_Ab=-2e_h+e_h+e_i=-e_h+e_i$.
   This means that taking the root at $a$  we have $b= -e_h-e_j +\theta_E(-e_h+e_j)=-2e_h,\ -2e_j$ the graph is not allowable.\end{proof}
   {\bf In conclusion}   We have treated all possible cases and verified in each case that a minimal degenerate graph, is not allowable, proving Theorem \ref{MM}.  In fact we have even shown what are the possible minimal degenerate graphs which are presented in the two figures \eqref{minG} and \eqref{minG0}.\section{Appendix}
   In this paper we have treated the case of the rectangle graph, which appears in the NLS  for $q=1$.  The first part of the paper in fact holds also for any $q$,  arriving to Theorem \ref{bou}. \smallskip
   
   Still Theorem \ref{main} holds for all graphs with only black vertices which in the arithmetic case excludes only finitely many  blocks   in the normal form of the NLS.
   
   In this more general case the difference is in the choice of the  edges $X_q=X_0^q\cup X_{-2}^q$  which now are a larger set, the constraints of rectangles are replaced by    \begin{equation}\label{basco}
\sum_{i=1}^{4q}(-1)^i k_i=0,\quad  \sum_{i=1}^{4q}(-1)^i |k_i|^2=0  .
\end{equation} The first constraint on the choice of the vectors $S$ is replaced by\begin{constraint}\label{co1b}\strut\begin{enumerate}
\item We assume   that  $\sum_{j=1}^m n_j \mathtt v_j \neq 0$ for all  $n_i\in\Z,\,\sum_in_i=0,\ 1<\sum_i|n_i|\leq 2q+2$. 

\item $|\sum_in_i\mathtt v_i|^2-\sum_in_i|\mathtt v_i|^2\neq 0$ when $n_i\in\Z,\,\sum_in_i=1,\ 1<\sum_i|n_i|\leq 2q+1$.

\item We assume that  $\sum_{j=1}^{m}\ell_j \mathtt v_j\neq 0 $, when $u:=\sum_{j=1}^{m}\ell_j e_j$ is either an edge or a sum or difference  of two distinct edges.

\item  $2\sum_{j=1}^{m}\ell_j |\mathtt v_j|^2+|\sum_{j=1}^{m}\ell_j \mathtt v_j|^2\neq 0$ for all edges  $\ell=\sum_{j=1}^{m}\ell_j e_j$ in $X^q_{-2}$.
\end{enumerate}
\end{constraint} We need to strenghten Constraint \ref{c3} to 
\begin{constraint}\label{c3b}
$\sum_{i=1}^m\nu_i\mathtt v_i\neq 0,\ \forall \nu_i\in\Z,\ \mid \sum_{i=1}^m|\nu_i|\leq 4q(n+1)$.
\end{constraint}
We have to give a different proof of Proposition \ref{main1}.  In that proposition 
since  we are assuming that there is a non trivial odd circuit starting from $x$, changing if necessary the starting point $x$, in the first step of the circuit    we may assume that 
 $x$ lies in   a sphere $S_{\ell}$ for some initial edge $\ell\in X^q_{-2}$  with $\eta(\ell)=-2$.

This implies that $x=-1/2\sum_in_iv_i$ satisfies a  relation of type 
\begin{equation}\label{pririn1}  |\sum_in_iv_i|^2-2( \sum_in_iv_i,\pi(\ell))=4 K(\ell).
\end{equation}  Where  $\ell=(\sum_i\ell_ie_i) $.   This formula  vanishes identically     if   $ a ^2-2  a \ell  =4 C(\ell )=-2(\ell ^2+\ell^{(2)})$.  Thus  $$(a-\ell)^2=-\ell^2-2\ell^{(2)}. $$

This implies that all  coefficients $a_i$ of $\ell$ must have $-a_i^2-2a_i\geq 0$ so since $a_i\in\Z$  if $a_i\neq 0$ must be $a_i=-1$ or $a_i=-2$, and, since $\eta(\ell)=-2$ then  $\ell=-e_i-e_j, -2e_i$.

This implies, if $\ell=-2e_i$  that $a=\ell$. In the first case if $\ell=-e_i-e_j$  we have $-\ell^2-2\ell^{(2)}= (e_i-e_j)^2$  so   $a-\ell=\pm (e_i-e_j)$ hence again $a=-2e_i, -2e_j$ and $x=v_i,v_j$.
\smallskip

Finally  we have to give a different proof of Lemma \ref{spegra}.
   \begin{lemma}\label{spegra1} In dimension $n$,
If   a   graph of rank $\geq n+1$ has a  generic solution to the associated system,   which is given by a polynomial, then the graph is special and the polynomial is of the form $v_i$ for some $i$.
\end{lemma}
\begin{proof}  The root $x$ is a solution of  the equations \eqref{bacos}
 $$(x,\pi(a_i))=K(a_i),\quad |x|^2+(x,\pi(b_j))=K(b_j).$$
 If the  solution $x$ is polynomial in the $v_i$,   it is linear by a simple degree computation. 
 
 Let $g\in O(n)$  be an element of the orthogonal group of $\R^n$,  substitute in the   equations $v_i\mapsto  g\cdot v_i$.     By their definition the functions $K$ are   invariant under $g$  and a transformed equations have a solution $x(g)$ with  $(x(g),g\pi(a_i))= K(a_i)$.
 
 We have   $(x(g), \pi(a_i))= (g^{-1}x(g),\pi(a_i))$  so $g^{-1}x(g)=x$    is also  equivariant under the orthogonal group of $\R^n$. It follows by simple invariant theory that it has the form $x=\sum_sc_s\mathtt v_s$ for some numbers $c_s$.

  By Lemma \ref{CK} and the fact that the given system of equations is satisfied for all $n$ dimensional vectors $\mathtt v_i$ it is valid for the vectors $\mathtt v_i$  with only the first coordinate  $x_i$ different from 0, or if we want for 1--dimensional vectors so that now the symbols $\mathtt v_i=v_i$ represent simple variables (and not vector variables). So we have, for a  black vertex $a_i=\sum_jm_je_j$
  $$\pi(a_i)=\sum_jm_jv_j,  \ K(a_j)=\frac12[ (\sum_jm_jv_j)^2+\sum_jm_jv_j^2]$$The equations \eqref{bacos} become 
  $$2(\sum_sc_sv_s) (\sum_jm_jv_j) = (\sum_jm_jv_j)^2+\sum_jm_jv_j^2$$ which implies that $ (\sum_jm_jv_j)$ divides $\sum_jm_jv_j^2$.
  
  Now $\sum_jm_jv_j^2$  if it is in $\geq 3$ variables it is an irreducible polynomial. In 2 variables  since  we have $\sum_jm_j=0$, the polynomial is    $ m(v_h^2-v_k^2)= m(v_h-v_k)(v_h+v_k) $ and
  $$2(\sum_sc_sv_s)=m(v_h-v_k)+ v_h+v_k=(1+m  )v_h+(1-m)v_k.$$ if there is another black vertex  $a_i\neq a_j$  we have a different linear equation  of the same type and  get $$2(\sum_sc_sv_s)= (1+p  )v_a+(1-p)v_b\implies  (1+m  )v_h+(1-m)v_k = (1+p  )v_a+(1-p)v_b$$  since the linear equation  is different  this can happen only if $m=\pm 1$ and $(\sum_sc_sv_s)=v_h, v_k$.  \smallskip
  
  If all other vertices are red we have an equation for $a_i=\sum_hn_he_h$ with $\eta(a_i)=-2$
  
$$x^2+x(\sum_an_av_a)=K(\sum_an_ae_a),\  2x =  (1+m  )v_h+(1-m)v_k.$$ So $ (1+m  )v_h+(1-m)v_k$ divides the quadratic polynomial $2K(\sum_an_ae_a)$. 

This implies first as before that $\sum_an_ae_a=ne_h-(2+n) e_k,\ n\geq 0$  so $$-2K(\sum_an_ae_a)=(nv_h-(2+n)  v_k)^2+nv_h^2-(2+n) v_k^2= (n^2+n)v_h^2+(n+2)  (n+1)v_k^2-2n(n+2)v_hv_k.$$
 For this a necessary condition to be factorizable over $\Z$ is that the discriminant $-n(n+2) \geq 0$  which implies $n=0,-2$.  In either case  $ 2x=(1+m  )v_h+(1-m)v_k$ divides $v_h$ or $v_k$  which implies $x=v_h,\ v_k$.\end{proof}
    \bibliographystyle{amsalpha}

\end{document}